\documentclass[11pt,reqno]{amsart}
\usepackage[english]{babel}
\usepackage[misc]{ifsym}
\usepackage{mathrsfs,mathtools}
\usepackage{esint} 
\usepackage{faktor}
\usepackage[utf8]{inputenc}
\usepackage[left=2.9cm,right=2.9cm,top=2.8cm,bottom=2.8cm]{geometry}
\usepackage{graphicx}
\usepackage{enumitem}
\usepackage{todonotes}
\usepackage[colorlinks=true,urlcolor=blue,citecolor=blue,linkcolor=blue,linktocpage,pdfpagelabels,bookmarksnumbered,bookmarksopen]{hyperref}

\newcommand{\N}{\mathbb{N}}

\newcommand{\R}{\mathbb{R}}
\newcommand{\A}{\mathcal{A}}
\newcommand{\C}{\mathcal{C}}
\newcommand{\D}{\mathcal{D}}
\newcommand{\T}{\mathcal{T}}
\newcommand{\B}{\mathcal{B}}
\renewcommand{\S}{\mathscr{S}}
\newcommand{\M}{\mathscr{M}}

\newcommand{\dx}{\, {\rm d} x}
\newcommand{\dy}{\, {\rm d} y}
\newcommand{\dt}{\, {\rm d} t}
\newcommand{\ds}{\, {\rm d} s}

\newcommand{\dz}{\, {\rm d} z}

\newcommand{\eps}{\varepsilon}
\newcommand{\loc}{{\rm loc}}

\makeatletter
\newcommand{\compact}{\mathrel{\mathpalette\comp@emb\relax}}
\newcommand{\comp@emb}[2]{
  \vcenter{
    \offinterlineskip\m@th
    \ialign{$#1##$\cr\hookrightarrow\cr\noalign{\vskip1pt}\hookrightarrow\cr}
  }
}
\makeatother

\newtheorem{lemma}{Lemma}[section]
\newtheorem{thm}[lemma]{Theorem}

\theoremstyle{definition}
\newtheorem{defi}[lemma]{Definition}
\newtheorem{rmk}[lemma]{Remark}

\numberwithin{equation}{section}

\usepackage[normalem]{ulem}

\newcommand\soutD{\bgroup\markoverwith

{\textcolor{blue}{\rule[.5ex]{2pt}{1pt}}}\ULon}

\allowdisplaybreaks
\begin{document}

\title[A Grushin problem in $\R^N$ with singular, convective, critical reaction]{Existence and decay for a Grushin problem  in $\R^N$ \\ with singular, convective, critical reaction}

\author[L. Baldelli]{Laura Baldelli}
\address[L. Baldelli]{IMAG, Departamento de Análisis Matemático, Universidad de Granada, Campus Fuentenueva, 18071 Granada, Spain}
\email{labaldelli@ugr.es}

\author[P. Malanchini]{Paolo Malanchini}
\address[P. Malanchini]{Dipartimento di Matematica e Applicazioni, Universit\`a degli Studi di Milano - Bicocca, via Roberto Cozzi 55, 20125 Milano, Italy}
\email{p.malanchini@campus.unimib.it}

\author[S. Secchi]{Simone Secchi}
\address[S. Secchi]{Dipartimento di Matematica e Applicazioni, Universit\`a degli Studi di Milano - Bi\-coc\-ca, via Roberto Cozzi 55, 20125 Milano, Italy}
\email{simone.secchi@unimib.it}

{
\let\thefootnote\relax
\footnote{\Letter \quad Corresponding author: .}
}
\setcounter{footnote}{0}

\begin{abstract}
We establish an existence result for a problem set in the whole Euclidean space involving the Grushin operator and featuring a critical term perturbed by a singular, convective reaction. Our approach combines variational methods, truncation techniques, and concentration-compactness arguments, together with set-valued analysis and fixed point theory. Additionally, we prove the decay at infinity of solutions in the absence of the convective term. The result is new even in the case where more than one feature between singularity, convectivity and criticality is taken into account.
\end{abstract}

\maketitle


\section{Introduction}
In the present paper we consider the problem
\begin{equation}
\label{prob}
\tag{${\rm P}$}
\left\{
\begin{alignedat}{2}
-\Delta_\gamma u &=\lambda_1 w_1(z)u^{-\eta}+\lambda_2 w_2(z)|\nabla_\gamma u|^{r-1} +  u^{2^*_\gamma-1} \quad &&\mbox{in $\R^N$}, \\
u &> 0 \quad &&\mbox{in $\R^N$}, 
\end{alignedat}
\right.
\end{equation}
where $N\ge 3$, $\Delta_\gamma$ is the Ba\-ouen\-di-Grushin operator, $\lambda_1>0$, $\lambda_2\ge0$, $0<\eta<1<r<2$, $2^*_\gamma$ is the Sobolev critical exponent defined in \eqref{eq:2stargamma}, $w_1$, $w_2 \colon \R^N\to [0,+\infty)$ are suitable weight functions whose properties will be described below. As such, our problem is singular at $u=0$ and has a critical growth in a sense which will be explained later.

\medskip
For the reader's convenience, we briefly recall the basic terminology of the Ba\-ouen\-di-Grushin operator. We split the euclidean space $\R^N$ as $\R^m \times \R^\ell$, where $m \geq 1$, $\ell \geq 1$ and $m+\ell=N$. 
A generic point $z \in \R^N$ can therefore be written as
\begin{displaymath}
	z = \left(x,y\right) = \left( x_1,\ldots,x_m,y_1,\ldots,y_\ell \right),
\end{displaymath}
where $x \in \R^m$ and $y \in \R^\ell$.
Let $\gamma$ be a nonnegative real parameter. For a differentiable function $u$ we define the Grushin gradient
\begin{displaymath}
	\nabla_\gamma u (z) =  (\nabla_x u(z), |x|^\gamma\nabla_y u(z))= (\partial_{x_1} u(z) , \dots, \partial_{x_m}u(z), |x|^\gamma \partial_{y_{1}}u(z), \dots, |x|^\gamma \partial_{y_{\ell}}u (z)).
\end{displaymath}
The Grushin operator $\Delta_\gamma$ is defined  by
\begin{displaymath}
	\Delta_	\gamma u(z) = \Delta_x u(z) + |x|^{2\gamma} \Delta_y u (z),
\end{displaymath} 
where $\Delta_x$ and $\Delta_y$ are the Laplace operators in the variables $x$ and $y$, respectively. A crucial property of the Grushin operator is that it is not uniformly elliptic in the space $\R^N$, since it is degenerate on the subspace $\Sigma=\{0\} \times \R^\ell$. We denote by $N_\gamma \coloneqq m + (1+\gamma)\ell$
the \emph{homogeneous dimension} associated to the decomposition $N=m+\ell$.

The operator $\Delta_\gamma$ was first introduced by Baouendi in his 1967 Ph.D. thesis, \cite{baouendi}. A few years later, Grushin investigated the hypoellipticity properties of $\Delta_\gamma$ in the case $\gamma\in\N$, see \cite{grushin}. 

In the early 80's, Franchi and Lanconelli \cite{franchilanconelli1982,franchilanconelli1983,franchilanconelli1984} introduced a class of operators of the form
\begin{displaymath}
	\Delta_\lambda = \sum_{i=1}^N \partial_{x_i} \left( \lambda_i^2 \partial_{x_i} \right),
\end{displaymath}
which include $\Delta_\gamma$. Here $\lambda_1,\ldots,\lambda_N$ are given functions which satisfy suitable conditions.

Moreover, the Grushin operator falls into the class of $X$-elliptic operators introduced in \cite{lanconelli2000x}. In fact, the Grushin operator is uniformly $X$-elliptic with respect to the family of vector fields $X=(X_1,\dots, X_N)$ defined as
\begin{displaymath}
    X_i = \frac{\partial}{\partial x_i} \,\,\hbox{for $i=1,\ldots,m$} \qquad
    X_{m+j} = \vert x \vert^\gamma \frac{\partial}{\partial y_j} \,\,\hbox{for $j=1,\ldots,\ell$}.
\end{displaymath}
With the same notation we can write the Grushin operator $\Delta_\gamma$ as sum of the vector fields
\begin{displaymath}
    \Delta_\gamma = \sum_{j=1}^{N} X_j^2.
\end{displaymath}
Note that, when $\gamma$ is an integer, the $C^\infty$-field $X=\left( X_1,\ldots,X_N \right)$ satisfies the H\"{o}rmander's finite rank condition, see \cite{hormander1967}.
\begin{rmk}
   The latter statement is obvious when $\gamma$ is an even integer. In fact, the operator $\Delta_\gamma$ can be expressed as a sum of smooth vector fields that satisfy H\"{o}rmander's finite rank condition even when $\gamma$ is odd, see \cite{abatangelo2024}. If $\gamma>0$ is a generic real number, the vector field $X$ fails to be smooth, and this is a well-known difficulty related to the Grushin operator, see Remark \ref{rmknat} for more details.
\end{rmk}

\medskip
We assume the following hypothesis on $w_1$ and $w_2$:
\begin{enumerate}[label={$({\rm H_1})$}]
\item \label{hypf} The functions $w_1$, $w_2$ are nonnegative, belong to  $L^1(\R^N) \cap L^\infty(\R^N)$ and  there exist a ball $B(z_0,\varrho) \subset \R^N$ and $\omega>0$ such that
\begin{displaymath}
    \int_{B(z_0,\varrho)}w_1 \dz\ge\omega>0.
\end{displaymath}
\end{enumerate}
\begin{rmk}
    The last condition on $w_1$ is satisfied as soon as $w_1 \not\equiv 0$, see \cite[Theorem 7.7]{Rudin}.
\end{rmk}
We assume the following decay condition on $w_1$: 
\begin{enumerate}[label={$({\rm H_2})$}]
\item \label{condw1} 
There exist constants $c_1>0$, $R>0$ and $\delta> N_\gamma+\eta\,(N_\gamma-2)$ such that, for every $z\in \R^N$ with $d(z)>R$, 
\begin{displaymath}
w_1(z)\le c_1 d(z)^{-\delta-2\gamma},
\end{displaymath}
where
\begin{equation}
d(z) = \left( |x|^{2\left(\gamma +1\right)} + (\gamma+1)^2 |y|^2\right)^{\frac{1}{2\left(\gamma+1\right)}} \label{eq:d}
\end{equation}
is the homogeneous distance associated to $\Delta_\gamma$.
\end{enumerate}
We will look for solutions to \eqref{prob} in the Beppo Levi space having null trace in the Grushin setting, i.e. $D^\gamma_0(\R^N)$, see Section \ref{functsett} for details.

By definition, a function $u\in D_0^\gamma (\R^N)$ is a weak solution to \eqref{prob} if $u>0$ in $\R^N$ and satisfies the following variational identity:
\begin{displaymath}
	\forall \varphi\in D_0^\gamma(\R^N) \colon \quad \int_{\R^N} \nabla_\gamma u \nabla_\gamma \varphi\dz = \int_{\R^N} \left(\lambda_1 w_1(z)u^{-\eta}+\lambda_2 w_2(z)|\nabla_\gamma u|^{r-1} +  u^{2^*_\gamma-1}\right)\varphi\dz.
\end{displaymath}
The main result of this paper is the following.
\begin{thm}
\label{mainthm}
Suppose that \ref{hypf}-\ref{condw1} hold. Then there exists $\Lambda>0$ such that, for any $\lambda_1\in (0,\Lambda)$ and any $\lambda_2 \in [0,\lambda_1]$,
problem \eqref{prob} admits a weak solution $u>0$ in $D_0^\gamma(\R^N) \cap L^\infty(\R^N)\cap C^{0,\tau}(\R^N)$ for some $\tau\in (0,1]$. Moreover, if $\lambda_2 = 0$ then
\begin{displaymath}
\lim_{d(z) \to +\infty} u(z) =0.
\end{displaymath}
\end{thm}

The analysis of critical problems was initiated by Brezis and Nirenberg in \cite{bn} for the Laplacian in a bounded domain. Their work marked the beginning of an endless stream of efforts to extend the seminal results.

In recent years, critical problems involving the Grushin operator have been extensively studied; see, for instance, \cite{BB, L, MMS, ZY}, among others. An analogue of the Brezis–Nirenberg result for the Grushin operator was recently established in \cite{alves2024brezis}.
Critical problems with the Grushin operator have also bees treated in unbounded domains, see for example \cite{AlvesHolanda, ZYb}.

Existence of positive weak solutions for a Grushin problem with singular nonlinearity has been proved in \cite{balbiswas} in a bounded domain. 

In \cite{BRW}, Bahrouni, R\u{a}dulescu, and Winkert introduced a Baouendi–Grushin operator with variable coefficients $\Delta_{G(x,y)}$, which reduces to the `usual' Grushin operator when $G(x,y) \equiv 2$. In the same paper they investigated a boundary value problem involving a convective term in a bounded domain. Problems in the whole space $\mathbb{R}^N$ with singular nonlinearities for the operator $\Delta_{G(x,y)}$ have been studied in \cite{BR} and \cite{BRR}. As far as we know, no results are available for entire problems involving convective terms. 

The first part of our work is inspired by \cite{BG}, where an analogous to problem \eqref{prob} has been treated for the $p$-Laplace operator, and we adapt the results therein to the setting of the Grushin operator. We point out that many ideas of \cite{BG} rely on the uniform ellipticity of the standard Laplace operator, and it is not trivial to adapt them to the degenerate elliptic operator $\Delta_\gamma$. Note also that problem \eqref{prob} mixes variational problems with double lack of compactness, in the sense of Sobolev's embeddings, with non-variational problems since convection terms destroy the variational structure.

Theorem \ref{mainthm} is a first attempt to study problems involving the Grushin operator with a critical, a convective and a singular term in the entire space $\R^N$. To the best of our knowledge, our result is new even in the case where more than one feature among singularity, convectivity and criticality is taken into account.

The last part of the paper is motivated by the regularity and decay properties obtained in \cite{BG2} for a critical problem involving the $p$-Laplace operator, along with gradient estimates. 

\subsection{Outline of the proof}

We start the proof of Theorem \ref{mainthm} by truncating the singular term $u^{-\eta}$ and \textit{freezing} the convective term $|\nabla_\gamma u|^{r-1}$, in order to cast the problem within a classical variational framework.

More precisely, we choose an appropriate function $\underline u_{\lambda_1}$ (to be defined in \eqref{usub}) and an arbitrary, fixed, function $v \in D_0^\gamma(\R^N)$. We then examine the auxiliary problem
\begin{equation}
\label{eq:varprob}
-\Delta_\gamma u = a(z,u) + u_+^{2^*_\gamma-1} \quad \mbox{in} \;\; \R^N,
\end{equation}
where
\begin{displaymath}
a(z,s)\coloneqq\lambda_1 w_1(z)\max\{s,\underline u_{\lambda_1}(z)\}^{-\eta} + \lambda_2 w_2(z) | \nabla_\gamma v(z) |^{r-1} \quad \forall (z,s)\in \R^N\times\R.
\end{displaymath}

The central part of the paper is dedicated to constructing a solution $u \in D_0^\gamma(\R^N)$ to \eqref{eq:varprob} by employing the mountain pass theorem on the associated energy functional $J$ (see Theorem \ref{existencefinal}). Using the concentration-compactness principle, we identify a \emph{critical} level $\hat{c}$ (see \eqref{hatc}) below which the functional $J$ recovers compactness, i.e., $J$ satisfies the Palais-Smale condition (Lemma \ref{PS}); moreover, we show that, for sufficiently small values of $\lambda_1$ and $\lambda_2$, $J$ exhibits the mountain pass geometry (Lemma \ref{mountainpassgeometry}); finally, we verify that the mountain pass level lies below the critical Palais-Smale level $\hat{c}$, again for small $\lambda_1$ and $\lambda_2$ (Lemma \ref{talentihatc}).

We then return to the original problem. Since the truncation was performed at the level of a sub-solution $\underline u_{\lambda_1}$ (see Lemma \ref{subsol} and \eqref{usub}), any solution to \eqref{eq:varprob} stays above $\underline u_{\lambda_1}$ (see Remark \ref{subcomparisonrmk}).

The next task is to \textit{unfreeze} the convection term: this is achieved by using set-valued analysis and fixed point theory. Specifically, we introduce a set-valued function $\S$, which associates to each function $v\in D_0^\gamma(\R^N)$ the set of solutions to \eqref{eq:varprob} having `low' energy (see \eqref{Sdef}). The function $\S$ is compact (Lemma \ref{Scompact}) and, for small values of $\lambda_1$ and $\lambda_2$, lower semi-continuous (Lemma \ref{lsc}). Moreover, its selection $\T$, defined as $\T(v) \coloneqq \min \S(v)$, still has the properties of compactness and continuity; so, Schauder's Fixed Point Theorem guarantees the existence of a fixed point of $\T$ (Theorem \ref{exsol}), that is, a solution to \eqref{prob}.

Then, we prove a pointwise decay of our solution at infinity when $\lambda_2 = 0$.
The lack of precise elliptic estimates for the Grushin gradients forces us to impose $\lambda_2 = 0.$ First, a global $L^\infty$ estimate is ensured via De Giorgi's technique (Theorem \ref{regularity}).

To get the decay of the solutions we then follow \cite{BG2}. We need a doubling  property in Lemma \ref{lem_dou}, a Lemma on weak Lebesgue spaces (Lemma \ref{weaklebboundlemma}) and an adaptation to our setting of a local boundedness theorem (Theorem \ref{locboundthm}), whose proof is postponed to Appendix \ref{appendix_a}.

It is also worth mentioning that we provide quantitative estimates on the threshold $\Lambda$ in Theorem \ref{mainthm}: see \eqref{smallness}, \eqref{smallness2}, \eqref{smallness3}, and \eqref{lsccond2}.

As an additional result, we also prove (Lemma \ref{decay_subsolution}) the decay of a subsolution to \eqref{prob}. This result requires some stronger assumptions on $\ell$ and $w_1$; see Remark \ref{rmk:decay} for further details.

\medskip
Unlike in the paper \cite{BG}, we can only prove that the solutions are locally H\"{o}lder-continuous since, as explained in Remark \ref{rmknat2}, a full regularity theory is not available in this context. We establish the continuity of the solutions (Theorem \ref{regularity}) via the non-homogeneous Harnack inequality for $X$-elliptic operators (see \cite[Theorem 5.5]{gutierrezlanconelli}).

\medskip
 The paper is organized as follows. In Section \ref{prel} we give some classical definitions and state some basic results, such as the concentration-compactness principles, the mountain pass theorem, and a fixed point theorem, together with minor lemmas which will be useful in the sequel. Section \ref{sec:existence} is devoted to prove an existence result for problem \eqref{prob}; in particular, in Section \ref{trfr} we study a truncated and frozen problem, while Section \ref{unfr} addresses the unfreezing of the convection term via set-valued analysis, finally giving a solution to \eqref{prob}. Section \ref{reg} discusses the boundedness and the decay of the solutions, concluding the proof of Theorem \ref{mainthm}.

\section{Preliminaries}\label{prel}

\subsection{The fundamental solution for $\Delta_\gamma$}

We have already introduced in \eqref{eq:d} the homogeneous distance associated to the Grushin operator $\Delta_\gamma$. 
Notice that this distance is homogeneous of degree $1$ with respect to the anisotropic dilatation
\begin{displaymath}
	\delta_\lambda(z,y) = (\lambda x, \lambda^{\gamma+1}y),\quad \lambda >0
\end{displaymath}
and, for $\gamma=0$, it reduces to the usual Euclidean distance in $\R^N$.
The function 
\begin{equation}\label{gammadef}
	\Gamma(z) = \frac{C}{d(z)^{N_\gamma -2}}
\end{equation}
is a fundamental solution for the Grushin operator with singularity at $z=0$, see \cite[Appendix B]{ambrosio2003}.
The constant $C$ is a suitable positive constant that depends on $m,\ell$ and $\gamma$. Precisely,
\[
C^{-1} = (N_\gamma -2) \int\limits_{\{z:\, d(z) = 1\}} \frac{| x | ^{2\gamma}}{\left(|x|^{2\left(1+2\gamma\right)}+ \left(1+\gamma\right)^2 |y|^2\right)^{1/2}} \, \rm d \sigma,
\]
where $\sigma$ denotes the surface measure of the manifold $\left\{ z:\, d(z)=1 \right\}$.

\subsection{Notation}

The complement of any set $A$ in a set $X$ is $A^c = X \setminus A$.
We indicate with $B_\gamma(z,R) =\{\xi\in\R^N :\,  d(z -\xi)<R \}$ the $d$-ball of center $z\in\R^N$ and radius $R>0$. 
We will omit $\gamma$ when it is equal to $0$ and the $d$-ball $B(z,R)$ reduces to the euclidean one.

Given any $A\subseteq \R^N$, we write $\chi_A$ to indicate the characteristic function of $A$. For any $N$-dimensional Lebesgue measurable set $\Omega$, the symbol $|\Omega|$ denotes the $N$-dimensional Lebesgue measure of $\Omega$.

We denote by $C^\infty_c(\R^N)$ the space of the compactly supported smooth functions on $\R^N$, while $C^{0,\tau}_\loc(\R^N)$, for any $\tau\in(0,1]$, denotes the space of locally $\tau$-H\"{o}lder continuous functions.
Given any measurable set $\Omega\subseteq \R^N$ and $q\in[1,+\infty]$, $L^q(\Omega)$ stands for the standard Lebesgue space, whose norm will be indicated with $\|\cdot\|_{L^q(\Omega)}$, or simply $\|\cdot\|_q$ when $\Omega=\R^N$. 

Given a real-valued function $\varphi$, we denote its positive and negative part by 
\begin{displaymath}
    \varphi_+\coloneqq \max\{\varphi,0\}, \quad \varphi_-\coloneqq \max\{-\varphi,0\}.
\end{displaymath}
We abbreviate $\{u>v\} = \{x\in\R^N: \, u(x)>v(z)\}$, and similarly for $\{u<v\}$, etc.
 
We indicate with $X^*$ the dual of a Banach space $X$, while $\langle\cdot,\cdot\rangle$ stands for the duality brackets. Given two Banach spaces $X,Y$, the continuous embedding of $X$ into $Y$ is indicated by $X\hookrightarrow Y$; if the embedding is compact, we write $X\compact Y$. If a sequence $(u_n)$ strongly converges to $u$ we write $u_n\to u$; if the convergence is in weak sense, we use $u_n\rightharpoonup u$. The letter $S$ denotes the best Sobolev constant; see the next subsection for details.

The space of all finite signed Radon measures is denoted by $\M(\R^N,\R)$. Concerning convergence of measures $(\mu_n)\subseteq \M(\R^N,\R)$, we write $\mu_n\stackrel{*}{\rightharpoonup}\mu$ and $\mu_n\rightharpoonup\mu$ to signify tight and weak convergence, respectively (the definitions are given in the next subsection). The symbol $\delta_z$ indicates the Dirac delta  concentrated at $z\in\R^N$.

The letter $C$ denotes a generic positive constant whose dependence on other variables is irrelevant in the specific context. Subscripts on $C$ emphasize its dependence from the specified parameters. 

For the sake of brevity, the sentence ``in $\R^N$'' should be understood as ``a.e.\,in $\R^N$''. We would like to point out that the degenerate region $\Sigma$ is \emph{always} a negligible set in the sense of Lebesgue measure. For instance, a function which is positive almost everywhere might well be negative on the whole $\Sigma$. 

\subsection{Functional setting}\label{functsett}

Given an open subset $\Omega\subseteq\R^N$ we introduce the Sobolev space $H^1_\gamma(\Omega)$ as
\begin{displaymath}
    H^1_\gamma(\Omega) = \left\{u\in L^2(\Omega):\,  |\nabla_\gamma u| \in L^2(\Omega) \right\}
\end{displaymath}
and the Beppo-Levi space $D^\gamma(\Omega)$ as 
\begin{displaymath}
    D^\gamma(\Omega) = \left\{u\in L^{2^*_\gamma}(\Omega) : |\nabla_\gamma u|\in L^2(\Omega) \right\},
\end{displaymath}
where $2_\gamma^*$ is the Sobolev critical exponent associated to $N_\gamma$ to be the number
\begin{gather} \label{eq:2stargamma}
	2_\gamma^* \coloneqq \frac{2N_\gamma}{N_\gamma -2}.
\end{gather}
The spaces $H^1_{0,\gamma}(\Omega)$, $D_0^\gamma(\Omega)$ are defined as the completion of $C_c^\infty(\Omega)$ with respect to the norms $\|\cdot\|_2+\|\nabla_\gamma(\cdot)\|_2$ and $\|\nabla_\gamma(\cdot)\|_2$, respectively.
In what follows we will denote the norm of the space $D_0^\gamma(\R^N)$ with the symbol
\begin{equation}\label{normdef}
	\| u \|_\gamma = \left(\int_{\R^N} \|\nabla_\gamma u\| ^2\dz \right)^{1/2}.	
\end{equation}
It is easy to check that $D_0^\gamma(\R^N)$ is a Hilbert space with respect to the inner product
\begin{displaymath}
    \langle u , v \rangle_\gamma = \int_{\R^N} \nabla_\gamma u \nabla_\gamma v \dz.
\end{displaymath}
Moreover, we trivially have the embedding
\begin{displaymath}
    H^1_{0,\gamma}(\R^N)=H^1_{\gamma}(\R^N)\hookrightarrow D_0^\gamma(\R^N)=D^\gamma(\R^N),
\end{displaymath}
while $H^1_\gamma(\R^N)\hookrightarrow L^q(\R^N)$ for any $q\in [2,2^*_\gamma]$ by the results of \cite[p. 1249]{AlvesHolanda23}.

We denote by $D^{-1,2,\gamma}_0(\R^N)$ the dual space of $D_0^\gamma(\R^N)$.
Following \cite[Eq. (3.1)]{alves2024brezis}, we define the best constant of the Sobolev-Gagliardo-Nirenberg embedding $D_0^\gamma(\R^N)\hookrightarrow L^{2^*_\gamma}(\R^N)$, see  \cite[Lemma 5.1]{loiudice2006}, as 
\begin{equation}
	\label{eq:sobconst}
	S \coloneqq \inf_{\substack{ u \in D_0^\gamma(\R^N) \\ u \neq 0}} \frac{\int_{\R^N} \left\vert \nabla_\gamma u \right\vert^2 \dz}{\left( \int_{\R^N} \left\vert u \right\vert^{2_\gamma^*} \dz \right)^{2/2_\gamma^*}}.
\end{equation}
For the sake of completeness, we now recall some embedding theorems in bounded domains.
If $\Omega\subset\R^N$ is bounded and satisfies the regularity condition (D)\footnote{From now on, we will assume that the bounded sets $\Omega\subseteq\R^N$ satisfy condition (D). We point out that the balls $B_\gamma(z,R)$ satisfies (D), as remarked in \cite{abatangelo2024}.} of \cite[Section 5]{abatangelo2024}, then $H^1_\gamma(\Omega)\hookrightarrow L^q(\Omega)$ for every $q\in [1,2^*_\gamma]$, and $H^1_\gamma(\Omega)\compact L^q(\Omega)$, for every $q\in [1,2^*_\gamma)$, see \cite[Proposition 5.5]{abatangelo2024}.
This latter fact, together with the continuous embedding $D_0^\gamma(\R^N)\hookrightarrow H^1_\gamma(\Omega)$,  gives $D_0^\gamma(\R^N)\compact L^q(\Omega)$ for all $q\in [1,2^*_\gamma)$.

A sequence of measures $(\mu_n)\subseteq \M(\R^N,\R)$ converges tightly to a measure $\mu$, written as $\mu_n \stackrel{*}{\rightharpoonup} \mu$, if
\begin{equation}
\label{measconv}
\int_{\R^N} f \, {\rm d}\mu_n \to \int_{\R^N} f \, {\rm d}\mu \quad \mbox{for all} \;\; f\in C_b(\R^N),
\end{equation}
where $C_b(\R^N)$ is the space of the bounded, continuous functions on $\R^N$. On the other hand, $(\mu_n)\subseteq \M(\R^N,\R)$ is said to converge weakly to $\mu$, written as $\mu_n \rightharpoonup \mu$, if \eqref{measconv} holds for all $f\in C_0(\R^N)$, where $C_0(\R^N)$ is the space of the continuous functions that vanish at infinity.\footnote{A function $f$ vanishes at infinity if, for each $\varepsilon >0$, there exists a compact set $K_\varepsilon$ such that $\sup_{x \in K_\varepsilon} |f(x)| < \varepsilon$.} Since $C_0(\R^N)\subseteq C_b(\R^N)$, tight convergence implies weak convergence. Moreover, if $(\mu_n)\subseteq \M(\R^N,\R)$ is bounded, then (up to sub-sequences) $\mu_n\rightharpoonup \mu$ for some $\mu\in \M(\R^N,\R)$: see \cite[Proposition 1.202]{FL}. Notice that weak convergence is the `natural' convergence in the space $\M(\R^N,\R)$, since $\M(\R^N,\R)=(C_0(\R^N))'$. It is worth pointing out that tight convergence can be seen as non-concentration at infinity: see \cite{BBF, BBFS}.

Let $(X,\|\cdot\|_X)$ be a Banach space and $J$ be a functional of class $C^1$ (hereafter indicated as $J\in C^1(X)$). A sequence $(u_n)\subseteq X$ is a Palais-Smale sequence at level $c\in\R$ if $J(u_n) \to c$ in $\R$ and $J'(u_n)\to 0$ in $X^*$. If each Palais-Smale sequence at level $c$ admits a strongly convergent sub-sequence, then $J$ is said to satisfy the Palais-Smale condition at level $c$; briefly, $J$ satisfies ${\rm (PS)}_c$.

A partially ordered set $(A,\leq)$ is said to be downward directed if for any $a,b\in A$ there exists $c\in A$ such that $c\leq a$ and $c\leq b$. We recall that, if $a$ is a minimal element of the downward directed poset $A$, then $a=\min A$: indeed, since $A$ is downward directed, for any $b\in A$ there exists $c\in A$ such that $c\leq a$ and $c\leq b$, and minimality of $a$ forces $c=a$, so that $c\leq b$ for all $b\in A$, i.e., $c=\min A$.

Let $(X,d_X),(Y,d_Y)$ be two metric spaces. A set-valued function $\S\colon X \to 2^Y$\footnote{Here $2^Y$ is the power set of $Y$.} is said to be lower semi-continuous if, for any $x_n\to x$ in $X$ and $y\in\S(x)$, there exists $(y_n)\subseteq Y$ such that $y_n\to y\in Y$ and $y_n\in\S(x_n)$ for all $n\in\N$; it is said to be compact if, for any bounded $K\subseteq X$, the set $\S(K)$ is relatively compact in $Y$.

\subsection{Some tools}\label{sec:tools}
We start by proving a simple result concerning weak convergence of the positive part of functions.
\begin{lemma}
\label{posparts}
Let $(u_n)\subseteq D_0^\gamma(\R^N)$, $u\in D_0^\gamma(\R^N)$ be such that $u_n \rightharpoonup u$ in $D_0^\gamma(\R^N)$. Then $(u_n)_+ \rightharpoonup u_+$ in $ D_0^\gamma(\R^N)$.
\end{lemma}

\begin{proof}
 The proof proceeds with only minor modifications to that of \cite[Lemma 2.1]{BG}, thanks to the compact embedding $D_0^\gamma(\mathbb{R}^N) \compact L^p(B_\gamma(0,R))$ for all fixed $R>0$ (see Section~\ref{functsett}), together with Stampacchia's lemma (cf., e.g., \cite[Corollary 2.2]{gutierrezlanconelli}).
\end{proof}
The following lemma establishes the (S)$_+$ property of the Grushin operator. We first recall the definition.
\begin{defi}
    Let $(X,\| \cdot \|)$ be a Banach space. An operator $A\colon X\to X^*$ is \emph{of type} (S)$_+$ if, for any $(u_n)\subseteq X$ and $u\in X$ such that 
    $$u_n\rightharpoonup u \,\, \hbox{in $X$}, \quad \limsup_{n\to\infty} \langle A(u_n), u_n-u\rangle \le 0,$$
    one has $u_n\to u$ in $X$.
\end{defi}
\begin{lemma}
\label{lemma:S+}
    The operator $A\colon D_0^\gamma(\R^N)\to D^{-1,2,\gamma}_0(\R^N)$ defined by 
        \begin{displaymath}
        \langle A(u), v\rangle \coloneqq \int_{\R^N} \nabla_\gamma u \nabla_\gamma v\dz
        \end{displaymath}
    is of type (S)$_+$.
\end{lemma}
\begin{proof}
Let $(u_n), u\in D_0^\gamma(\R^N)$ such that
\begin{equation}\label{hpS}
u_n\rightharpoonup u \quad \text{in}\,\, D_0^\gamma(\R^N), \quad \limsup_{n\to\infty} \langle A(u_n), u_n-u\rangle \le 0.
\end{equation}
We have
\[
\langle A(u_n) - A(u), u_n - u\rangle \ge 0 \quad \forall n\in\N,
\]
implying
\[
\liminf_{n\to\infty} \langle A(u_n) - A(u), u_n - u \rangle \ge 0.
\]
Therefore by \eqref{hpS}
\[
\lim_{n\to\infty} \langle A(u_n) - A(u), u_n - u \rangle = 0,
\]
which is equivalent to require $\lVert u_n - u \rVert_\gamma ^2 \to 0$ by recalling \eqref{normdef}.    
\end{proof} 
Let us introduce two lemmas that are useful for handling concentration of compactness at points and at infinity, respectively.

The first concentration compactness lemma was proved in \cite[Theorem 3.3]{MMSp} for a $p$-Grushin operator in a bounded domain $\Omega\subseteq \R^N$, but the case in the whole $\R^N$, where the tight convergence appears, follows with the same argument.
\begin{thm}
	\label{thm:concentration}
    Suppose $(u_n)\subseteq D_0^\gamma(\R^N)$ to be such that $u_n\rightharpoonup u$ in $D_0^\gamma(\R^N)$, and both $| \nabla_\gamma u_n|\rightharpoonup \mu$, $| u_n |^{2^*_\gamma} \stackrel{*}{\rightharpoonup}\nu$ in the sense of measures, for some $u\in D_0^\gamma(\R^N)$ and $\mu,\nu$ bounded non-negative measures on $\R^N$.
 	Then there exist an at most countable set $\A$, a family $\lbrace z_j\rbrace_{j\in \A}$ of distinct points in $\R^N$ and two families of numbers $\lbrace \mu_j \rbrace_{j\in \A}$, $\lbrace \nu_j \rbrace_{j\in \A}\subseteq (0,+\infty)$ such that
		\begin{equation*}
		   \nu = | u |^{2^*_\gamma} + \sum_{j\in \A} \nu_j \delta_{z_j},  \quad  
              \mu\ge | \nabla_\gamma u|^2 + \sum_{j\in \A}\mu_j \delta_{z_j}, \quad
              \mu_j \ge S \nu_j^{\frac{2}{2^*_\gamma}}\quad\hbox{for all $j\in \A$,} 
		\end{equation*}
		where $\delta_z$ is the Dirac measure concentrated at $z\in\R^N$ and
		$S$ is defined in \eqref{eq:sobconst}. In particular,
		\begin{displaymath}
			\sum_{j\in \A} (\nu_j)^{\frac{2}{2^*_\gamma}}<\infty.
		\end{displaymath}
\end{thm}
Lemma \ref{thm:concentration} requires the tight convergence of the sequence of measures involving $2^*_\gamma$. However, the direct proof of this condition is rather difficult and technical, also for $\gamma=0$. Thus, we report \cite[Lemma 3.3]{alves2024brezis}, that is a version of Lemma \ref{thm:concentration} known as {\it escape to infinity principle}, where the concentration at infinity is contained in the parameters $\nu_\infty$ and $\mu_\infty$. 

\begin{lemma}
\label{bennaoum}
Suppose that $(u_n) \subseteq D_0^\gamma(\R^N)$ is bounded and define
$$ \nu_\infty \coloneqq  \lim_{R\to+\infty} \limsup_{n\to\infty} \int_{B^c_\gamma(0,R)} |u_n|^{2^*_\gamma} \dz, \quad \mu_\infty \coloneqq  \lim_{R\to+\infty} \limsup_{n\to\infty} \int_{B^c_\gamma(0,R)} |\nabla_\gamma u_n|^2 \dz.$$
Then, $S\nu_\infty^{2/2^*_\gamma} \leq \mu_\infty$ and
\begin{displaymath}
\begin{split}
\limsup_{n\to\infty} \int_{\R^N} |u_n|^{2^*_\gamma} \dz = \int_{\R^N} \, {\rm d}\nu + \nu_\infty, \qquad
\limsup_{n\to\infty} \int_{\R^N} |\nabla_\gamma u_n|^2 \dz = \int_{\R^N} \, {\rm d}\mu + \mu_\infty, 
\end{split}
\end{displaymath}
where $\nu$ and $\mu$ are as in Lemma \ref{thm:concentration}.
\end{lemma}
Lemmas \ref{thm:concentration} and \ref{bennaoum} will be used to avoid concentration both at points, i.e. $\nu_j=\mu_j=0$ for all $j\in \A$, and at
infinity, i.e. $\nu_\infty=\mu_\infty=0$.

\begin{rmk}
\label{nujmuj}
As observed in \cite[Remark 2.4]{BG}, the numbers $\nu_j$, $\mu_j$ in Theorem \ref{thm:concentration} represent the concentration of the measures $\nu$ and $\mu$ at the points $z_j$ for all $j\in \A$, i.e. $\nu_j=\nu(\{z_j\})$ and $\mu_j=\mu(\{z_j\})$.
\end{rmk}
We recall two celebrated results in Nonlinear Analysis.
\begin{thm}[Ambrosetti, Rabinowitz; cf. {\cite[Theorem 5.40]{MMP}}]
\label{mountainpass}
Let $(X,\|\cdot\|_X)$ be a Banach space and $J\in C^1(X)$. Let $u_0,u_1\in X$, $\rho>0$ such that
\begin{displaymath}
    \max\{J(u_0),J(u_1)\} < \inf_{\partial B(u_0,\rho)} J \eqqcolon  b \quad \mbox{and} \quad \|u_1-u_0\|_X > \rho.
\end{displaymath}
Set
\begin{displaymath} 
    \Phi\coloneqq  \{\phi\in C^0([0,1];X) : \, \phi(0)=u_0, \, \phi(1)=u_1\}, \quad c_M\coloneqq  \inf_{\phi\in\Phi} \sup_{t\in[0,1]} J(\phi(t)).
\end{displaymath} 
If $J$ satisfies ${\rm (PS)}_{c_M}$, then $c_M\geq b$ and there exists $u\in X$ such that both $J(u)=c_M$ and $J'(u)=0$. Moreover, if $c_M=b$, then $u$ can be taken on $\partial B(u_0,\rho)$.
\end{thm}

\begin{thm}[Schauder; cf. {\cite[Theorem 6.3.2 p.119]{GD}}]
\label{schauder}
Let $K$ be a non-empty bounded convex subset of a normed linear space $E$, and let $T\colon K\to K$ be a compact map. Then $T$ has a fixed point.
\end{thm}
We will apply the following result about boundedness of sequences defined by recursion.

\begin{lemma}{\cite[Lemma 2.7]{BG}}
\label{reclemma}
Let the sequence $(b_k) \subseteq [0,+\infty)$ satisfy, for some $c>0$, $K>1$ and $\alpha>1$, the recursion
\begin{displaymath}
b_k \leq c + Kb_{k-1}^{\alpha} \quad \mbox{for all}\;\; k\in\N.
\end{displaymath}
If
\begin{equation}
\label{smallnessconds}
K b_0^{\alpha-1}\leq\frac{1}{2} \quad \mbox{and} \quad Kc^{\alpha-1}<2^{-\alpha},
\end{equation}
then the sequence $(b_k)$ is bounded.
\end{lemma}

\medskip
For technical reasons we will need the following generalization of a well-known inequality. The proof is elementary, but we include it for completeness.

\begin{lemma}[Peter-Paul's inequality]\label{peterpaul}
Suppose that $\eps>0$, $b_1,\ldots,b_k\geq 0$ and $p_1,\ldots,p_k\in(1,+\infty)$ satisfy $\sum_{j=1}^k p_j^{-1} = 1$. The following inequality holds:
\[
    \prod_{j=1}^k b_j \leq \eps \sum_{j=1}^{k-1} \frac{b_j^{p_j}}{p_j} + \frac{1}{\eps^{p_k-1}}\cdot \frac{b_k^{p_k}}{p_k}.
\]
\end{lemma}
\begin{proof}
Take any $b_1,\ldots,b_k\geq 0$ and $p_1,\ldots,p_k\in(1,+\infty)$ such that $\sum_{j=1}^k p_j^{-1} = 1$. By Young's inequality for products,  see for instance \cite[p. 92]{B},
we have
\[
\prod_{j=1}^k b_j \leq \sum_{j=1}^{k} \frac{b_j^{p_j}}{p_j}.
\]
For any $\eps>0$ we write
\begin{align*}
    \prod_{j=1}^k b_j &=\eps\prod_{j=1}^{k-1} b_j\,\cdot \frac{b_k}{\eps} 
    \leq \eps \left( \sum_{j=1}^{k-1} \frac{b_j^{p_j}}{p_j} + \frac{1}{p_k} \left( \frac{b_k}{\eps} \right)^{p_k}\right) \\
    &= \eps \sum_{j=1}^{k-1} \frac{b_j^{p_j}}{p_j} + \frac{1}{\eps^{p_k-1}} \frac{b_k^{p_k}}{p_k}.
\end{align*}
\end{proof}
The following result is a straightforward adaptation of the weak comparison principle (cf. \cite[Theorem 3.4.1]{PS}) to our framework. The proof relies on the same ideas as in \cite[Lemma 2.8]{BG}; however, we include it here for completeness.
\begin{lemma}
\label{weakcomp}
Suppose that $u,v\in D_0^\gamma(\R^N)$ satisfy
\begin{equation}
\label{weakcomptest}
\langle -\Delta_\gamma u,\varphi \rangle \leq \langle -\Delta_\gamma v,\varphi\rangle
\end{equation}
for all $\varphi\in D_0^\gamma(\R^N)$ such that $\varphi\geq 0$ in $\R^N$ and $\varphi\equiv 0$ on $\{u\leq v\}$. Then $u\leq v$ in $\R^N$.
\end{lemma}
\begin{proof}
Testing \eqref{weakcomptest} with $(u-v)_+\in D_0^\gamma(\R^N)$, and by Stampacchia's lemma (cf., e.g., \cite[Corollary 2.2]{gutierrezlanconelli}), gives
$$ \int_{\{u>v\}} \left( \nabla_\gamma u-\nabla_\gamma v\right)^2 \dz \leq 0. $$
Hence, $\| (u-v)_+\|_\gamma =0$ in $\R^N$, which implies $(u-v)_+=0$ in $\R^N$, since $(u-v)_+\in D_0^\gamma(\R^N)$.
\end{proof}
We also state a weak comparison principle for exterior domains, inspired by \cite[Corollary p.830]{DB}. We omit the proof, since it is similar to the proof of Lemma \ref{weakcomp}.
\begin{lemma}
\label{weakcomp_ext}
Let $\Omega \subset \R^N$ be a smooth, bounded domain. Suppose that $u, v \colon \Omega^c \to \R$ are measurable functions such that $|\nabla_\gamma u|,|\nabla_\gamma v| \in L^2(\Omega^c)$. Suppose moreover that 
\begin{displaymath}
    \langle-\Delta_{\gamma} u,\varphi\rangle \leq \langle-\Delta_{\gamma} v,\varphi \rangle 
\end{displaymath}
for all $\varphi\in D^\gamma(\Omega^c)$ such that $\varphi\ge0$ in $\Omega^c$. Moreover, assume $u \leq v$ on $\partial \Omega$, i.e., $(u - v)^+ \in D_0^\gamma(\Omega^c)$. Then $u \leq v$ in $\Omega^c$.
\end{lemma}

We will also need the following doubling Lemma.
\begin{lemma}[{\cite[Theorem 5.1]{Souplet}}]\label{lem_dou}
Let $(X,d)$ be a complete metric space, and let $\emptyset \neq D \subset W \subset X$, with $W$ closed. Set $\Gamma\coloneqq W \setminus D$. Finally, let $k>0$ and $M\colon D \to (0,+\infty)$ be bounded on compact subsets of $D$. If $z\in D$ is such that
$$M(z)>\frac{2k}{d(z,\Gamma)},$$
then there exists $\xi\in D$ such that
\begin{equation}\label{cond_dou}
M(\xi)>\frac{2k}{d(\xi,\Gamma)}, \quad M(\xi)\ge M(z),
\end{equation}
and
$$M(\hat{z})\le 2M(\xi) \quad \mbox{for all} \;\; \hat{z}\in D\cap \overline{B}_X(\xi,kM^{-1}(\xi)),$$
where $\overline{B}_X(z_0,r)$ is the closure of the $X$-ball having center $z_0$ and radius $r$.
\end{lemma}

\begin{rmk}
Let $X=\R^N$ and $\Omega$ be an open subset of $\R^N$. Put $D\coloneqq \Omega$ and $W\coloneqq \overline{\Omega}$, so that $\Gamma=\partial\Omega$. Then $\overline{B}_X(\xi,kM^{-1}(\xi))\subset D$ for all $\xi\in D$. Indeed, since $D$ is open, $\eqref{cond_dou}$ implies that
$$d(\xi,X\setminus D)=d(\xi,\Gamma)>2kM^{-1}(\xi).$$
\end{rmk}

We conclude this section recalling the definition of the weak Lebesgue spaces.  For any $s\in(0,\infty)$, we define the weak Lebesgue space $L^{s,\infty}(\R^N)$ as the set of all measurable functions $u\colon \R^N\to\R$ such that
$$\|u\|_{s,\infty}\coloneqq\sup_{h>0}\left(h \left|\left\{\left|u\right|>h\right\}\right|^{1/s}\right)<\infty.$$
The quantity $\|\cdot\|_{s,\infty}$ makes $L^{s,\infty}(\R^N)$ a quasi-normed space (see for instance \cite{Gra}). Moreover, the embedding $L^s(\R^N)\hookrightarrow L^{s,\infty}(\R^N)$ is continuous (see \cite[Proposition 1.1.6]{Gra}). 
The same holds for
\begin{equation}
\label{embedding}
L^{s,\infty}(\Omega)\hookrightarrow L^{s-\eps}(\Omega) \quad \mbox{for all} \;\; \eps\in(0,s),
\end{equation}
provided $\Omega$ has finite measure (see, e.g., \cite[Exercise 1.1.11]{Gra}).
Incidentally, we recall the following interpolation inequality (see \cite[Proposition 1.1.14]{Gra}): given any $0<p,q\leq\infty$,
\begin{equation}
\label{interpolation}
\|f\|_r \leq C \|f\|_{p,\infty}^\iota \|f\|_{q,\infty}^{1-\iota} \quad \mbox{for all} \;\; f\in L^{p,\infty}(\R^N)\cap L^{q,\infty}(\R^N),
\end{equation}
where $\frac{1}{r}=\frac{\iota}{p}+\frac{1-\iota}{q}$ and $C>0$ is a suitable constant depending on $p,q,r$. We also set
\begin{displaymath}
    2_{*,\gamma}\coloneqq\frac{2(N_\gamma-1)}{N_\gamma-2}.
\end{displaymath}

\section{The existence result}
\label{sec:existence}
Hereafter we will tacitly retain assumption \ref{hypf}. 
The present section is devoted to prove the first part of Theorem \ref{mainthm}, that is the existence of at least one positive weak solution to \eqref{prob}, i.e.

\begin{thm}
\label{exsol}
There exists $\Lambda>0$ such that, for any $\lambda_1\in (0,\Lambda)$ and any $\lambda_2 \in [0,\lambda_1]$, the problem 
\begin{equation}
\label{probnotdecay}
\left\{
\begin{alignedat}{2}
-\Delta_\gamma u &=\lambda_1  w_1 (z) u^{-\eta} + \lambda_2 w_2(z) \lvert \nabla_\gamma u \rvert^{r-1}+ u^{2^*_\gamma-1} \quad &&\mbox{in} \;\; \R^N, \\
u &> 0 \quad &&\mbox{in} \;\; \R^N, \\
\end{alignedat}
\right.
\end{equation}
admits a solution $u\in D_0^\gamma(\R^N)$.
\end{thm}
We first prove a Lemma, which ensures the existence of a sub-solution to \eqref{probnotdecay}. We also establish its decay as $d(z)\to+\infty$, under stronger assumptions on $\ell$ and $w_1$.
\begin{lemma}\label{subsol}
Let $\eta\in(0,1)$, then there exists a unique $u\in  C^{0,\tau}_\loc(\R^N)$ solution to
\begin{equation}
\label{subprob}
\left\{ \begin{alignedat}{2}
-\Delta_\gamma u &= w_1(z) u^{-\eta} \quad &&\mbox{in} \;\; \R^N, \\
u &> 0 \quad &&\mbox{in} \;\; \R^N. \\
\end{alignedat}
\right.
\end{equation}
Moreover, if \ref{condw1} holds, then $w_1u^{-\eta}\in  L^1(\R^N)\cap L^\infty(\R^N) $.
\end{lemma}
\begin{proof}
For all $n\in\N$, consider the regularized problems
\begin{equation}
\label{subregular}
\tag{${\rm\underline{P}}_n$}
-\Delta_\gamma u_n = w_1(z) \left((u_n)_++\frac{1}{n}\right)^{-\eta} \quad \mbox{in} \;\; \R^N.
\end{equation}
Fix any $n\in\N$. Direct methods of Calculus of Variations (see \cite[Theorem I.1.2]{S})
ensure that there exists $u_n\in D_0^\gamma(\R^N)$ solution to \eqref{subregular}. 
The (local) H\"{o}lder-continuity of the solutions $u_n$ comes from the non-homogeneous Harnack inequality for $X$-elliptic operators established in \cite[Theorem 5.5]{gutierrezlanconelli} as already observed in \cite{kogojlanconelli}. 

Taking $(u_n)_-$ as test function in \eqref{subregular} we find that $\| (u_n)_- \|_\gamma =0$, so $u_n\ge 0$ a.e. in $\R^N$. Recall that $\Sigma = \{0\} \times \R^\ell$ is the degenerate set for the Grushin operator which is a uniformly elliptic operator on a bounded domain $\Omega\subset \R^N\setminus\Sigma$. 
Let $\Omega$ be any bounded domain contained in any connected component $\mathcal{C}$ in $\R^N\setminus\Sigma$ \footnote{It is easy to see that $\R^N\setminus\Sigma$ is disconnected if and only if $\ell\ge N-2$.}. By standard regularity theory, $u_n\in C^2(\Omega)$ and either $u_n\equiv 0$ on $\Omega$ or $u_n>0$ in $\Omega$ by the strong maximum principle for uniformly elliptic operators. In the former case $u_n\equiv 0$ on $\mathcal{C}$. Clearly, $u_n\equiv 0$ in $\R^N$ is not a solution of \eqref{subregular}. 
Suppose now that $u_n\equiv 0$ on $\mathcal{C}$. 
If we consider as a test function in \eqref{subregular} a cut-off function with compact support contained in $\mathcal{C}$, the left-hand side of \eqref{subregular} vanishes, while the right-hand side remains strictly positive. Hence $u_n>0$ in $\R^N\setminus\Sigma$.

By \cite[Sect. 3.1]{dambrosiomitidieripohozaev} and recalling that $u_n$ is a continuous function, we have the following representation formula
\begin{displaymath}
    u_n(z) = C \int_{\R^N} \frac{w_1(\xi) \left( \left( (u_n(\xi))_+ + 1/n \right)^{-\eta} \right)}{d(z-\xi)^{N_\gamma -2}} \,\mathrm{d}\xi \quad \text{for a.e.\footnotemark} \ z \in \Sigma,
\end{displaymath}
\footnotetext{With respect to the Lebesgue measure in $\R^\ell$.}
where $d$ is the distance defined in \eqref{eq:d}. The nonnegativity of $w_1$ in condition \ref{hypf} guarantees that $u_n>0$ a.e. in $\Sigma$. 

Testing \eqref{subregular} with $u_n$, besides using H\"{o}lder's and Sobolev's inequalities, yields
\begin{align*}
\|\nabla_\gamma u_n\|_2^2 &= \int_{\R^N} w_1\left(u_n+\frac{1}{n}\right)^{-\eta} u_n \dz \leq \int_{\R^N} w_1u_n^{1-\eta} \dz \leq \|w_1\|_\zeta \|u_n\|_{2^*_\gamma}^{1-\eta} \\
&\leq S^{-\frac{1-\eta}{2}}\|w_1\|_\zeta \|\nabla_\gamma u_n\|_2^{1-\eta},
\end{align*}
provided that $\zeta>1$ satisfies $\frac{1}{\zeta}+\frac{1-\eta}{2^*_\gamma}=1$. We deduce $\|\nabla_\gamma u_n\|_2 \leq (S^{-\frac{1-\eta}{2}}\|w_1\|_\zeta)^{\frac{1}{1+\eta}}$, so that $(u_n)$ is bounded in $ D_0^\gamma(\R^N)$. By reflexivity, $u_n\rightharpoonup u$ in $ D_0^\gamma(\R^N)$ for some $u\in D_0^\gamma(\R^N)$, up to sub-sequences.

Observing that $u_n+\frac{1}{n}>u_{n+1}+\frac{1}{n+1}$ on $\{u_n>u_{n+1}\}$ we get the weak inequality
$$-\Delta_\gamma u_n = w_1(z) \left(u_n+\frac{1}{n}\right)^{-\eta} \leq w_1(z)\left(u_{n+1}+\frac{1}{n+1}\right)^{-\eta} = -\Delta_\gamma u_{n+1} \quad \mbox{on} \;\; \{u_n>u_{n+1}\}.
$$
According to Lemma \ref{weakcomp}, it turns out that $u_n\leq u_{n+1}$ in $\R^N$. Thus, we can define a measurable function $\tilde{u}$ such that $u_n \nearrow \tilde{u}$ in $\R^N$.

We show that $u=\tilde u$ in $\R^N$. For all $k\in\N$, $ D_0^\gamma(\R^N)\compact L^p(B_\gamma(0,k))$, so $u_n \to u$ in $L^p(B_\gamma(0,k))$ and $u_n\to u$ in $B_\gamma(0,k)$, up to sub-sequences.
A diagonal argument ensures that $u_n\to u$ in $\R^N$, whence $u=\tilde{u}$ in $\R^N$ (see, e.g., \cite[p.3044]{GG1} for details).

Now we prove that we can pass to the limit in the weak formulation of \eqref{subregular}. Clearly,
\begin{equation}
\label{sublimit1}
\lim_{n\to\infty} \int_{\R^N}\nabla_\gamma u_n \nabla_\gamma\varphi\dz = \int_{\R^N}  \nabla_\gamma u \nabla_\gamma \varphi \dz.
\end{equation}
 On the other hand, splitting $\varphi=\varphi_+-\varphi_-$, Beppo Levi's monotone convergence theorem guarantees
\begin{displaymath}
\lim_{n\to\infty} \int_{\R^N} w_1\left(u_n+\frac{1}{n}\right)^{-\eta}\varphi_+ \dz = \int_{\R^N} w_1u^{-\eta}\varphi_+ \dz.
\end{displaymath}
Analogously,
\begin{displaymath}
\lim_{n\to\infty} \int_{\R^N} w_1\left(u_n+\frac{1}{n}\right)^{-\eta}\varphi_- \dz = \int_{\R^N} w_1u^{-\eta}\varphi_- \dz,
\end{displaymath}
giving
\begin{equation}
\label{sublimit2}
\lim_{n\to\infty} \int_{\R^N} w_1\left(u_n+\frac{1}{n}\right)^{-\eta}\varphi \dz = \int_{\R^N} w_1u^{-\eta}\varphi \dz.
\end{equation}
Hence \eqref{sublimit1}--\eqref{sublimit2} ensure that $u$ is positive and it solves the equation in  \eqref{subprob}. 

Uniqueness of $u$ can be proved by comparison as follows. If $u_1,u_2\in  D_0^\gamma(\R^N)$ are two solutions to \eqref{subprob} then Lemma \ref{weakcomp} yields
$u_1\leq u_2$ in $\R^N$, since $w_1u_1^{-\eta}<w_1u_2^{-\eta}$ on $\{u_1>u_2\}$. Reversing the roles of $u_1$ and $u_2$ leads to $u_2\leq u_1$ in $\R^N$, whence $u_1=u_2$ in $\R^N$.

To prove that $u\in C^{0,\tau}_\loc(\R^N)$ we first notice that $-\Delta_\gamma u \geq 0$ in $B_\gamma^c(0,R)$ for any fixed $R>0$. Given $\sigma>0$, set $\Phi_\sigma(z)\coloneqq \sigma\Gamma(z)$, where $\Gamma$ is the function defined in \eqref{gammadef}, and observe that $-\Delta_\gamma \Phi_\sigma = 0$ in $B^c(0,R)$. According to the weak comparison principle in exterior domains, Lemma \ref{weakcomp_ext}, there exists $\sigma>0$ small enough such that  
\begin{displaymath}
u\geq \Phi_\sigma\quad \mbox{in}\;\; B_\gamma^c(0,R).
\end{displaymath}
This implies that $u$ satisfies
\begin{equation}\label{eqcomplem}
-\Delta_\gamma u = w_1(z) u^{-\eta} \leq w_1(z) \sigma^{-\eta}d(z)^{\eta\left(N_\gamma - 2\right)} \quad \mbox{in}\;\; B_\gamma^c(0,R).
\end{equation}
Since $u_n\nearrow u$ in $\R^N$, then 
\begin{equation}\label{upos}
u\ge u_1, \quad\text{where} \quad \inf_{B_\gamma(0,R)} u_1>0 \quad\text{for all} \quad  R>0,
\end{equation}
being $u_1$ solution of \eqref{subregular} with $n=1$, implying 
\begin{equation}\label{linfloc}
w_1 u^{-\eta}\in L^\infty_{\loc}(\R^N),
\end{equation}
so we can apply again the non-homogeneous Harnack inequality, \cite[Theorem 5.5]{gutierrezlanconelli}, and the arguments in \cite{kogojlanconelli} which yield that $u\in C^{0,\tau}_\loc(\R^N).$ 

It remains to prove $w_1u^{-\eta}\in L^1(\R^N) \cap L^\infty(\R^N).$ First, recalling \eqref{linfloc}, we have $w_1u^{-\eta}\in L^1(B_\gamma(0,R)) \cap L^\infty(B_\gamma(0,R)).$

On the other hand, by \eqref{eqcomplem}-\ref{condw1} we have
$$w_1u^{-\eta} \leq C d(z)^{\eta(N_\gamma-2)-\delta-2\gamma} \quad \mbox{in} \;\; B^c_\gamma(0,R),$$ 
which implies $w_1u^{-\eta}\in L^1(B^c_\gamma(0,R))\cap L^\infty(B^c_\gamma(0,R))$ since $\delta>\eta(N_\gamma-2)+N_\gamma-2\gamma$, by \ref{condw1}. Summarizing, $w_1u^{-\eta}\in L^1(\R^N) \cap L^\infty(\R^N)$.
\end{proof}

\begin{rmk}\label{rmknat}
A strong maximum principle for weak solutions has been proved in \cite[Theorem 2.6]{balbiswas}, which can also be applied here when $m=1$ and $\gamma\ge 1$.
When $\gamma\in\N$ , following an idea of \cite{biagi2022sublinear}, we can provide an alternative proof of the positivity of $u_n$. We recall that, if $\gamma\in\N$, $-\Delta_\gamma$ satisfies the H\"{o}rmander condition. Since $w_1\ge 0$, we see that $\int_{\R^N} \nabla_\gamma u_n \nabla_\gamma \varphi \dz \geq 0$ for every $\varphi \in C_c^\infty(\R^N)$ such that $\varphi \geq 0$. Hence $\int_{\R^N} u (-\Delta_\gamma \varphi)\dz \geq 0$ for every such $\varphi$. Recalling that $u_n$ is a continuous function, it follows that $u_n$ is a viscosity sub-solution of the operator $-\Delta_\gamma$. The strong maximum principle proved in \cite[Corollary 1.4]{bardi2019new} yields that either $u_n \equiv 0$ or $u_n>0$ in $\R^N$. Again, $u_n\equiv 0$ cannot be solution of \eqref{subregular}.
\end{rmk}

\begin{rmk} \label{rmknat2}
Remark \ref{rmknat} shows a typical difficulty arising in the Grushin operator's framework,  which can be overcome in the case $\gamma\in\N$: the strict positivity of the solution $u$ cannot be expected for our weak solutions. Actually, a Hopf Lemma was proved in \cite{Monti_2006} which would imply $u>0$, provided that $u \in C^2(\R^N)$. However, a full theory of elliptic regularity is not available for weak solutions to degenerate operators, and we do not expect $u$ to be of class $C^2$ across the degenerate set $\Sigma$.
\end{rmk}

Now we prove the decay of the solution $u$ found in Lemma \ref{subsol}. 
\begin{lemma}
\label{decay_subsolution}
    Let $u\in C^{0,\tau}_\loc(\R^N)$ as in Lemma \ref{subsol}, $\ell>4$ and set
\begin{displaymath}
\delta_1 \coloneqq N_\gamma+\eta\,(N_\gamma-2) \quad \mbox{and} \quad \delta_2 \coloneqq\eta(N_\gamma-2) - 2\gamma + (\ell-2)(\gamma+1).   
\end{displaymath} 
If there exist constants $c_1>0$, $R>0$ and $\delta$ such that, for every $z\in \R^N$ with $d(z)>R$, 
\begin{equation}  \label{condw1'} 
    w_1(z)\le c_1 d(z)^{-\delta-2\gamma}, \qquad \delta>\max\{\delta_1,\delta_2\},
\end{equation}
then
    \begin{equation}
        \label{decayd}
            u(z)\to 0\quad \mbox{as}\quad d(z)\to+\infty.
    \end{equation}
\end{lemma}

\begin{proof}
Let $a$ a (small) fixed parameter and define $\Sigma_a\coloneqq \{z = (x,y)\in\R^N:\, |x|\le a\}$. Set also $\Sigma_a^c \coloneqq \R^N\setminus \Sigma_a$. We divide the proof into two cases, according if $z\in \Sigma_a$ or $z\in \Sigma_a^c$.
\medskip

{\bf Case I:} If $z\in B_\gamma^c(0,R)\cap\Sigma_a^c$. Take any $q<0$, $R, M>0$, and define 
\begin{displaymath}
    \Psi(z)\coloneqq M\left( |x|^{2\left(\gamma +1\right)} + (\gamma+1)^2 |y|^2\right)^{\frac{q}{2\left(\gamma+1\right)}} = M d(z)^q, \qquad z\in B_\gamma^c(0,R)\cap\Sigma_a^c.
\end{displaymath}
By direct computation we have
$$\nabla_x \Psi(z)= M q d(z)^{q-2(\gamma+1)}|x|^{2\gamma}x,\qquad \nabla_y \Psi(z)= M q(\gamma+1) d(z)^{q-2(\gamma+1)}y$$
so that
$$\nabla_\gamma\Psi(z)= M q d(z)^{q-2(\gamma+1)}|x|^{\gamma}(|x|^{\gamma}x, (\gamma+1) y).$$
By $(3.2)$-$(3.6)$\footnote{In  $(3.2)$-$(3.6)$, the authors deal with the case $m,\ell \geq 2$. They observe, as indicated in $(3.12)$, that their conclusions remain valid even when $m = 1$.} in \cite{abatangelo2024}, choosing $\rho = d(z)$, we have
$$\int_{\R^m}\int_{\R^\ell}|\nabla_\gamma\Psi|^2\dz= M^2  \int\limits_{\{z:\, d(z) = 1\}}\int_R^\infty \rho^{2q+N_\gamma-3}(\sin\varphi)^{\frac{2\gamma}{\gamma+1}} \,{\rm d}\rho\, {\rm d} \mathcal H_\gamma^{N-1},$$
where $\mathcal{H}_\gamma^{N-1}$ is the $(N-1)$-dimensional measure defined in \cite[pag. 8]{abatangelo2024} and $\varphi\in (0,\pi/2)$. So  $|\nabla_\gamma \Psi| \in L^2(B_\gamma^c(0,R))$ if $q<1-N_\gamma/2$. 
Moreover, note that by \cite[Eq. (5.22)]{ambrosio2003}, it is possible to write
$$
\int_{\R^N} |\Psi(z)|^{2^*_\gamma}\dz = M^{2^*_\gamma} s_N \int_R^{+\infty} \rho^{N_\gamma-1} \rho^{2^*_\gamma q}\,\mathrm{d}\rho,
$$
where $s_N$ is an explicit constant that depends only on $N$, implying $\Psi\in L^{2^*_\gamma}(B_\gamma^c(0,R))$ again if $q<1-N_\gamma/2$.
Furthermore,
\begin{displaymath}
\Delta_\gamma\Psi=Mq(q-2+N_\gamma)\left( |x|^{2\left(\gamma +1\right)} + (\gamma+1)^2 |y|^2\right)^{\frac{q}{2\left(\gamma+1\right)}-1}|x|^{2\gamma}
=Mq(q-2+N_\gamma)d(z)^{q-2-2\gamma}|x|^{2\gamma}.
\end{displaymath}
So $\Psi\in D^\gamma(B_\gamma^c(0,R))$ and it solves
\begin{displaymath}
\left\{ \begin{alignedat}{2}
-\Delta_\gamma  \Psi &= C_Md(z)^{q-2-2\gamma}|x|^{2\gamma} \quad &&\mbox{in} \;\; B_\gamma^c(0,R), \\
\Psi &\to 0 \quad &&\mbox{as} \;\; d(z)\to+\infty,
\end{alignedat}
\right.
\end{displaymath} 
where 
\begin{displaymath}
    C_M \coloneqq -M q (q-2+N_\gamma),
\end{displaymath}
which is positive if we impose the further condition $q>2-N_\gamma$. That is, we suppose $2-N_\gamma<q<1-N_\gamma/2$. 

Since $z\in B_\gamma^c(0,R)\cap\Sigma_a^c$, recalling that $u$ satisfies \eqref{eqcomplem}, by \eqref{condw1'}, we have 
\begin{multline}  
\label{compuPsi}
-\Delta_\gamma u\le c_1 \sigma^{-\eta}d(z)^{\eta(N_\gamma-2)-\delta -2\gamma}\le C_M d(z)^{q-2-2\gamma} a^{2\gamma} \le C_M d(z)^{q-2-2\gamma} |x|^{2\gamma} = \Delta_\gamma \Psi,
\end{multline}
enlarging $M$ if necessary. Note that it is sufficient for the estimate \eqref{compuPsi} to have $$\eta(N_\gamma-2)-\delta-2\gamma \le q-2-2\gamma,$$ that is
$\delta\ge \eta(N_\gamma -2) -q+2,$
which is satisfied when $q>2-N_\gamma$ since $\delta>\delta_1$.

Choose $M$ large enough such that $\Psi(R)>u(R)$. Applying Lemma \ref{weakcomp_ext} to $u$ and $\Psi$ in $\R^N\setminus\Sigma_a$ we obtain
\begin{displaymath}
    u(z)\le \Psi(z) \quad\hbox{for all $z\in B_\gamma^c(0,R)\cap \Sigma_a^c$},
\end{displaymath}
implying the required decay \eqref{decayd} in $\Sigma_a^c$.
\medskip

{\bf Case II:} If $z\in B_\gamma^c(0,R)\cap\Sigma_a$, we define
$\tilde\Psi(z) \coloneqq \tilde{M} f(x) g(y)$, 
where $\tilde{M}>0$. By direct computations, we obtain
\begin{equation}\label{laplPsi}
\Delta_\gamma \tilde\Psi = \tilde{M} \Delta_x f(x)\cdot  g(y) + \tilde{M} |x|^{2\gamma} \Delta_y g(y)\cdot f(x).
\end{equation}
We aim to choose $f$ and $g$ accurately in order to recover again the estimate $-\Delta_\gamma u\le -\Delta_\gamma\tilde\Psi$.
Define $f(x) \coloneqq e^{-|x|^2}$ and $g(y) \coloneqq|y|^q$ for some $q<0$, for all $z = (x,y) \in \Sigma_a$. 

First we prove that $\tilde\Psi \in  D^\gamma(B_\gamma^c(0,R)\cap \Sigma_a)$ for some values of $q$. 
By computing
$$
\nabla_\gamma \tilde\Psi = e^{-|x|^2}\left(-2x |y|^q, |x|^{\gamma} |y|^{q-2} q \right)
$$
and taking $z\in B_\gamma^c(0,R)\cap \Sigma_a$, that is we can suppose $|y|\ge a$, we can estimate
\begin{displaymath}   
    |\nabla_\gamma\tilde\Psi|^2 = e^{-2|x|^2} |y|^{2q} \left[4|x|^2 + q^2 |x|^{2\gamma} |y|^{-2} \right]\le  e^{-2|x|^2} |y|^{2q} \left[4|x|^2 + q^2 |x|^{2\gamma} a^{-2} \right],
\end{displaymath}
which is integrable in $B_\gamma^c(0,R)\cap \Sigma_a$ if $q<-\ell/2$. 
Moreover,
$$
\int_{ B_\gamma^c(0,R)\cap\Sigma_a} \tilde{\Psi}^{2^*_\gamma}\dz =\tilde M \int_{|x|<a} \left( e^{-|x|^2}\right)^{2^*_\gamma} \dx \cdot\int_{ B_\gamma^c(0,R) \cap \R^\ell} |y|^{q2^*_\gamma}\dy,
$$
so $\tilde\Psi\in L^{2^*_\gamma}(B_\gamma^c(0,R)\cap\Sigma_a)$ if $q<-\ell/2^*_\gamma$.
Since $2^*_\gamma\ge 2$ we consider $q<-\ell/2$. 

For any $z\in \Sigma_a\cap B_\gamma^c(0,R)$ we have
\begin{displaymath}
    \Delta_x f(x) = e^{-|x|^2}(4|x|^2 - 2m), \quad \Delta_y g(y) =  q(q-2+\ell) |y|^{q-2},    
\end{displaymath}
which are both negative if $a$ is sufficiently small and $q>2-\ell$. This latter condition on $q$ fits with $q<-\ell/2$ since we are supposing $\ell>4$. Recalling \eqref{laplPsi}, we have
\begin{displaymath}
-\Delta_\gamma\tilde \Psi = \tilde{M} (2m - 4|x|^2) e^{-|x|^2} |y|^q + C_{\tilde M} |x|^{2\gamma} e^{-|x|^2} |y|^{q-2}\\\ge C_a \tilde{M} |y|^q.  
\end{displaymath}
where $C_{\tilde M} \coloneqq -\tilde M q (q-2+\ell)>0$ and $C_a\coloneqq (2m - 4a^2) e^{-a^2}$. 
Recalling \eqref{eq:d}, then $d(z)^{\gamma+1}\ge |y|$ for all $z\in\R^N$. We know by \eqref{eqcomplem}-\eqref{condw1'} that
$$-\Delta_\gamma u \le  c_1 \sigma^{-\eta}  d(z)^{\eta(N_\gamma-2)-\delta-2\gamma} \le c_1 \sigma^{-\eta} |y|^{\frac{\eta(N_\gamma-2) -\delta-2\gamma}{\gamma+1}},$$
 since $\eta(N_\gamma-2) - \delta-2\gamma \le 0$ by \eqref{condw1'}. So $-\Delta_\gamma u \le -\Delta_\gamma \tilde \Psi$ for every $z\in B_\gamma^c(0,R)\cap \Sigma_a$,
enlarging $\tilde{M}$ if necessary, by requiring
$$q\ge {\frac{\eta(N_\gamma-2) -\delta-2\gamma}{\gamma+1}} $$
to be satisfied, or equivalently
$$\delta\ge \eta(N_\gamma-2) - 2\gamma - q(\gamma+1),$$
which holds since $q>2-\ell$ and $\delta>\delta_2$.

Moreover, choosing $\tilde M$ large enough such that $\tilde\Psi(R)>u(R)$ in $B_\gamma^c(0,R)\cap \Sigma_a$, we apply again Lemma \ref{weakcomp_ext} with $\Sigma_a$ instead of $\R^N$ and find
\begin{displaymath}
    u(z)\le \tilde\Psi(z) \quad\hbox{for all $z\in B_\gamma^c(0,R)\cap \Sigma_a$}.
\end{displaymath}
Since $\tilde\Psi(z) \to 0$ as $d(z)\to+\infty$, then \eqref{decayd} is proved also in $\Sigma_a$.
\end{proof}

\begin{rmk}
\label{rmk:decay}
We have proved the decay by constructing appropriate barrier functions and then  applying the weak comparison principle for exterior domains, Lemma \ref{weakcomp_ext}. Since the Grushin operator is not invariant under the action of $O(N)$, standard radial functions as power of the distance $d(\cdot)$ are not suitable candidates, in particular near the degenerate set $\Sigma$. We have constructed some barrier functions adapted to the geometry of the problem; this choice has forced us to impose $\ell>4$ and condition $\eqref{condw1'}$, which is slightly stronger then \ref{condw1}. We conjecture that the restrictive condition $\ell>4$ is only of technical nature.
\end{rmk}

\subsection{A truncated and frozen problem} \label{trfr} 
We now make problem \eqref{probnotdecay} variational by truncating and freezing it.

Let $\underline u_{\lambda_1}\in C^{0,\tau}_\loc(\R^N)$  be the solution to
\begin{displaymath}
\left\{ \begin{alignedat}{2}
-\Delta_\gamma u &=\lambda_1 w_1(z) u^{-\eta} \quad &&\mbox{in} \;\; \R^N, \\
u &> 0 \quad &&\mbox{in} \;\;\R^N, 
\end{alignedat}
\right.
\end{displaymath}
whose existence and uniqueness are guaranteed by Lemma \ref{subsol}, replacing $w_1$ with $\lambda_1 w_1$. According to linearity of the Grushin operator and the singularity $u^{-\eta}$, we have 
\begin{equation}\label{usub}
\underline u_{\lambda_1} = \lambda_1^{\frac{1}{1+\eta}}\underline{u},
\end{equation}
where $\underline{u}$ is a solution of \eqref{subprob}. For any fixed $v\in D^\gamma_0(\R^N)$, set
\begin{equation}
\label{a}
a(z,s)\coloneqq\lambda_1 w_1(z)\max\{s,\underline u_{\lambda_1}(z)\}^{-\eta} + \lambda_2 w_2(z)  | \nabla_\gamma v(z) |^{r-1} \quad \forall (z,s)\in \R^N\times\R \quad
\end{equation}
and consider the `truncated and frozen' problem
\begin{equation}
\label{varprob}
\tag{${\rm \hat{P}}$}
-\Delta_\gamma u = a(z,u) + u_+^{2^*_\gamma-1} \quad \mbox{in} \;\; \R^N.
\end{equation}
Problem \eqref{varprob} has a variational structure: its energy functional $J\colon D_0^\gamma(\R^N)\to\R$ is defined by
\begin{displaymath}
 J(u) \coloneqq \frac{1}{2} \|\nabla_\gamma u\|_2^2 -\int_{\R^N} A(z,u) \dz - \frac{1}{2^*_\gamma} \|u_+\|_{2^*_\gamma}^{2^*_\gamma},
\end{displaymath}
where $A(z,s)\coloneqq\int_0^s a(z,t) \dt$. 
It is standard to see that $J$ is well-defined and of class $C^1$, with
\begin{displaymath}
    \langle J'(u),\varphi \rangle = \int_{\R^N}\nabla_\gamma u\nabla_\gamma \varphi \dz - \int_{\R^N} a(z,u)\varphi \dz - \int_{\R^N}u_+^{2^*_\gamma-1}\varphi \dz.
\end{displaymath}
Given a sequence $(v_n)\subseteq D_0^\gamma(\R^N)$ we define 
\begin{equation}\label{defapicc}
a_n(z,s)\coloneqq\lambda_1 w_1(z)\max\{s,\underline u_{\lambda_1}(z)\}^{-\eta} + \lambda_2 w_2(z) | \nabla_\gamma v_n(z) |^{r-1} \quad \forall (z,s)\in \R^N\times\R,
\end{equation}
\begin{displaymath}
    A_n(z,s)\coloneqq\int_0^s a_n(z,t) \dt,
\end{displaymath}
and the functional
\begin{displaymath}
    J_n(u) \coloneqq \frac{1}{2} \|\nabla_\gamma u\|_2^2 -\int_{\R^N} A_n(z,u) \dz - \frac{1}{2^*_\gamma} \|u_+\|_{2^*_\gamma}^{2^*_\gamma}.
\end{displaymath}
From the definition of $a$ we get the following relations, which hold for all $(z,s)\in\R^N\times\R$:
$$A(z,s) \ge -\left(\lambda_1 w_1(z) \underline u_{\lambda_1}(z)^{-\eta}+\lambda_2 w_2(z) |\nabla_\gamma v(z)|^{r-1}\right)|s|,$$ 
\begin{equation}
\label{Aest}
A(z,s) \leq \frac{\lambda_1}{1-\eta}w_1(z) |s|^{1-\eta} + \left(\lambda_1 w_1(z)\underline u_{\lambda_1}(z)^{-\eta}+\lambda_2 w_2(z)|\nabla_\gamma v(z)|^{r-1}\right)|s|.
\end{equation}
For future reference we introduce the parameters $\zeta$, $\theta\in(1,+\infty)$ defined by the equations
\begin{equation}
    \label{def:teta}
    \frac{1}{\zeta}+\frac{1-\eta}{2^*_\gamma}=1, \quad \frac{1}{\theta}+\frac{r-1}{2}+\frac{1}{2^*_\gamma}=1. 
\end{equation}
First, we prove two general results concerning, in particular, Palais-Smale sequences associated with the functionals $J_n$. Lemma \ref{enestlemma} provides an energy estimate, while Lemma \ref{ccplemma} detects a ``critical'' energy level under which compactness is recovered.

\begin{lemma}[Energy estimate]
\label{enestlemma}
Suppose $\lambda_1\in(0,1]$ and $\lambda_2\in [0,1]$. Let $c\in\R$, $L>0$, and $(u_n),(v_n)\subseteq D_0^\gamma(\R^N)$ such that
\begin{equation}
\label{generalhyps}
\begin{aligned}
&\limsup_{n\to\infty} J_n(u_n)\leq c, \\
&\limsup_{n\to\infty} \|\nabla_\gamma v_n\|_2 \leq L,\\
&\lim_{n\to\infty} J_n'(u_n) = 0 \quad \mbox{in} \;\; D^{-1,2}_\gamma(\R^N).
\end{aligned}
\end{equation}
Then there exists $\hat{C}=\hat{C}(N,w_1, w_2,\eta,r,\gamma)>0$ such that
\begin{displaymath}
\limsup_{n\to\infty} \|\nabla_\gamma u_n\|_2^2 \leq 2N_\gamma\left[\hat C\max\{\lambda_1,\lambda_2\}^{\frac{2}{1+\eta}} \left(1+L^{2(r-1)}\right)+c \right].
\end{displaymath}
\end{lemma}
\begin{proof}
 According to \eqref{generalhyps} we have
\begin{equation}
\label{enest:start}
\begin{split}
c+o(1)(1+\|\nabla_\gamma u_n\|_2) &\geq J_n(u_n)-\frac{1}{2^*_\gamma} \langle J'_n(u_n),u_n \rangle \\
&=\frac{1}{N_\gamma}\|\nabla_\gamma u_n\|_2^2-\int_{\R^N} A_n(z,u_n) \dz + \frac{1}{2^*_\gamma} \int_{\R^N} a_n(z,u_n)u_n \dz.
\end{split}
\end{equation}
By means of \eqref{usub} and \eqref{Aest}, as well as H\"{o}lder's and Sobolev's inequalities, we get
\begin{equation}\label{enest:A}\begin{aligned}
\int_{\R^N}& A_n(z,u_n)\dz \\
&\leq \frac{\lambda_1}{1-\eta} \int_{\R^N} w_1|u_n|^{1-\eta}\dz + \lambda_1 \int_{\R^N} w_1\underline u_{\lambda_1}^{-\eta}|u_n|\dz +\lambda_2\int_{\R^N}w_2|\nabla_\gamma v_n|^{r-1}|u_n| \dz \\
&\leq \frac{\lambda_1}{1-\eta} \|w_1\|_\zeta \|u_n\|_{2^*_\gamma}^{1-\eta} + \lambda_1^{\frac{1}{1+\eta}}\|w_1\underline{u}^{-\eta}\|_{(2^*_\gamma)'}\|u_n\|_{2^*_\gamma} +\lambda_2\|w_2\|_\theta \|\nabla_\gamma v_n\|_2^{r-1} \|u_n\|_{2^*_\gamma}.  
\end{aligned}\end{equation}
Moreover, by \eqref{defapicc} and \eqref{generalhyps}
\begin{equation}
\label{enest:a}\begin{aligned}
\int_{\R^N} a_n(z,u_n)u_n \dz &\geq - \int_{\R^N} a_n(z,u_n)|u_n| \dz \\
&\geq - \int_{\R^N} \left(\lambda_1 w_1\underline u_{\lambda_1}^{-\eta}+\lambda_2 w_2 |\nabla_\gamma v_n|^{r-1}\right)|u_n| \dz \\
&\geq -  S^{-\frac{1}{2}} \left(\lambda_1^{\frac{1}{1+\eta}}\|w_1\underline{u}^{-\eta}\|_{(2^*_\gamma)'} + \lambda_2\|w_2\|_\theta L^{r-1}+o(1)\right) \|\nabla_\gamma u_n\|_2.
\end{aligned}\end{equation}
Inserting this latter inequalities in \eqref{enest:start} and by using Young's inequality we have
$$\begin{aligned}
&c+o(1)(1+\|\nabla_\gamma u_n\|_2) \\&\geq \frac{1}{N_\gamma}\|\nabla_\gamma u_n\|_2^2 - \frac{\lambda_1}{1-\eta}S^{-\frac{1-\eta}{2}} \|w_1\|_\zeta \|\nabla_\gamma u_n\|_2^{1-\eta}\\
&\quad- S^{-\frac{1}{2}} \left(1+\frac{1}{2^*_\gamma}\right) \left(\lambda_1^{\frac{1}{1+\eta}}\|w_1\underline{u}^{-\eta}\|_{(2^*_\gamma)'} + \lambda_2\|w_2\|_\theta L^{r-1}+o(1)\right) \|\nabla_\gamma u_n\|_2\\
&\geq\frac{1}{N_\gamma}\|\nabla_\gamma u_n\|_2^2 - C_\eps \lambda_1^{\frac{2}{1+\eta}} -\eps\|\nabla_\gamma u_n\|_2^{2} - C_\eps \left(\lambda_1^{\frac{2}{1+\eta}} + \lambda_2^2 L^{2(r-1)}+o(1)\right)  -\eps \|\nabla_\gamma u_n\|_2^2.
\end{aligned}$$
Now, re-absorbing the terms $\|\nabla_\gamma u_n\|_2$ on the right-hand side, we get
$$\left(\frac{1}{N_\gamma}-2\eps\right)\|\nabla_\gamma u_n\|_2^2 \leq C_\eps\left(\lambda_1^{\frac{2}{1+\eta}}+\lambda_2^2L^{2(r-1)}\right)+c+o(1),$$
If $\lambda_1,\lambda_2\le 1$, then
\begin{equation}
\label{enest:mid}
  \left(\frac{1}{N_\gamma}-2\eps\right)\|\nabla_\gamma u_n\|_2^2 \leq\begin{cases}
C_\eps\lambda_1^{\frac{2}{1+\eta}}\left(1+L^{2(r-1)}\right)+c+o(1), \qquad&\text{if }\lambda_1\ge \lambda_2,\\
C_\eps\lambda_2^{\frac{2}{1+\eta}}\left(1+L^{2(r-1)}\right)+c+o(1),\qquad&\text{if }\lambda_1\le \lambda_2,\end{cases}   
\end{equation}
for a suitable $C_\eps=C_\eps(\eps,N,w_1, w_2,\eta,r,\gamma)$ (since $S$ is a function of only $N_\gamma$ and $\underline{u}$ depends uniquely on $w_1,\eta,\gamma$).
Choosing $\eps=\frac{1}{4N_\gamma}$ and setting $\hat{C}\coloneqq C_\eps$, \eqref{enest:mid} becomes
\begin{equation}\label{enest:end}
\frac{1}{2N_\gamma}\|\nabla_\gamma u_n\|_2^2 \leq \hat C\max\{\lambda_1,\lambda_2\}^{\frac{2}{1+\eta}} \left(1+L^{2(r-1)}\right)+c+o(1).
\end{equation}
We conclude by taking the upper limit as $n \to +\infty$.
\end{proof}
\begin{lemma}[Concentration-compactness]
\label{ccplemma}
With the same hypotheses of Lemma \ref{enestlemma}, if
\begin{equation}
\label{hatc}
c<\frac{S^{N_\gamma/2}}{N_\gamma}-\hat C\max\{\lambda_1,\lambda_2\}^{\frac{2}{1+\eta}} \left( L^{2(r-1)}+1\right)\eqqcolon \hat{c},
\end{equation}
then there exists $u\in  D_0^\gamma(\R^N)$ such that, up to sub-sequences, $u_n\rightharpoonup u$ in $ D_0^\gamma(\R^N)$ and
\begin{equation}
\label{strongconv}
(u_n)_+\to u_+ \quad \mbox{in} \;\; L^{2^*_\gamma}(\R^N).
\end{equation}
In particular, $((u_n)_+)$ is uniformly equi-integrable in $L^{2^*_\gamma}(\R^N)$, i.e.,
\begin{equation}
\label{equiint}
\int_\Omega (u_n)_+^{2^*_\gamma} \dz \to 0 \quad \mbox{as} \;\; |\Omega|\to 0, \quad \mbox{uniformly in} \;\; n\in\N.
\end{equation}
\end{lemma}
\begin{proof}
By Lemma \ref{enestlemma} we deduce that $(u_n)$ is bounded in $ D_0^\gamma(\R^N)$. Hence there exists $u\in D_0^\gamma(\R^N)$ such that, up to sub-sequences, $u_n \rightharpoonup u$ in $ D_0^\gamma(\R^N)$, which implies $(u_n)_+ \rightharpoonup u_+$ in $ D_0^\gamma(\R^N)$ by Lemma \ref{posparts}. In particular, $\nabla_\gamma (u_n)_+ \rightharpoonup \nabla_\gamma u_+$ in $L^2(\R^N)$, ensuring that the sequence of measures $(|\nabla_\gamma (u_n)_+|^2)$ is bounded. Thus, $|\nabla_\gamma (u_n)_+|^2 \rightharpoonup \mu$ for some bounded positive measure $\mu$; analogously, $(u_n)_+^{2^*_\gamma} \rightharpoonup \nu$ for an opportune bounded positive measure $\nu$.
According to Lemmas \ref{thm:concentration}--\ref{bennaoum}, applied to $((u_n)_+)$ in place of $(u_n)$, there exist some at most countable set $\A$, a family of points $(z_j)_{j\in \A} \subseteq \R^N$, and two families of numbers $(\mu_j)_{j\in \A},(\nu_j)_{j\in \A}\subseteq (0,+\infty)$ satisfying
\begin{equation}
\label{PS:measprops}
\begin{split}
\nu=u_+^{2^*_\gamma}+\sum_{j\in \A} \nu_j \delta_{z_j}, &\quad \mu\geq |\nabla_\gamma u_+|^2+\sum_{j\in \A} \mu_j \delta_{z_j}, \\
\limsup_{n\to\infty} \int_{\R^N} (u_n)_+^{2^*_\gamma} \dz = \int_{\R^N} \, {\rm d}\nu + \nu_\infty, &\quad \limsup_{n\to\infty} \int_{\R^N} |\nabla_\gamma (u_n)_+|^2 \dz = \int_{\R^N} \, {\rm d}\mu+ \mu_\infty,
\end{split}
\end{equation}
and
\begin{equation}
\label{PS:measbounds}
S\nu_j^{2/2^*_\gamma} \leq \mu_j \quad \mbox{for all} \;\; j\in \A, \quad S\nu_\infty^{2/2^*_\gamma}\leq \mu_\infty,
\end{equation}
being
\begin{displaymath}
\nu_\infty \coloneqq  \lim_{R\to+\infty} \limsup_{n\to\infty} \int_{B_\gamma^c(0,R)} (u_n)_+^{2^*_\gamma} \dz, \quad \mu_\infty \coloneqq  \lim_{R\to\infty} \limsup_{n\to\infty} \int_{B_\gamma^c(0,R)} |\nabla_\gamma (u_n)_+|^2 \dz.
\end{displaymath}
We claim that $\A=\emptyset$. By contradiction, let $j\in \A$ and $\psi\in C^\infty_c(\R^N)$ be a standard cut-off function fulfilling $\psi\equiv 1$ in $\overline{B_\gamma(0,1/2)}$, $\psi\equiv 0$ in $B_\gamma^c(0,1)$, and $0\leq \psi \leq 1$ in $\R^N$. For each $\eps\in(0,1)$ set
$$\psi_\eps(z)\coloneqq \psi\left(\frac{z - z_j}{\eps}\right).$$
By hypothesis \eqref{generalhyps} we have $\langle J_n'(u_n), (u_n)_+ \psi_\eps \rangle\to 0$, that is,
\begin{equation}
\label{PS:conctest}
\begin{split}
&\int_{\R^N} |\nabla_\gamma (u_n)_+|^2 \psi_\eps \dz + \int_{\R^N} (u_n)_+ \nabla_\gamma u_n \nabla_\gamma \psi_\eps \dz \\
&= \int_{\R^N} a_n(z,u_n)(u_n)_+\psi_\eps \dz +  \int_{\R^N} (u_n)_+^{2^*_\gamma} \psi_\eps \dz + o(1).
\end{split}
\end{equation}
Since $ D_0^\gamma(\R^N)\compact L^2(B_\gamma(z_j,\eps))$, we have (up to sub-sequences) $u_n\to u$ both in $L^2(B_\gamma(z_j,\eps))$ and in $\R^N$. Reasoning as for \eqref{enest:a}, besides recalling that $(|\nabla_\gamma u_n|)$ is bounded in $L^2(\R^N)$, one has
\begin{equation}
\label{PS:conctest1}
\begin{aligned}
\lim_{\eps \to 0}& \lim_{n\to\infty} \left| \int_{\R^N} a_n(z,u_n)(u_n)_+\psi_\eps\dz \right| \\
&\leq  \lim_{\eps \to 0} \lim_{n\to\infty} \int_{B_\gamma(z_j,\eps)} \left(\lambda_1 w_1 \underline u_{\lambda_1}^{-\eta} + \lambda_2 w_2 |\nabla_\gamma v(z) |^{r-1} \right)|u_n| \dz \\
&\leq \lambda_1^{-\frac{\eta}{1+\eta}} S^{-\frac{1}{2}} \left(\sup_{n\in\N} \|\nabla_\gamma u_n\|_2\right)\cdot\\& \qquad \qquad\cdot\lim_{\eps\to 0}\left(\lambda_1\|w_1\underline{u}^{-\eta}\|_{L^{(2^*_\gamma)'}(B_\gamma(z_j,\eps))} + \lambda_2 \|w_2\|_{L^\theta(B_\gamma(z_j,\eps))}L^{r-1}\right) = 0,
\end{aligned}
\end{equation}
where we used \ref{hypf}, \eqref{generalhyps} and $\lim_{\eps\to 0} \|w_1 \underline u^{-\eta}\|_{L^{(2^*_\gamma)'}(B_\gamma(z_j,\eps))} = 0$, since $w_1 \underline u^{-\eta}\in L^\infty(\R^N)$ by Lemma \ref{subsol}.

Since $\nabla_\gamma \psi_\eps$ is bounded and compactly supported in $B_\gamma(z_j,\eps)$, we deduce $u_n\nabla_\gamma \psi_\eps \to u \nabla_\gamma \psi_\eps$ in $L^2(\R^N)$ as $n\to\infty$. By H\"{o}lder's inequality and the change of variable  $\xi\mapsto (z-z_j)/\eps$ we infer
\begin{equation}
\label{PS:conctest2}
\begin{split}
\lim_{\eps \to 0}& \lim_{n\to\infty} \left|\int_{\R^N} u_n \nabla_\gamma u_n \nabla_\gamma \psi_\eps \dz\right| \leq \lim_{\eps \to 0} \lim_{n\to\infty} \|\nabla_\gamma u_n\|_2 \|u_n\nabla_\gamma \psi_\eps\|_2 \\
&\leq\sup_{n\in\N} \|\nabla_\gamma u_n\|_2\cdot \lim_{\eps \to 0} \|u\nabla_\gamma \psi_\eps\|_2 \leq \sup_{n\in\N} \|\nabla_\gamma u_n\|_2\cdot \|\nabla_\gamma \psi\|_{N_\gamma} \lim_{\eps \to 0} \|u\|_{L^{2^*_\gamma}(B_\gamma(z_j,\eps))} = 0.
\end{split}
\end{equation}
Passing to the limit in \eqref{PS:conctest} via \eqref{PS:conctest1}--\eqref{PS:conctest2} we get
\begin{displaymath}
\lim_{\eps \to 0} \lim_{n\to\infty} \int_{\R^N} |\nabla_\gamma (u_n)_+|^2 \psi_\eps \dz = \lim_{\eps \to 0} \lim_{n\to\infty} \int_{\R^N} (u_n)_+^{2^*_\gamma} \psi_\eps \dz.
\end{displaymath}
Recalling Remark \ref{nujmuj}, we obtain
\begin{equation}
\label{PS:proportional}
\mu_j = \nu_j.
\end{equation}
Combining \eqref{PS:proportional} with \eqref{PS:measbounds} entails $S\nu_j^{2/2^*_\gamma-1}\leq1$, whence
\begin{equation}
\label{PS:lowerbounds}
\mu_j=\nu_j \geq S^{N_\gamma/2}.
\end{equation}
Now using \eqref{PS:lowerbounds} and arguing as in \eqref{enest:end} we have 
\begin{equation}
\label{PS:concfinal}\begin{aligned}
c+o(1)&\ge J_n(u_n)-\frac{1}{2^*_\gamma}\langle J_n'(u_n),u_n\rangle
\\&
\ge\frac{1}{N_\gamma}\|\nabla_\gamma u_n\|_2^2+\left[\frac{1}{2^*_\gamma}\int_{\R^N}a_n(z,u_n)u_n \dz-\int_{\R^N}A_n(z,u_n) \dz\right]\\
&\ge\frac{S^{N_\gamma/2}}{N_\gamma}+\frac{1}{N_\gamma}\|\nabla_\gamma u\|_2^2+\left[\frac{1}{2^*_\gamma}\int_{\R^N} a_n(z,u)u \dz-\int_{\R^N} A_n (z,u) \dz\right]+o(1)\\
&\ge \frac{S^{N_\gamma/2}}{N_\gamma}+\frac{1}{2N_\gamma}\|\nabla_\gamma u\|_2^2-\hat C\max\{\lambda_1,\lambda_2\}^{\frac{2}{1+\eta}} \left( L^{2(r-1)}+1\right)+o(1)\\
&\ge \frac{S^{N_\gamma/2}}{N_\gamma}-\hat C\max\{\lambda_1,\lambda_2\}^{\frac{2}{1+\eta}} \left( L^{2(r-1)}+1\right)+o(1),
\end{aligned}\end{equation}
contradicting \eqref{hatc} as $n\to\infty$. This forces $\A=\emptyset$.

A similar argument proves that concentration cannot occur at infinity, that is, $\nu_\infty=\mu_\infty=0$: indeed, using a cut-off function $\psi_R\in C^\infty(\R^N)$ such that $\psi_R \equiv 0$ in $\overline{B_\gamma(0,R)}$, $\psi_R \equiv 1$ in $B_\gamma^c(0,2R)$, and $0\leq \psi_R\leq 1$ in $\R^N$,  one arrives at $\nu_\infty = \mu_\infty$, and then the conclusion follows as in \eqref{PS:concfinal}.

According to \eqref{PS:measprops}, besides $\A=\emptyset$ and $\nu_\infty=0$, we obtain
\begin{displaymath}
\limsup_{n\to\infty} \int_{\R^N} (u_n)_+^{2^*_\gamma} \dz = \int_{\R^N} u_+^{2^*_\gamma} \dz,
\end{displaymath}
that is \eqref{strongconv}, by uniform convexity (see \cite[Proposition 3.32]{B}); in turn, \eqref{strongconv} forces \eqref{equiint} by Vitali's convergence theorem (see, e.g., \cite[Corollary 4.5.5]{Bo}).
\end{proof}
Problem \eqref{varprob} exhibits a double lack of compactness, due to the setting in $\R^N$ and the presence of a reaction term with critical growth. The aim of the next lemma is to recover compactness on the energy levels that lie under the critical level $\hat{c}$ defined in \eqref{hatc}.

\begin{lemma}[Palais-Smale condition]
\label{PS}
Let  $\lambda_1\in(0,1]$, $\lambda_2\in [0,1]$, $L>0$, and $v\in D_0^\gamma(\R^N)$ be such that $\|\nabla_\gamma v\|_2\leq L$. Then $J$ satisfies the ${\rm(PS)}_c$ condition for all $c\in\R$ fulfilling \eqref{hatc}.
\end{lemma}
\begin{proof}
Take any $c\in\R$ as in \eqref{hatc} and consider an arbitrary ${\rm (PS)}_c$ sequence $(u_n)$ associated to the functional $J$. Applying Lemma \ref{ccplemma} with $v_n\equiv v$ and $J_n\equiv J$ entails $u_n\rightharpoonup u$ in $ D_0^\gamma(\R^N)$ and $(u_n)_+\to u_+$ in $L^{2^*_\gamma}(\R^N)$. Let us evaluate
\begin{equation}
\label{PS:S+start}
\begin{split}
\langle J'(u_n),u_n-u \rangle &= \int_{\R^N}  \nabla_\gamma u_n\nabla_\gamma (u_n-u)\dz - \int_{\R^N} a(z,u_n)(u_n-u) \dz \\
&\quad- \int_{\R^N} (u_n)_+^{2^*_\gamma-1}(u_n-u) \dz.
\end{split}
\end{equation}
Since $ D_0^\gamma(\R^N)\hookrightarrow L^{2^*_\gamma}(\R^N)$, we have $u_n \rightharpoonup u$ in $L^{2^*_\gamma}(\R^N)$ and $u_n\to u$ in $\R^N$, up to sub-sequences (see, e.g., \cite[p.3044]{GG1}).

Thus, observing that \eqref{strongconv} forces $(u_n)_+^{2^*_\gamma-1}\to u_+^{2^*_\gamma-1}$ in $L^{(2^*_\gamma)'}(\R^N)$, we get
\begin{equation}
\label{PS:S+crit}
\lim_{n\to\infty} \int_{\R^N} (u_n)_+^{2^*_\gamma-1}(u_n-u) \dz = 0. 
\end{equation}
Reasoning as for \eqref{enest:a} one has
$$ \|\lambda_1w _1 \underline u_{\lambda_1}^{-\eta}+ \lambda_2 w_2|\nabla_\gamma v|^{r-1}\|_{(2^*_\gamma)'} \leq \lambda_1^{\frac{1}{1+\eta}}\|w_1\underline{u}^{-\eta}\|_{(2^*_\gamma)'}+\lambda_2\|w_2\|_\theta L^{r-1} < +\infty, $$
so the linear functional $\psi\mapsto \int_{\R^N} \left(\lambda_1 w_1 \underline u_{\lambda_1}^{-\eta}+\lambda_2 w_2 |\nabla_\gamma v|^{r-1} \right)\psi \dz$ is continuous in $L^{2^*_\gamma}(\R^N)$.
Moreover, $u_n \rightharpoonup u$ in $L^{2^*_\gamma}(\R^N)$ guarantees that $|u_n-u|\rightharpoonup 0$ in $L^{2^*_\gamma}(\R^N)$. By the definition of $a_n$ in \eqref{defapicc}
\begin{equation}
\label{PS:S+sing}
\lim_{n\to\infty} \left|\int_{\R^N} a(z,u_n)(u_n-u) \dz \right| \leq \lim_{n\to\infty} \int \left(\lambda_1 w_1\underline u_{\lambda_1}^{-\eta}+\lambda_2 w_2 |\nabla_\gamma v|^{r-1}\right)|u_n-u| \dz = 0.
\end{equation}
Using \eqref{PS:S+crit}--\eqref{PS:S+sing} and recalling that $\langle J'(u_n),u_n-u \rangle \to 0$, due to the fact that $(u_n)$ is a ${\rm (PS)}_c$ sequence, by \eqref{PS:S+start} we conclude
$$
\lim_{n\to\infty} \int_{\R^N}  \nabla_\gamma u_n\nabla_\gamma (u_n-u)\dz = 0.
$$
Then the ${\rm (S_+)}$ property of the Grushin operator, Lemma \ref{lemma:S+}, yields $u_n \to u$ in $ D_0^\gamma(\R^N)$.
\end{proof}

The next two lemmas are devoted to verify the mountain pass geometry for the functional $J$ and ensure the mountain pass level $c_M$ (see Theorem \ref{mountainpass}) lies below the critical Palais-Smale level $\hat{c}$ (see \eqref{hatc}), provided $\lambda_1$ and $\lambda_2$ small enough.

\begin{lemma}[Mountain pass geometry]
\label{mountainpassgeometry}
Let $L>0$ and $v\in D_0^\gamma(\R^N)$ be such that $\|\nabla_\gamma v\|_2\leq L$. Then there exists $\tilde{C}=\tilde{C}(N,w_1,w_2,\eta,r,\gamma)>1$ such that, for all $\lambda_1>0,\lambda_2\ge 0$ satisfying 
\begin{equation}
\label{smallness}
\lambda_2\le \lambda_1 < \Lambda_1\coloneqq \left[\frac{S^{N_\gamma/2}}{\tilde{C}N_\gamma(S^{N_\gamma/4}+1)(1+L^{r-1})}\right]^{1+\eta},
\end{equation}
the functional $J$ satisfies the mountain pass geometry. More precisely, 
one has
$$ J(\hat{u})<J(0)=0<\inf_{\partial B_\gamma(0,\rho)} J \quad \mbox{and} \quad \|\nabla_\gamma \hat{u}\|_2 > \rho, \quad \mbox{being} \;\; \rho\coloneqq S^{N_\gamma/4} ,$$
for any $\hat{u}\in D_0^\gamma(\R^N)$ such that $\hat{u}_+ \not\equiv 0$ and $ \|\nabla_\gamma \hat{u}\|_2$ is sufficiently large.
\end{lemma}

\begin{proof}
Fix $\tilde{C}=\tilde{C}(N,w_1,w_2,\eta,r,\gamma)>1$  such that
 $$ \frac{1}{1-\eta}S^{-\frac{1-\eta}{2}}\|w_1\|_\zeta + S^{-\frac{1}{2}} \left(\|w_1\underline{u}^{-\eta}\|_{(2^*_\gamma)'}+\|w_2\|_\theta L^{r-1}\right) \leq \tilde{C}(1+L^{r-1}). $$
Let $\lambda_1, \lambda_2$ fulfill \eqref{smallness} and pick any $t>0$. Reasoning as in \eqref{enest:A}, besides recalling the choice of $\tilde{C}$, we get
\begin{displaymath}
\begin{aligned}
&\inf_{\partial B_\gamma(0,t)} J 
\geq \inf_{u\in \partial B_\gamma(0,t)} \left[ \frac{1}{2}\|\nabla_\gamma u\|_2^2 - \frac{\lambda_1}{1-\eta}\int_{\R^N} w_1 |u|^{1-\eta} \dz \right.\\ &\left. \qquad -    \int_{\R^N} \left(\lambda_1 w_1\underline u_{\lambda_1}^{-\eta}+\lambda_2 w_2|\nabla_\gamma v|^{r-1}\right)|u| \dz - \frac{1}{2^*_\gamma} \|u\|_{2^*_\gamma}^{2^*_\gamma} \right] \\
&\geq \inf_{u\in \partial B_\gamma(0,t)} \left[ \frac{1}{2}\|\nabla_\gamma u\|_2^2 - \frac{\lambda_1}{1-\eta}S^{-\frac{1-\eta}{2}}\|w_1\|_\zeta \|\nabla_\gamma u\|_2^{1-\eta} \right. \\
&\left. \qquad -   S^{-\frac{1}{2}} \left(\lambda_1^{\frac{1}{1+\eta}}\|w_1\underline{u}^{-\eta}\|_{(2^*_\gamma)'}+\lambda_2\|w_2\|_\theta L^{r-1}\right)\|\nabla_\gamma u\|_2 - \frac{1}{2^*_\gamma}S^{-\frac{2^*_\gamma}{2}} \|\nabla_\gamma u\|_2^{2^*_\gamma} \right] \\
&= \frac{1}{2}t^2 - \frac{\lambda_1}{1-\eta}S^{-\frac{1-\eta}{2}}\|w_1\|_\zeta t^{1-\eta} - S^{-\frac{1}{2}} \left(\lambda_1^{\frac{1}{1+\eta}}\|w_1\underline{u}^{-\eta}\|_{(2^*_\gamma)'}+\lambda_2\|w_2\|_\theta L^{r-1}\right)t  - \frac{1}{2^*_\gamma}S^{-\frac{2^*_\gamma}{2}} t^{2^*_\gamma} \\
&\geq \frac{1}{2}t^2 - \frac{\lambda_1}{1-\eta}S^{-\frac{1-\eta}{2}}\|w_1\|_\zeta (t+1) \\
& \qquad-   S^{-\frac{1}{2}} \left(\lambda_1^{\frac{1}{1+\eta}}\|w_1\underline{u}^{-\eta}\|_{(2^*_\gamma)'}+\lambda_2\|w_2\|_\theta L^{r-1}\right)(t+1) - \frac{1}{2^*_\gamma}S^{-\frac{2^*_\gamma}{2}} t^{2^*_\gamma} \\
&\geq \frac{1}{2}t^2 - (t+1)\left[\frac{\lambda_1}{1-\eta}S^{-\frac{1-\eta}{2}}\|w_1\|_\zeta  +  S^{-\frac{1}{2}} \left(\lambda_1^{\frac{1}{1+\eta}}\|w_1\underline{u}^{-\eta}\|_{(2^*_\gamma)'}+\lambda_2\|w_2\|_\theta L^{r-1}\right)\right]  - \frac{1}{2^*_\gamma}S^{-\frac{2^*_\gamma}{2}} t^{2^*_\gamma} \\
&\geq \frac{1}{2}t^2 - \tilde{C}\lambda_1^{\frac{1}{1+\eta}}(1+L^{r-1})(t+1) - \frac{1}{2^*_\gamma}S^{-\frac{2^*_\gamma}{2}} t^{2^*_\gamma},
\end{aligned}
\end{displaymath}
since $\|\nabla_\gamma u\|_2=t$ whenever $u\in\partial B_\gamma(0,t)$. Let us consider the real-valued function $g\colon(0,+\infty)\to\R$ defined as
\begin{displaymath}
g(t) \coloneqq \frac{1}{2} t^2 -\eps (t+1) -\frac{1}{2^*_\gamma}S^{-\frac{2^*_\gamma}{2}} t^{2^*_\gamma},
\end{displaymath}
where $\eps\coloneqq \tilde{C}\lambda_1^{\frac{1}{1+\eta}}(1+L^{r-1})$. Condition \eqref{smallness} forces $\eps<\frac{S^{N_\gamma/2}}{N_\gamma(S^{N_\gamma/4}+1)}$, so that
$$g(S^{N_\gamma/4})=\frac{S^{N_\gamma/2}}{N_\gamma}-\eps (S^{N_\gamma/4}+1)>0.$$
Setting $\rho\coloneqq S^{N_\gamma/4}$ we get
$$\inf_{\partial B_\gamma(0,\rho)} J \ge  g(\rho)>0.$$
Given any $u\in D_0^\gamma(\R^N)$ such that $u_+\not\equiv 0$, we have $J(tu)\to-\infty$ as $t\to+\infty$: indeed, according to \eqref{Aest} and the fact that $\|u_+\|_{2^*_\gamma}>0$,
\begin{align*}
    \limsup_{t\to+\infty} J(tu) &\leq \lim_{t\to+\infty} \left[ \frac{t^2}{2}\|\nabla_\gamma u\|_2^2 + t\int_{\R^N} (\lambda_1 w_1\underline u_{\lambda_1}^{-\eta}+\lambda_2 w_2|\nabla_\gamma v|^{r-1})|u| \dz -\frac{t^{2^*_\gamma}}{2^*_\gamma}\|u_+\|_{2^*_\gamma}^{2^*_\gamma} \right] \\ 
    &= -\infty.
\end{align*}
Hence, setting $\hat{u}\coloneqq tu$, we have $J(\hat{u})<0<\inf_{\partial B_\gamma(0,\rho)} J$ and $ \|\nabla_\gamma \hat{u}\|_2>\rho$, provided $t$ is sufficiently large.
\end{proof}

\begin{lemma}\label{talentihatc}
Let $L>0$ and $v\in D_0^\gamma(\R^N)$ be such that $\|\nabla_\gamma v\|_2\leq L$.  
There exists $\Lambda_2\in(0,1)$ such that for every $\lambda_1 > 0, \lambda_2\ge 0$, being $\lambda_2\le\lambda_1<\Lambda_2$ implies
\begin{displaymath}
\inf_{\phi\in\Phi} \sup_{t\in[0,1]} J(\phi(t))<c<\hat{c}, \quad \Phi\coloneqq \{\phi\in C^0([0,1]; D_0^\gamma(\R^N)): \, \phi(0)=0, \, \phi(1)=\hat{u}\}
\end{displaymath}
for a suitable $c=c(\lambda_1,\lambda_2,N,w_1, w_2,\eta,\gamma)>0$, where $\hat{c}$ is defined in \eqref{hatc} and $\hat{u}$ is the function that satisfies the minimizer for the embedding constant $S$, that is $\hat{u}\in D_0^\gamma(\R^N)$ and satisfies 
\begin{displaymath} \lVert \nabla_\gamma \hat{u} \rVert_2^2 = \inf_{\substack{
    u \in D_0^\gamma(\mathbb{R}^N) \\
    \|u\|_{2^*_\gamma} = 1
}} \lVert \nabla_\gamma u \rVert_2^2,
\end{displaymath}
see \cite[Theorem 3.4]{alves2024brezis}.
\end{lemma}
\begin{proof}
Firstly, we notice that $\hat{u}$ is a solution, up to a normalization, of the equation $-\Delta_\gamma \hat{u}=\hat{u}^{2_\gamma^*-1}$ in $\R^N$, so that $ \|\nabla_\gamma \hat{u}\|_2^2=\|\hat{u}\|_{2^*_\gamma}^{2^*_\gamma}=S^{N_\gamma/2}$, see \cite[Eq. (6.2)]{alves2024brezis}.
Next, we observe that the path $t \in [0,1] \mapsto t\hat{u}$ belongs to $\Phi$, so
\begin{displaymath}
\inf_{\phi\in\Phi} \sup_{t\in[0,1]} J(\phi(t)) \leq \sup_{t\in[0,1]} J(t\hat{u}) \leq \sup_{t\in[0,+\infty)} J(t\hat{u}).
\end{displaymath}
Accordingly, let us compute the maximizer $\overline{t}$ of the function $t\mapsto J(t\hat{u})$, being $t\geq 0$.
\begin{displaymath}
0 = \frac{{\rm d}}{{\rm d}t}{\Big\vert_{t=\overline{t}}} (J(t\hat{u})) = \langle J'(\overline{t}\hat{u}),\hat{u} \rangle = \overline{t} \|\nabla_\gamma \hat{u}\|_2^2 - \int_{\R^N} a(z,\overline{t}\hat{u})\hat{u} \dz -  \overline{t}^{2^*_\gamma-1} \|\hat{u}\|_{2^*_\gamma}^{2^*_\gamma},
\end{displaymath}
whence
\begin{equation}
\label{lambdacrit}
\overline{t}^{2^*_\gamma-1} \|\hat{u}\|_{2^*_\gamma}^{2^*_\gamma} = \overline{t} \|\nabla_\gamma \hat{u}\|_2^2 - \int_{\R^N} a(z,\overline{t}\hat{u})\hat{u} \dz.
\end{equation}
From \eqref{lambdacrit} we deduce 
$$\overline{t}^{2^*_\gamma-1} \|\hat{u}\|_{2^*_\gamma}^{2^*_\gamma} \le \overline{t} \|\nabla_\gamma \hat{u}\|_2^2=\overline{t}\|\hat{u}\|_{2^*_\gamma}^{2^*_\gamma},$$
forcing $\overline{t}\in[0,1]$.
Fix any $\Lambda_2\in(0,1)$ and define 
$$\Upsilon_1\coloneqq \left\{\lambda\in(0,\Lambda_2): \, \overline{t}\in \left[\frac{1}{2},1\right]\right\}, \qquad \Upsilon_2\coloneqq  \left\{\lambda\in(0,\Lambda_2): \, \overline{t}\in \left[0,\frac{1}{2}\right]\right\}.$$
Suppose $\lambda_1\in\Upsilon_1$. 
Exploiting the monotonicity of $a(z,\cdot)$ and
$$ \max_{t\in(0,+\infty)}\left(\frac{t^2}{2}- \frac{t^{2^*_\gamma}}{2^*_\gamma}\right)=\frac{1}{N_\gamma}, $$
we get
\begin{equation}\label{MPlevel:Jest}
\begin{split}
J(\overline{t}\hat{u}) &= \frac{\overline{t}^2}{2}\|\nabla\hat{u}\|_2^2 - \int_{\R^N} a(z,\overline{t}\hat{u}) \dz - \frac{\overline{t}^{2^*_\gamma}}{2^*_\gamma} \|\hat{u}\|_{2^*_\gamma}^{2^*_\gamma}\\
&=\left(\frac{\overline{t}^2}{2}- \frac{\overline{t}^{2^*_\gamma}}{2^*_\gamma} \right)\|\nabla\hat{u}\|_2^2 - \int_{\R^N} a(z,\overline{t}\hat{u}) \dz\le \frac{\|\nabla\hat{u}\|_2^2}{N_\gamma}-\int_{\R^N} A\left(z, \frac{\hat{u}}{2}\right) \dz.
\end{split}
\end{equation}
Let $\hat m\coloneqq \inf_{B(z_0,\varrho)}\hat u>0$, being $z_0,\varrho$ as in \ref{hypf}.
Imposing
\begin{equation}\label{smallness2}
\Lambda_2\leq \left(\frac{\hat m}{4\|\underline u\|_\infty}\right)^{1+\eta},
\end{equation}
one has $\|\underline u_{\lambda_1}\|_\infty\leq \Lambda_2^{\frac{1}{1+\eta}} \|\underline{u}\|_\infty\leq \hat m/4$. Accordingly, as in \eqref{enest:a}, using the nonnegativity of $w_2$ in \ref{hypf},
\begin{align*}
&\int_{\R^N} A\left(z, \frac{\hat{u}}{2}\right) \dz\ge \int_{\{\hat{u}/2>\underline u_{\lambda_1}\}} \left(\int_{\underline u_{\lambda_1}}^{\hat{u}/2} a(z,t) \dt \right)\dz\ge \int_{\{\hat{u}/2>\underline u_{\lambda_1}\}} \left(\int_{\underline u_{\lambda_1}}^{\hat{u}/2} \lambda_1 w_1 t^{-\eta} \dt \right)\dz\\
&=\frac{\lambda_1}{1-\eta}\int_{\{\hat{u}/2>\underline u_{\lambda_1}\}} w_1(z)\left[\left(\frac{\hat u}{2}\right)^{1-\eta}-\underline u_{\lambda_1}^{1-\eta}\right]\dz\ge \frac{\lambda_1}{1-\eta}\int_{B(z_0,\varrho)} w_1(z)\left[\left(\frac{\hat u}{2}\right)^{1-\eta}-\underline u_{\lambda_1}^{1-\eta}\right]\dz\\
&\ge \frac{\lambda_1}{1-\eta}\left[\left(\frac{\hat m}{2}\right)^{1-\eta}-\left(\frac{\hat m}{4}\right)^{1-\eta}\right]\int_{B(z_0,\varrho)} w_1(z)\dz\ge \frac{\lambda_1}{1-\eta}\left[\left(\frac{\hat m}{2}\right)^{1-\eta}-\left(\frac{\hat m}{4}\right)^{1-\eta}\right] \omega\eqqcolon 2 \lambda_1\check C,
\end{align*}
since $B(z_0,\varrho)\subseteq\{\hat{u}/2>\underline u_{\lambda_1}\}$. Incidentally, notice that $\check C$ depends only on $N, \gamma,w_1,\eta$. From \eqref{MPlevel:Jest} we deduce
\begin{equation}\label{Jest1}
J(\overline{t}\hat{u})\le \frac{S^{N_\gamma/2}}{N_\gamma}-2\lambda_1 \check C.
\end{equation}
Now assume $\lambda_1\in \Upsilon_2$. Setting
$$ \check{c}\coloneqq \max_{t\in\left[0,\frac{1}{2}\right]}\left(\frac{t^2}{2}- \frac{t^{2^*_\gamma}}{2^*_\gamma}\right)<\frac{1}{N_\gamma}, $$
from the nonnegativity of $a$ defined in \eqref{defapicc}, we deduce
\begin{equation}\label{Jest2}
\begin{split}
J(\overline{t}\hat{u}) &= \frac{\overline{t}^2}{2}\|\nabla\hat{u}\|_2^2 - \int_{\R^N} a(z,\overline{t}\hat{u}) \dz - \frac{\overline{t}^{2^*_\gamma}}{2^*_\gamma} \|\hat{u}\|_{2^*_\gamma}^{2^*_\gamma}\le\left(\frac{\overline{t}^2}{2}- \frac{\overline{t}^{2^*_\gamma}}{2^*_\gamma} \right)\|\nabla\hat{u}\|_2^2 \le \check c S^{N_\gamma/2}.
\end{split}\end{equation}
Recalling that
$$\hat{c} =\frac{S^{N_\gamma/2}}{N_\gamma}-\lambda_1^{\frac{2}{1+\eta}} \hat C\left(1+L^{2(r-1)}\right)$$
since $\lambda_2\le \lambda_1$, and noticing that $\frac{2}{1+\eta}>1$, we can impose the additional bound
\begin{equation}\label{smallness3}
\Lambda_2\leq \min\left\{\frac{1-N_\gamma\check c}{N_\gamma\check C} \, S^{N_\gamma/2},\left[\frac{\check C}{\hat C\left(1+L^{2(r-1)}\right)}\right]^{\frac{1+\eta}{1-\eta}}\right\},
\end{equation}
so that
$$J(\overline{t}\hat{u})\leq \max\left\{\check c S^{N_\gamma/2}, \frac{S^{N_\gamma/2}}{N_\gamma}-2\lambda_1\check C\right\} < \frac{S^{N_\gamma/2}}{N_\gamma}-\lambda_1\check C <\hat c$$
by \eqref{Jest1} and \eqref{Jest2}. The proof is concluded by choosing $c=\frac{S^{N_\gamma/2}}{N_\gamma}-\lambda_1\check C$.
\end{proof}
Now we are ready to prove existence of solutions to the truncated and frozen problem \eqref{varprob}.
\begin{thm}\label{existencefinal}
Let $L>0$ and $v\in D_0^\gamma(\R^N)$ be such that $\|\nabla_\gamma v\|_2\leq L$. Let $\lambda_1>0, \lambda_2\ge0$ such that $\lambda_2\le\lambda_1<\Lambda$ with $\Lambda\coloneqq \min\{\Lambda_1,\Lambda_2\}$, being $\Lambda_1$ and $\Lambda_2$ defined in Lemmas \ref{mountainpassgeometry} and \ref{talentihatc}, respectively. Then there exists $u\in  D_0^\gamma(\R^N)$ solution to \eqref{varprob}.
\end{thm}

\begin{proof}
By virtue of Lemmas \ref{PS}, \ref{mountainpassgeometry}, and \ref{talentihatc}, the conclusion follows from Theorem \ref{mountainpass}.
\end{proof}

\begin{rmk}
\label{subcomparisonrmk}
Any solution $u$ to either \eqref{varprob} or \eqref{prob} satisfies $u\geq \underline u_{\lambda_1}$, being $\underline u_{\lambda_1}$ defined in \eqref{usub}. Indeed, if $u$ solves \eqref{varprob}, then it satisfies (in weak sense)
\begin{displaymath}
-\Delta_\gamma u \geq a(\cdot,u) = \lambda_1 w_1\underline u_{\lambda_1}^{-\eta} + \lambda_2 w_2  | \nabla_\gamma v|^{r-1} \geq \lambda_1 w_1\underline u_{\lambda_1}^{-\eta} = -\Delta_\gamma \underline u_{\lambda_1} \quad \mbox{on} \;\; \{u<\underline u_{\lambda_1}\},
\end{displaymath}
because of \eqref{a} and the nonnegativity of $w_2$.
Thus, Lemma \ref{weakcomp} yields $u\geq \underline u_{\lambda_1}$ in $\R^N$. 

Note that, by \eqref{usub} and \eqref{upos} we have $\underline u_{\lambda_1}=\lambda_1^{\frac{1}{1+\eta}}\underline u\ge C>0$, where $\underline u$ satisfies \eqref{subprob}, implying $u\ge C>0$.
A similar argument holds for solutions to \eqref{prob}, after noticing that they are positive by definition.
\end{rmk}

\begin{rmk}
Note that Theorem \ref{existencefinal} holds true also for $\lambda_1=\lambda_2=0$, giving the existence of a nontrivial solution of mountain pass type to the critical equation
$$-\Delta_\gamma u= u^{2_\gamma^*-1}\qquad \text{in}\,\, \R^N.$$
\end{rmk}

As a last step, we show that the minimum of two super-solutions to \eqref{varprob} is still a super-solution. We adapt the arguments of \cite[Lemma 10]{LMZ}.
\begin{lemma} \label{lemma:supersolutions}
   Let $u_1$, $u_2 \in D_0^\gamma(\R^N)$ be super-solutions to \eqref{varprob}. Then $u\coloneqq \min\{u_1, u_2 \}$ is a super-solution to \eqref{varprob}.
\end{lemma}
\begin{proof}
    Let us choose $\eta\in C^\infty(\R^N)$ such that $\eta(z) = 0$ for $z\le 0$, $\eta'(z)\ge 0$ for $z\in (0,1)$ and $\eta(z) = 1$ for $z\ge 1$. Given $\eps>0$ we define
    \begin{displaymath}
    \eta_\eps(z)\coloneqq \eta\left( \frac{u_1(z) - u_2(z)}{\eps}\right) \quad \hbox{and} \quad \overline\eta_\eps(z) \coloneqq 1-\eta_\eps(z).
    \end{displaymath}
    Notice that $\eta_\eps\varphi\in D_0^\gamma(\R^N)$ for each $\varphi\in D_0^\gamma(\R^N)$. In fact, we trivially have $\eta_\eps \varphi\in L^{2^*_\gamma}(\R^N)$ and 
    $$\nabla_\gamma(\eta_\eps(z) \varphi) = \frac{1}{\eps} \eta'\left(\frac{u_1(z) - u_2(z)}{\eps}\right) \nabla_\gamma(u_2-u_1)\varphi + \eta_\eps(z)\nabla_\gamma \varphi\in L^2(\R^N).$$
The same holds for $\overline\eta_\eps\varphi.$ So
    \begin{align*} 
        \int_{\R^N} \nabla_\gamma u_2 \nabla_\gamma (\eta_\eps \varphi) \dz \ge \int_{\R^N} \left( a(z,u_2) + (u_2)_+^{2^*_\gamma-2} \right) \eta_\eps \varphi\dz,\\
        \int_{\R^N} \nabla_\gamma u_1 \nabla_\gamma (\overline\eta_\eps \varphi) \dz \ge \int_{\R^N} \left( a(z,u_1) + (u_1)_+^{2^*_\gamma-2} \right)  \overline\eta_\eps \varphi\dz,\\
    \end{align*}
which can be rewritten as
        \begin{multline} \label{stima1}
        \int_{\{u_1\ge u_2 \}} \eta_\eps \nabla_\gamma u_2 \nabla_\gamma \varphi \dz + \frac{1}{\eps}\int_{\{0<u_2-u_1<\eps\}} \eta' \left( \frac{u_1 - u_2}{\eps}\right) \varphi \nabla_\gamma u_2 \nabla_\gamma (u_1-u_2)\dz\\   \ge \int_{\R^N} \left( a(z,u_2) + (u_2)_+^{2^*_\gamma-2}  \right)\eta_\eps\varphi \dz
            \end{multline}
and       
      \begin{multline}  \label{stima2}
        \int_{\{u_1\le u_2 \}} \overline\eta_\eps \nabla_\gamma u_1 \nabla_\gamma \varphi \dz - \frac{1}{\eps}\int_{\{0<u_2-u_1<\eps\}} \eta' \left( \frac{u_1 - u_2}{\eps}\right) \varphi \nabla_\gamma u_1 \nabla_\gamma (u_1-u_2)\dz \\   \ge \int_{\R^N} \left( a(z,u_1) + (u_1)_+^{2^*_\gamma-2}  \right) \overline\eta_\eps \varphi\dz.
    \end{multline}
Adding \eqref{stima1} and \eqref{stima2} and recalling that $\eta_\eps'\ge 0$, we get
\begin{align} \label{stima3}
    &\int_{\{u_1\ge u_2 \}} \eta_\eps \nabla_\gamma u_2 \nabla_\gamma \varphi \dz + \int_{\{u_1\le u_2 \}} \overline\eta_\eps \nabla_\gamma u_1 \nabla_\gamma \varphi \dz \\\notag &\ge  \int_{\R^N} \left( a(z,u_2) + (u_2)_+^{2^*_\gamma-2} \right)\eta_\eps\varphi \dz + \int_{\R^N} \left( a(z,u_1) + (u_1)_+^{2^*_\gamma-2}  \right) \overline\eta_\eps \varphi\dz \\ \notag &\qquad + \frac{1}{\eps}\int_{\{0<u_2-u_1<\eps\}} \eta'_\eps \left( \frac{u_1 - u_2}{\eps}\right) \varphi \lvert \nabla_\gamma (u_1 - u_2)\rvert^2\dz\\ \notag &
    \ge \int_{ \{u_1\ge u_2 \}} \left( a(z,u_2) + (u_2)_+^{2^*_\gamma-2} \right)\eta_\eps\varphi \dz + \int_{ \{u_1\le u_2 \}} \left( a(z,u_1) + (u_1)_+^{2^*_\gamma-2}  \right) \overline\eta_\eps \varphi\dz.
\end{align}
Passing to the limit as $\eps\to 0$ in \eqref{stima3} by Lebesgue Dominated Convergence Theorem, then
\begin{align*}
   &\int_{\{u_1\le u_2 \}} \nabla_\gamma u_1 \nabla_\gamma \varphi\dz + \int_{\{u_1\ge u_2 \}}\nabla_\gamma u_2 \nabla_\gamma\varphi\dz \ge  \int_{\{u_1\le u_2 \}} \left( a(z,u_1)+ (u_1)_+^{2^*_\gamma}\right)\varphi\dz  \\ & \qquad + \int_{\{u_1\ge u_2 \}} \left( a(z,u_2) + (u_2)_+^{2^*_\gamma}\right)\varphi\dz
\end{align*}
and the claim follows.
\end{proof}

\subsection{Unfreezing the convection term}\label{unfr}
Let us introduce the constant
\begin{equation}
\label{Ldef}
L\coloneqq 2S^{N_\gamma/2}.
\end{equation}

\begin{rmk}
    As observed in \cite[Remark 4.6]{BG}, the choice \eqref{Ldef} is made for the sake of definiteness: actually, any choice of $L>S^{N_\gamma/2}$ allows to prove Theorem \ref{exsol}, provided $\Lambda$ is small enough.
\end{rmk}

We recall that the constants $\Lambda_1$ and $\Lambda_2$ were defined in Lemmas \ref{mountainpassgeometry} and \ref{talentihatc}. For a suitable constant $\Lambda_3$, whose value will be chosen later (see \eqref{lsccond2}), we set
\begin{equation*}
    \Lambda \coloneqq \min \left\lbrace \Lambda_1,\Lambda_2,\Lambda_3 \right\rbrace 
\end{equation*}
and we fix $\lambda_1>0, \lambda_2\ge0$ such that $\lambda_2 \leq \lambda_1 < \Lambda $.
We set
\begin{displaymath}
    \B\coloneqq \left\{u\in D_0^\gamma(\R^N): \, \|\nabla_\gamma u\|_2 < L \right\}\in 2^{D_0^\gamma(\R^N)}
\end{displaymath}
and we define
$\S \colon \B \to 2^{D_0^\gamma(\R^N)}$ by
\begin{equation}
\label{Sdef}
\S(v) \coloneqq  \left\{u\in D_0^\gamma(\R^N): \mbox{$u$ solves \eqref{varprob} and satisfies $J(u)< c$} \right\},
\end{equation}
where $c\in(0,\hat{c})$ is defined in Lemma \ref{talentihatc}. We explicitly notice that $\S$ depends on $\lambda_1$ and $\lambda_2$; anyway, for the sake of simplicity, we omit this dependence.

\begin{lemma}
\label{welldefined}
The set-valued function $\S$ is well defined, i.e., $\S(\B)\subseteq \B$.
\end{lemma}
\begin{proof}
Take any $v\in\B$ and $u\in\S(v)$. Applying Lemma \ref{enestlemma} with $u_n\equiv u$, $v_n\equiv v$, and $J_n\equiv J$, after observing that $J(u)<c<\hat{c}$ and $J'(u)=0$ by definition of $\S$, we get
$$ \|\nabla_\gamma u\|_2^2 < 2N_\gamma\left[ \hat{C}\lambda_1^{\frac{2}{1+\eta}}\left(L^{2(r-1)}+1\right)+\hat{c}\right]. $$
The conclusion then follows by recalling \eqref{hatc}.
\end{proof}
\begin{lemma}
\label{selectionlemma}
For any $v\in\B$, the set $\S(v)$ is non-empty and admits minimum.
\end{lemma}
\begin{proof}
Fix any $v\in\B$. The fact that $\S(v)\neq\emptyset$ is guaranteed by Theorem \ref{existencefinal}.

Now we prove that $\S(v)$ is downward directed. Let $u_1$, $u_2\in\S(v)$ and define $\overline{u}\coloneqq \min\{u_1,u_2\}$. Since $u_1, u_2\in\S(v)$, then $u_1$ and $u_2$ in particular are super-solutions to \eqref{varprob}. Thus, Lemma \ref{lemma:supersolutions} gives that $\overline{u}$ is a super-solution to \eqref{varprob}.

We now introduce the truncation $T\colon D_0^\gamma(\R^N)\to  D_0^\gamma(\R^N)$ defined by $T(u)(z)=\tau(z,u(z))$, where
$\tau\colon\R^N\times \R\to \R$ is the function
\begin{displaymath}
\tau(z,t)=\left\{
\begin{alignedat}{2}
&\underline u_{\lambda_1}(z) \quad &&\mbox{if} \;\; t<\underline u_{\lambda_1}(z), \\
&t \quad &&\mbox{if} \;\; \underline u_{\lambda_1}(z)\leq t\leq \overline{u}(z), \\
&\overline{u}(z) \quad &&\mbox{if} \;\; t>\overline{u}(z),
\end{alignedat}\right.
\end{displaymath}
and $\underline u_{\lambda_1}$ is defined in \eqref{usub}.  Note that $\underline u_{\lambda_1}\le \overline{u}$ in $\R^N$. In fact, since $\underline{u}_{\lambda_1}$ is a sub-solution to \eqref{varprob}, the inequality $\underline u_{\lambda_1}\le \overline{u}$  comes from Remark \ref{subcomparisonrmk}.

We claim that there exists $\check u\in D_0^\gamma(\R^N)$ such that $J(\check u)<\hat{c}$ and
\begin{equation}
\label{subprob1}
-\Delta_\gamma \check u = a(z,T(\check u)) + (T(\check u))^{2^*_\gamma-1} \quad \mbox{in} \;\; \R^N.
\end{equation}
The energy functional associated to \eqref{subprob1} is
$$\hat J(u)=\frac{1}{2}\|\nabla_\gamma u\|_2^2-\int_{\R^N} \hat B(z,u) \dz,$$
where
$$\hat B(z,s)\coloneqq \int_0^s \hat b(z,t) \dt, \qquad \hat b(z,t)\coloneqq  a(z,\tau(z,t))+\tau(z,t)^{2^*_\gamma-1}.$$
From the definition of $a$ in \eqref{a}, \eqref{usub}, and $\lambda_2\le \lambda_1\in(0,1)$ we estimate, reasoning as in Lemma \ref{enestlemma},
\begin{displaymath}
\begin{aligned}
\hat b(z,t) & = \lambda_1 w_1(z)\tau(z,t)^{-\eta}+ \lambda_2 w_2(z) |\nabla_\gamma v(z)|^{r-1}+ \tau(z,t)^{2^*_\gamma-1} \\
&\leq \lambda_1 w_1(z)\underline{u}_{\lambda_1}^{-\eta}+ \lambda_2 w_2(z) |\nabla_\gamma v(z)|^{r-1}+ \tau(z,t)^{2^*_\gamma-1} \\
&= \lambda_1^{\frac{1}{1+\eta}} w_1(z)\underline{u}^{-\eta}+ \lambda_2 w_2(z) |\nabla_\gamma v(z)|^{r-1}+ \tau(z,t)^{2^*_\gamma-1}\\
&\leq  \lambda_1^{\frac{1}{1+\eta}}\left(w_1(z) \underline{u}(z)^{-\eta}+ w_2(z)|\nabla_\gamma v(z)|^{r-1}\right)+ \overline{u}(z)^{2^*_\gamma-1}\eqqcolon h(z).
\end{aligned}
\end{displaymath}
Since $h\in L^{(2^*_\gamma)'}(\R^N)$, $\hat J$ is coercive: indeed,
\begin{align*}
    \hat J(u) &\geq \frac{1}{2}\|\nabla_\gamma u\|_2^2-\int_{\R^N} h(z)|u| \dz \geq \frac{1}{2}\|\nabla_\gamma u\|_2^2-\|h\|_{(2^*_\gamma)'}\|u\|_{2^*_\gamma} \\
    &\geq \frac{1}{2}\|\nabla_\gamma u\|_2^2-S^{-\frac{1}{2}}\|h\|_{(2^*_\gamma)'}\|\nabla_\gamma u\|_{2}.
\end{align*}
Moreover, it is readily seen that $\hat J$ is weakly sequentially lower semi-continuous. Thus, applying the direct methods of Calculus of Variations (see \cite[Theorem I.1.2]{S}),
there exists $\check u\in  D_0^\gamma(\R^N)$ such that $\hat J(\check u)=\min_{ D_0^\gamma(\R^N)}\hat J$, so $\check{u}$ satisfies \eqref{subprob1}. In particular, $\hat J(\check u)\leq \hat J(0)=0$.
In addition, $\underline{u}_{\lambda_1} \leq \check u\leq \overline{u}$. In fact, in $\{ \underline{u}_{\lambda_1} > \check{u}\}$, then $T(\check u) = \underline{u}_{\lambda_1}$ and
$$-\Delta_\gamma \check{u} = \lambda_1 w_1(z)\underline{u}_{\lambda_1}^{-\eta} + \lambda_2 w_2 (z)\lvert \nabla_\gamma v(z)\rvert^{r-1} +  (\underline{u}_{\lambda_1})_+^{2^*_\gamma-1} \ge -\Delta_\gamma \underline{u}_{\lambda_1},$$
being $\underline{u}_{\lambda_1}$ a sub-solution to \eqref{varprob}. Lemma \ref{weakcomp} implies $\underline{u}_{\lambda_1}\le \check u$. Analogously, in $\{ \check u > \overline{u}\}$, $T(\check u) = \overline{u}$ and
$$-\Delta_\gamma \check u = \lambda_1 w_1(z) \overline{u}^{-\eta} + \lambda_2 w_2(z) \lvert \nabla_\gamma v(z) \rvert^{r-1}+ \overline{u}_+^{2^*_\gamma-1} \le -\Delta_\gamma \overline{u},$$
being $\overline{u}$ a super-solution to \eqref{varprob}. Again we apply Lemma \ref{weakcomp}.

Accordingly, $\check{u}$ solves \eqref{varprob} and $\hat J(\check u)= J(\check{u})$, so that $J(\check u)\leq 0 < c$. The claim is proved. We also obtained $\check u\in\S(v)$. By arbitrariness of $u_1$ and $u_2$, the set $\S(v)$ is downward directed.

Let us consider a chain $\C\subset\S(v)$ and a decreasing sequence $(u_n)\subseteq\C$, i.e. $u_n\ge u_{n+1}$ for all $n$, such that $u_n$ solves \eqref{varprob} and satisfies $J(u_n)< c$ for all $n$.

Since $(u_n)\subseteq\C$, then Lemma \ref{PS} gives that there exist $u\in D_0^\gamma(\R^N)$ such that, up to subsequences $u_n\to u$ in $D_0^\gamma(\R^N)$, implying that $u$ solves \eqref{varprob} and satisfies $J(u)< c$, i.e. $u\in \S(v)$ and $u\le u_n$. An application of Zorn's Lemma yields that $\S(v)$ admits a minimum.
\end{proof}

\begin{lemma}
\label{Scompact}
The set-valued function $\S$ is compact.
\end{lemma}
\begin{proof}
Let $(v_n)$ be a (bounded) sequence in $\B$. For any $n\in\N$, pick $u_n\in\S(v_n)$. Our aim is to prove that $u_n \to u$ in $ D_0^\gamma(\R^N)$ for some $u\in D_0^\gamma(\R^N)$.

Remark \ref{subcomparisonrmk} ensures $u_n\geq 0$ in $\R^N$ for all $n\in\N$, so Lemma \ref{ccplemma} gives the existence of $u\in D_0^\gamma(\R^N)$ such that $u_n \rightharpoonup u$ in $ D_0^\gamma(\R^N)$ and $u_n\to u$ in $L^{2^*_\gamma}(\R^N)$. Notice that, for all $n\in\N$,
\begin{equation}
\label{S+test}
\begin{aligned}
0 = \langle J'_n(u_n),u_n-u \rangle &= \int_{\R^N}  \nabla_\gamma u_n \nabla_\gamma (u_n-u) \dz \\
&\quad-\int_{\R^N} a_n(z,u_n)(u_n-u) \dz - \int_{\R^N} u_n^{2^*_\gamma-1}(u_n-u) \dz.
\end{aligned}
\end{equation}
Computations similar to the ones of \eqref{PS:conctest1} show that $(a_n(\cdot,u_n))$ is bounded in $L^{(2^*_\gamma)'}(\R^N)$, and $(u_n^{2^*_\gamma-1})$ enjoys the same property; hence
$$ \lim_{n\to\infty} \int_{\R^N} a_n(z,u_n)(u_n-u) \dz = \lim_{n\to\infty} \int_{\R^N} u_n^{2^*_\gamma-1}(u_n-u) \dz = 0. $$
Letting $n\to\infty$ in \eqref{S+test} entails
$$ \lim_{n\to\infty} \langle -\Delta_\gamma u_n, u_n-u \rangle = \lim_{n\to\infty} \int_{\R^N}  \nabla_\gamma u_n \nabla_\gamma (u_n-u) \dz = 0, $$
so that the ${\rm (S_+)}$ property of  $-\Delta_\gamma$, Lemma \ref{lemma:S+}, yields $u_n \to u$ in $ D_0^\gamma(\R^N)$.
\end{proof}

\begin{lemma}
\label{lsc}
If $\Lambda_3=\Lambda_3(N,w_1, w_2,\eta,r,\gamma)>0$ is sufficiently small, then $\S$ is lower semi-continuous.
\end{lemma}

\begin{proof}
Let $v_n\to v$ in $ D_0^\gamma(\R^N)$ and $u\in\S(v)$. We have to construct a sequence $(u_n)\subseteq D_0^\gamma(\R^N)$ such that $u_n\in \S(v_n)$ for every $n\in\N$ and $u_n\to u$ in $ D_0^\gamma(\R^N)$. To this aim, we consider the following family of problems, parameterized by indexes $n,m\in\N$ and defined by recursion on $m$:
\begin{equation}
\label{recprob}
\left\{
\begin{alignedat}{2}
-\Delta_\gamma u_n^m &= a_n(z, u_n^{m-1}) + (u_n^{m-1})^{2^*_\gamma-1} \quad &&\mbox{in} \;\; \R^N, \\
u_n^0 &= u \quad &&\mbox{for all} \;\; n\in\N.
\end{alignedat}
\right.
\end{equation}
By induction on $m\in\N$, problem \eqref{recprob} admits a unique solution $u_n^m\in D_0^\gamma(\R^N)$ for all $n,m\in\N$, according to Minty-Browder's theorem \cite[Theorem 5.16]{B} since $D_0^\gamma(\R^N)$ is a reflexive Banach space.
Fixed $R>0$, for every $g\colon\R^N\to\R$ we define $g_R:\R^N\to\R$ as $g_R(z)\coloneqq g(Rx, R^{1+\gamma}y)$ for all $z\in\R^N$, see \cite[p. 531]{ambrosio2003}. By a change of variables
$\hat x\mapsto R x$ and $\hat{y} \mapsto R^{1+\gamma}y$ we get
$$\|g_R\|_q = R^{-N_\gamma/q} \|g\|_q\quad \hbox{ for all} \quad q\geq 1.$$
Moreover
$$\Delta_x g_R = R^2 \Delta_x g(Rx, R^{1+\gamma } y), \qquad \Delta_y g_R = R^{1+\gamma} \Delta_y g(Rx, R^{1+\gamma}y)$$
so that
\begin{equation}\label{laplgr}
\Delta_\gamma g_R = R^2 (\Delta_x g + (|x|R)^{2\gamma} \Delta_y g) = R^2\Delta_\gamma g(Rx, R^{1+\gamma} y).
\end{equation}
We want to determine $R,\Lambda_3>0$ such that the set $\{z_n^m: \, n,m\in\N\}$ is bounded in $ D_0^\gamma(\R^N)$, being $z_n^m\coloneqq (u_n^m)_R$ for all $n,m\in\N$.  By using \eqref{laplgr}, we observe that $z_n^m$ solves
\begin{equation}
\label{blowup}
-\Delta_\gamma z_n^m(z) = R^2 \left[a_n(Rx, R^{1+\gamma}y,z_n^{m-1})+(z_n^{m-1}(z))^{2^*_\gamma-1}\right].
\end{equation}
Testing \eqref{blowup} with $z_n^m$, besides using \ref{hypf}, \eqref{usub} the definition of $a_n$ in \eqref{defapicc}, $\lambda_2\le\lambda_1 \in(0,\Lambda) \subseteq(0,1)$, H\"{o}lder's inequality, and the boundedness of $(v_n)$ in $ D_0^\gamma(\R^N)$, produces
\begin{displaymath}
\begin{aligned}
&\|\nabla_\gamma z_n^m\|_2^2 \\
&= R^2 \left[ \int_{\R^N} \left(\lambda_1 (w_1)_R (\underline u_{\lambda_1})_R^{-\eta} + \lambda_2 (w_2)_R |(\nabla_\gamma v_n)_R|^{r-1}\right) z_n^m \dz + \int_{\R^N} (z_n^{m-1})^{2^*_\gamma-1}z_n^m \dz \right] \\
&\leq R^2 \left[\lambda_1^{\frac{1}{1+\eta}}\int_{\R^N} \left((w_1)_R \,\underline{u}_R^{-\eta} + (w_2)_R|(\nabla_\gamma v_n)_R|^{r-1}\right) z_n^m \dz + \|z_n^{m-1}\|_{2^*_\gamma}^{2^*_\gamma-1} \|z_n^m\|_{2^*_\gamma} \right] \\
&\leq R^2 \left[ \lambda_1^{\frac{1}{1+\eta}} \left( \|(w_1)_R\, \underline{u}_R^{-\eta}\|_{(2^*_\gamma)'} + \|(w_2)_R\|_\theta\|(\nabla_\gamma v_n)_R\|_2^{r-1} \right) \|z_n^m\|_{2^*_\gamma} + \|z_n^{m-1}\|_{2^*_\gamma}^{2^*_\gamma-1} \|z_n^m\|_{2^*_\gamma} \right] \\
&\leq R^2 S^{-\frac{1}{2}}\|\nabla_\gamma z_n^m\|_2 \left[\lambda_1^{\frac{1}{1+\eta}} R^{-\frac{N_\gamma}{(2^*_\gamma)'}} \left( \|w_1\underline{u}^{-\eta}\|_{(2^*_\gamma)'} + \|w_2\|_\theta\|\nabla_\gamma v_n\|_2^{r-1} \right) + S^{-\frac{2^*_\gamma-1}{2}} \|\nabla_\gamma z_n^{m-1}\|_2^{2^*_\gamma-1} \right],
\end{aligned}
\end{displaymath}
where $\theta$ is defined in \eqref{def:teta}.

Setting $H\coloneqq\left[ S^{-\frac{1}{2}}\left(\|w_1\underline{u}^{-\eta}\|_{(2^*_\gamma)'} + \|w_2\|_\theta L^{r-1}\right)\right]$, being $\lambda_1 \in(0,\Lambda)$ we get
\begin{equation}
\label{rec}
\|\nabla_\gamma z_n^m\|_2 \leq H \Lambda^{\frac{1}{1+\eta}} R^{ 1- \frac{N_\gamma}{2}} + R^2 S^{-\frac{2^*_\gamma}{2}}\|\nabla_\gamma z_n^{m-1}\|_2^{2^*_\gamma-1}.
\end{equation}
Now we want to apply Lemma \ref{reclemma} to \eqref{rec}. First we estimate, via Lemma \ref{welldefined},
\begin{displaymath}
\|\nabla_\gamma z_n^0\|_2 = \|(\nabla_\gamma u)_R\|_2 = R^{-\frac{N_\gamma}{2}}\|\nabla_\gamma u\|_2 \leq R^{-\frac{N_\gamma}{2}}L.
\end{displaymath}
Hence the first condition in \eqref{smallnessconds} is met provided
\begin{equation}
\label{lsccond1}
\frac{1}{2} \geq R^{2} S^{-\frac{2^*_\gamma}{2}} \left(R^{-\frac{N_\gamma}{2}}L\right)^{2^*_\gamma-2} = S^{-\frac{2^*_\gamma}{2}} \left(\frac{L}{R}\right)^{2^*_\gamma-2}.
\end{equation}
On the other hand, the second condition in \eqref{smallnessconds} fulfilled whenever
\begin{equation}
\label{lsccond2}
2^{1-2^*_\gamma} > R^{2} S^{-\frac{2^*_\gamma}{2}} \left(H \Lambda^{\frac{1}{1+\eta}} R^{1-\frac{N_\gamma}{2}}\right)^{2^*_\gamma-2} = H^{2^*_\gamma-2} \Lambda^{\frac{2^*_\gamma-2}{1+\eta}} S^{-\frac{2^*_\gamma}{2}}.
\end{equation}
Choosing $R=R(N)>0$ sufficiently large and $\Lambda_3=\Lambda_3(N,w_1, w_2,\eta,r,\gamma)>0$ small enough, both conditions \eqref{lsccond1}--\eqref{lsccond2} are fulfilled. By virtue of Lemma \ref{reclemma}, after noticing that all the quantities appearing in \eqref{lsccond1}--\eqref{lsccond2} do not depend on $n$, we conclude that there exists $\hat{L}=\hat{L}(N,R,\Lambda,\gamma)>0$ such that $\|\nabla_\gamma z_n^m\|_2 \leq \hat{L} $ for all $n,m\in\N$, which implies $\|\nabla_\gamma u_n^m\|_2 \leq R^{N_\gamma/2}\hat{L}$ for all $n,m\in\N$.

Now we pass to the weak limit the double sequence $(u_n^m)$ with respect to each index separately: up to sub-sequences, there exists $(u_n),(u^m)\subseteq D_0^\gamma(\R^N)$ such that
\begin{equation}
\label{weaklimits}
\begin{aligned}
u_n^m &\rightharpoonup u^m \quad \mbox{in} \;\;  D_0^\gamma(\R^N) \;\; \mbox{as} \;\; n\to\infty, \quad \forall m\in\N, \\
u_n^m &\rightharpoonup u_n \quad \mbox{in} \;\;  D_0^\gamma(\R^N) \;\; \mbox{as} \;\; m\to\infty, \quad \forall n\in\N.
\end{aligned}
\end{equation}
Letting $n\to\infty$ in the weak formulation of \eqref{recprob}, it turns out that both $u^1$ and $u$ solve the problem
\begin{equation}
\label{limitprob}
\left\{
\begin{alignedat}{2}
-\Delta_\gamma U &= a(z, u(z)) + u(z)^{2^*_\gamma-1} \quad &&\mbox{in} \;\; \R^N, \\
U &\in  D_0^\gamma(\R^N).
\end{alignedat}
\right.
\end{equation}
Since \eqref{limitprob} admits a unique solution by Minty-Browder's theorem, we deduce $u^1=u$. Reasoning inductively on $m\in\N$, it follows that $u^m=u$ for all $m\in\N$. Pick an arbitrary $\rho>0$. Since $ D_0^\gamma(\R^N)\compact L^2(B_\gamma(0,\rho))$, the convergences mentioned in \eqref{weaklimits} are strong in $L^2(B_\gamma(0,\rho))$. Accordingly, the double limit lemma \cite[Proposition A.2.35]{GP} which holds in any metric space, guarantees, up to sub-sequences,
$$ \lim_{n\to\infty} u_n = \lim_{n\to\infty} \lim_{m\to\infty} u_n^m = \lim_{m\to\infty} \lim_{n\to\infty} u_n^m = \lim_{m\to\infty} u^m = u \quad \mbox{in} \;\; L^2(B_\gamma(0,\rho)). $$
In particular, since $\rho$ was arbitrary, a diagonal argument ensures $u_n\to u$ in $\R^N$.

Now we prove that $u_n\in\S(v_n)$ for all $n\in\N$. Letting $m\to\infty$ in the weak formulation of \eqref{recprob} reveals that $u_n$ solves \eqref{varprob} with $v=v_n$, for all $n\in\N$. Reasoning as in Lemma \ref{ccplemma}, boundedness of $(u_n)$ in $ D_0^\gamma(\R^N)$ allows us to assume $|\nabla_\gamma u_n|^2 \rightharpoonup \mu$ and $u_n^{2^*_\gamma} \rightharpoonup \nu$ for some bounded measures $\mu,\nu$. According to Lemmas \ref{thm:concentration}--\ref{bennaoum}, there exist some at most countable set $\A$, a family of points $(z_j)_{j\in \A} \subseteq \R^N$, and two families of numbers $(\mu_j)_{j\in \A},(\nu_j)_{j\in \A}\subseteq(0,+\infty)$ such that
\begin{displaymath}
\begin{split}
\nu=u^{2^*_\gamma}+\sum_{j\in \A} \nu_j \delta_{z_j}, &\quad \mu\geq |\nabla_\gamma u|^2+\sum_{j\in \A} \mu_j \delta_{z_j}, \\
\limsup_{n\to\infty} \int_{\R^N} u_n^{2^*_\gamma} \dz = \int_{\R^N} \, {\rm d}\nu + \nu_\infty, &\quad \limsup_{n\to\infty} \int_{\R^N} |\nabla_\gamma u_n|^2 \dz = \int_{\R^N} \, {\rm d}\mu+ \mu_\infty,
\end{split}
\end{displaymath}
and
\begin{equation}
\label{lsc:measbounds}
S\nu_j^{2/2^*_\gamma} \leq \mu_j \quad \mbox{for all} \;\; j\in \A, \quad S\nu_\infty^{2/2^*_\gamma}\leq \mu_\infty,
\end{equation}
being $\mu_\infty,\nu_\infty$ as in Lemma \ref{bennaoum}. Suppose by contradiction that $J\neq\emptyset$, so that $\mu_j=\nu_j\geq S^{N_\gamma/2}$ for some $j\in \A$, according to \eqref{lsc:measbounds}. A computation analogous to \eqref{PS:concfinal}, jointly with $u\in\S(v)$, ensures that
$$ c > J(u) = J(u)-\frac{1}{2^*_\gamma}\langle J'(u),u \rangle \geq \hat{c} $$
being $\hat{c}$ defined by \eqref{hatc}, which contradicts $c<\hat{c}$. Hence concentration at points cannot occur; as in Lemma \ref{ccplemma}, a similar argument excludes concentration at infinity. We deduce $u_n\to u$ in $L^{2^*_\gamma}(\R^N)$, which is the starting point of the proof of Lemma \ref{Scompact}; thus we infer $u_n\to u$ in $ D_0^\gamma(\R^N)$. In particular, $J_n(u_n)\to J(u)$ as $n\to\infty$, so $J_n(u_n)<c$ for all $n$ sufficiently large, ensuring $u_n\in\S(v_n)$.
\end{proof}

Now we are ready to prove the main result of the present section, that is Theorem \ref{exsol}. 

{\it Proof of Theorem \ref{exsol}.}
Let us consider the following selection of the multi-function $\S$ defined in \eqref{Sdef}:
\begin{equation*}
\T\colon \B\to\B, \quad \T(v) \coloneqq \min \S(v).
\end{equation*}
The function $\T$ is well defined, according to Lemma \ref{selectionlemma}. Suppose we have proved that $\T$ is continuous and compact. Then Schauder's theorem \ref{schauder} applies, and $\T$ possesses a fixed point $u\in D_0^\gamma(\R^N)$. Remark \ref{subcomparisonrmk} guarantees $u\geq \underline u_{\lambda_1}$, so that $u$ solves \eqref{probnotdecay}. 

We only need to prove that $\T$ is continuous and compact.
Let $\A\subseteq \B$ be any bounded set. Since $\T(\A)\subseteq\S(\A)$ and $\S$ is relatively compact from Lemma \ref{Scompact}, then also $\T(\A)$ is relatively compact. 

Let $(w_n)\subseteq \B$ such that $w_n\to w$ in $D_0^\gamma(\R^N)$. For $u_n \coloneqq \T(w_n)$, the compactness of $\T$ implies that, up to a subsequence, $u_n\to u$ in $D_0^\gamma(\R^N)$.
First, we claim that $u\in \S(w)$. Since $u_n$ satisfies \eqref{varprob} with $v=w_n$, then a straightforward application of the Dominated Convergence Theorem shows that $u$ satisfies \eqref{varprob} with $v=w$. Moreover, $u_n\in \S(w_n)$ implies $J(u_n)<c$. Now $J(u)<c$ follows from the convergence $u_n\to u$ in $D_0^\gamma(\R^N)$. The claim is proved.

It remains to prove that $u=\T(w)$. Lemma \ref{lsc} provides a sequence $(v_n)\subseteq D_0^\gamma(\R^N)$ such that $v_n\in \S(w_n)$ and $v_n\to \T(w)$ in $D_0^\gamma(\R^N)$. By the definition of $\T$, then $u_n\coloneqq\T(w_n)\le v_n$. Letting $n\to\infty$, we have
$$\T(w)\le u=\lim_{n\to\infty} u_n \le \lim_{n\to\infty} v_n=\T(w), \qquad  \text{in}\,\, \R^N,$$
where the first inequality holds thanks to the definition of $\T(w)$ and $u\in \S(w)$, giving $u=\T(w)$.
\qed

\section{Regularity and decay of solutions}\label{reg}
From now on we will assume conditions \ref{hypf}-\ref{condw1}. 
We start this section by proving that any weak solution $u$ to \eqref{probnotdecay} belongs to $L^\infty(\R^N)$. We also establish the (local)-H\"{o}lder continuity of the solution $u$.

\begin{thm}
\label{regularity}
Let $\lambda_1, \lambda_2\ge 0$ and $u\in D_0^\gamma(\R^N)$ be a solution to \eqref{probnotdecay}. Then
$$\|u\|_\infty \leq M,$$
for a suitable $M=M(N,w_1, w_2 ,r,\gamma, \lambda_1, \lambda_2)>0$.  Moreover, $u\in C^{0, \tau}_\loc(\R^N)$ for some $\tau\in (0,1].$
\end{thm}
\begin{proof}
Given any $k>1$, we test \eqref{probnotdecay} with $(u-k)_+$. Since $u>k>1$ in $\Omega_k\coloneqq \{z\in\R^N: \, u(z)>k\}$, besides using \ref{hypf} and Peter-Paul's inequality (Lemma \ref{peterpaul}) we get
\begin{displaymath}
\begin{aligned}
&\|\nabla_\gamma (u-k)\|_{L^2(\Omega_k)}^2 \\
&\leq \lambda_1 \int_{\Omega_k} w_1 u^{1-\eta} \dz +\lambda_2\int_{\Omega_k} w_2 \lvert \nabla_\gamma u\rvert^{r-1} u\dz + \|u\|_{L^{2^*_\gamma}(\Omega_k)}^{2^*_\gamma} \\
&\leq \lambda_1 \|w_1\|_\infty \int_{\Omega_k}  u^{2^*_\gamma} \dz + \lambda_2 \int_{\Omega_k} w_2|\nabla_\gamma u|^{r-1}u \dz + \|u\|_{L^{2^*_\gamma}(\Omega_k)}^{2^*_\gamma} \\
&\leq \lambda_1 \|w_1\|_\infty\|u\|_{L^{2^*_\gamma}(\Omega_k)}^{2^*_\gamma}+  \lambda_2  \eps\|w_2\|_\infty^\theta |\Omega_k| + \lambda_2 \eps\|\nabla_\gamma u\|_{L^2(\Omega_k)}^2 + \lambda_2 C_\eps \|u\|_{L^{2^*_\gamma}(\Omega_k)}^{2^*_\gamma}  +  \|u\|_{L^{2^*_\gamma}(\Omega_k)}^{2^*_\gamma} \\
&\leq \lambda_2 \eps \|\nabla_\gamma (u-k)\|_{L^2(\Omega_k)}^2 + C_\eps\left(\|u\|_{L^{2^*_\gamma}(\Omega_k)}^{2^*_\gamma} + |\Omega_k|\right) \\
&\leq  \lambda_2 \eps \|\nabla_\gamma (u-k)\|_{L^2(\Omega_k)}^2 + C_\eps\left(\|u-k\|_{L^{2^*_\gamma}(\Omega_k)}^{2^*_\gamma}+k^{2^*_\gamma}|\Omega_k| + |\Omega_k|\right),
\end{aligned}
\end{displaymath}
where $\theta$ is defined in \eqref{def:teta} and $C_\varepsilon$ depends on $\lambda_1, \lambda_2$.
Choosing $\eps= \frac{1}{2\lambda_2}$ (when $\lambda_2 =0$ or $\lambda_1=0$ the estimate is trivial) and re-absorbing the term $\|\nabla_\gamma (u-k)\|_{L^2(\Omega_k)}^2$ on the left-hand side, we get
$$ \|\nabla_\gamma (u-k)\|_{L^2(\Omega_k)}^2 \leq C\left(\|u-k\|_{L^{2^*_\gamma}(\Omega_k)}^{2^*_\gamma}+k^{2^*_\gamma}|\Omega_k|\right), $$
for some $C>0$ depending only on $N,w_1, w_2, r,\gamma, \eta,$  $\lambda_1,$ $\lambda_2$. By Sobolev's inequality we obtain
$$ \|u-k\|_{L^{2^*_\gamma}(\Omega_k)}^2 \leq C\left(\|u-k\|_{L^{2^*_\gamma}(\Omega_k)}^{2^*_\gamma}+k^{2^*_\gamma}|\Omega_k|\right), $$
enlarging $C$ if necessary. Let $M>2$ and set $k_n\coloneqq M(1-2^{-n})$ for all $n\in\N$. Repeating verbatim the proof of \cite[Lemma 3.2]{CGL},
we infer that $\|u-k_n\|_{L^{2^*_\gamma}(\Omega_{k_n})} \to 0$ as $n\to\infty$, provided $\|u-M/2\|_{L^{2^*_\gamma}(\Omega_{M/2})}$ is small enough. 
We claim that
\begin{displaymath}
    \lim_{M \to +\infty} \|u-M/2\|_{L^{2^*_\gamma}(\Omega_{M/2})} =0.
\end{displaymath}
Indeed, since $u \in L^{2_\gamma^*}(\R^N)$, we have that, for any $K>0$,
\begin{displaymath}
    K \left\vert \left\{z\,:  u(z)^{2_\gamma^*}>K \right\} \right\vert \leq \Vert u \Vert_{L^{2_\gamma^*}(\R^N)}^{ 2^*_\gamma}.
\end{displaymath}
In particular,
\begin{displaymath}
    \lim_{M \to +\infty} \left\vert \Omega_{M/2} \right\vert =0.
\end{displaymath}
By the absolute continuity of the integral, see \cite[Proposition 16.3]{Secchi}, for any $\varepsilon>0$ we can pick $M$ so large that
\begin{displaymath}
    \int_{\Omega_{M/2}} u^{2_\gamma^*} \dz \leq \varepsilon.
\end{displaymath}
As a consequence,
\begin{displaymath}
\int_{\Omega_{M/2}} \left(u-\frac{M}{2}\right)^{2^*_\gamma} \dz \leq \int_{\Omega_{M/2}} u^{2^*_\gamma} \dz \to 0 \quad \mbox{as $M\to +\infty$}.
\end{displaymath}
Hence, $u\in L^\infty(\R^N)$.

Finally, the (local)-H\"{o}lder continuity of $u$ follows by the non-homogeneous Harnack inequality, \cite[Theorem 5.5]{gutierrezlanconelli}, and the arguments in \cite{kogojlanconelli}.
\end{proof}

Now we are ready to prove the last part to complete the main result of the paper, that is Theorem \ref{mainthm}. In particular, Theorem \ref{exsol} gives the existence of a weak solution to problem \eqref{prob}. Its regularity and decay (for $\lambda_2=0$) are given by Theorem \ref{regularity} above and Theorem \ref{thm:decay} below, respectively.

\begin{thm}
\label{thm:decay}
Let  $\lambda_1>0$ and $u\in D_0^\gamma(\R^N)$ be a solution to
\begin{equation}
\label{problema:decay}
\left\{
\begin{alignedat}{2}
-\Delta_\gamma u &=\lambda_1  w_1 (z) u^{-\eta} + u^{2^*_\gamma-1} \quad &&\mbox{in} \;\; \R^N, \\
u &> 0 \quad &&\mbox{in} \;\; \R^N. \\
\end{alignedat}
\right.
\end{equation}
Then
\begin{displaymath}
u(z)\to 0 \quad \mbox{as} \;\; d(z)\to+\infty.
\end{displaymath}
\end{thm}

\subsection{Preliminary estimates}
Now we give a preliminary lemma on weak Lebesgue spaces; we refer to Section \ref{sec:tools} for the definitions.
 \begin{lemma}
\label{weaklebboundlemma}
Any solution $u$ to \eqref{problema:decay} belongs to $L^{2_{*,\gamma}-1,\infty}(\R^N)$. In particular, $u\in L^r(\R^N)$ for all $r\in(2_{*,\gamma}-1,\infty]$.
\end{lemma}
\begin{proof}
Take any solution $u$ of \eqref{problema:decay}. Theorem \ref{regularity} ensures that $u\in L^\infty(\R^N)$, so we put $M\coloneqq \|u\|_\infty$. For any $h\in(0,M)$, test \eqref{problema:decay} with $T_h(u)\coloneqq \min\{u,h\}$. Then
\begin{equation}
\label{Thtest}
\int_{\{u\leq h\}} |\nabla_\gamma u|^2 \dz = \int_{\R^N} \lambda_1 w_1 u^{-\eta} T_h(u) \dz + \left[\int_{\{u\leq h\}} u^{2^*_\gamma} \dz + h\int_{\{u>h\}} u^{2^*_\gamma-1} \dz\right].
\end{equation}
Recall that, by Remark \ref{subcomparisonrmk}, $u\ge C>0$ a.e. in $\R^N$, where $C$ is a positive constant. Then we can estimate \eqref{Thtest} as
\begin{equation}
\label{Thtest1b}
\int_{\{u\leq h\}} |\nabla_\gamma u|^2 \dz \le  \lambda_1 C^{-\eta} \int_{\R^N} w_1 T_h(u) \dz + \left[\int_{\{u\leq h\}} u^{2^*_\gamma} \dz + h\int_{\{u>h\}} u^{2^*_\gamma-1} \dz\right].
\end{equation}
Now we estimate the terms into square brackets. It is readily seen that
\begin{equation}
\label{estterm1}
\int_{\{u\leq h\}} u^{2^*_\gamma} \dz = \int_{\R^N} T_h(u)^{2^*_\gamma} \dz - h^{2^*_\gamma}|\{u>h\}|,
\end{equation}
while the layer-cake lemma \cite[Proposition 1.1.4]{Gra} entails
\begin{equation}
\label{estterm2}
\begin{aligned}
\int_{\{u>h\}} u^{2^*_\gamma-1} \dz &= (2^*_\gamma-1)\int_0^M s^{2^*_\gamma-2} |\{u>\max\{s,h\}\}| \ds \\
&= h^{2^*_\gamma-1}|\{u>h\}|+(2^*_\gamma-1)\int_h^M s^{2^*_\gamma-2}|\{u>s\}| \ds.
\end{aligned}
\end{equation}
Thus, plugging \eqref{estterm1}--\eqref{estterm2} into \eqref{Thtest1b} yields
\begin{equation}
\label{Thtest2}
\begin{aligned}
&\int_{\{u\leq h\}} |\nabla_\gamma u|^2 \dz \\
&\le \lambda_1 C^{-\eta}\int_{\R^N} w_1 T_h(u) \dz + \left[\int_{\R^N} T_h(u)^{2^*_\gamma} \dz + (2^*_\gamma-1)h\int_h^M s^{2^*_\gamma-2}|\{u>s\}| \ds\right].
\end{aligned}
\end{equation}
In order to absorb $\int_{\R^N} w_1 T_h(u) \dz$ into $h\int_h^M s^{2^*_\gamma-2}|\{u>s\}| \ds$, we notice that
\begin{displaymath}
\lim_{h\to 0^+} \int_h^M s^{2^*_\gamma-1}|\{u>s\}| \ds = \frac{\|u\|_{2^*_\gamma}^{2^*_\gamma}}{2^*_\gamma}
\end{displaymath}
by the layer-cake lemma, so there exists $\overline{h}>0$ (depending on $u$) such that 
$$\int_h^M s^{2^*_\gamma-1}|\{u>s\}| \ds > \frac{\|u\|_{2^*_\gamma}^{2^*_\gamma}}{2 \cdot 2^*_\gamma} \quad \mbox{for all} \;\; h\in(0,\overline{h}).$$
Hence, exploiting also the H\"{o}lder inequality,
\begin{equation}
\label{estterm3}
\begin{aligned}
h\int_h^M s^{2^*_\gamma-2}|\{u>s\}| \ds &\geq \frac{h}{M} \int_h^M s^{2^*_\gamma-1}|\{u>s\}| \ds > \frac{\|u\|_{2^*_\gamma}^{2^*_\gamma}}{2M2^*_\gamma}\,h \\
&= \frac{\|u\|_{2^*_\gamma}^{2^*_\gamma}}{2M2^*_\gamma\|w_1\|_1}\,\|w_1\|_1 h \geq \frac{\|u\|_{2^*_\gamma}^{2^*_\gamma}}{2M2^*_\gamma\|w_1\|_1}\int_{\R^N} w_1 T_h(u) \dz
\end{aligned}
\end{equation}
for every $h\in(0,\overline{h})$. Putting \eqref{Thtest2} and \eqref{estterm3} together yields
$$ \int_{\{u\leq h\}} |\nabla_\gamma u|^2 \dz \leq C\left[\int_{\R^N} T_h(u)^{2^*_\gamma} \dz + h\int_h^M s^{2^*_\gamma-2}|\{u>s\}| \ds\right] $$
for some $C>0$ depending on $N,w_1,u, \gamma,\eta,\lambda_1$. Consequently, through Sobolev's inequality we obtain
\begin{equation}
\label{Thtest3}
\begin{aligned}
\int_{\R^N} T_h(u)^{2^*_\gamma} \dz &\leq S^{-\frac{N_\gamma}{N_\gamma-2}} \left(\int_{\{u\leq h\}} |\nabla_\gamma u|^2 \dz\right)^{\frac{N_\gamma}{N_\gamma-2}} \\
&\leq C\left[\int_{\R^N} T_h(u)^{2^*_\gamma} \dz + h\int_h^M s^{2^*_\gamma-2}|\{u>s\}| \ds\right]^{\frac{N_\gamma}{N_\gamma-2}} \\
&\leq C\left[\left(\int_{\R^N} T_h(u)^{2^*_\gamma} \dz\right)^{\frac{N_\gamma}{N_\gamma-2}} + \left(h\int_h^M s^{2^*_\gamma-2}|\{u>s\}| \ds\right)^{\frac{N_\gamma}{N_\gamma-2}}\right]
\end{aligned}
\end{equation}
for all $h\in(0,\overline{h})$. Lebesgue's dominated convergence theorem and $u\in L^{2^*_\gamma}(\R^N)$ entail $\int_{\R^N} T_h(u)^{2^*_\gamma} \dz \to 0$ as $h\to 0^+$, so taking a smaller $\overline{h}$ we have
\begin{displaymath}
C\left(\int_{\R^N} T_h(u)^{2^*_\gamma} \dz\right)^{\frac{N_\gamma}{N_\gamma-2}} \leq \frac{1}{2}\int_{\R^N} T_h(u)^{2^*_\gamma} \dz \quad \mbox{for all} \;\; h\in(0,\overline{h}).
\end{displaymath}
Thus, re-absorbing $\int_{\R^N} T_h(u)^{2^*_\gamma} \dz$ to the left, from \eqref{Thtest3} we infer
\begin{displaymath}
\int_{\R^N} T_h(u)^{2^*_\gamma} \dz \leq C\left(h\int_h^M s^{2^*_\gamma-2}|\{u>s\}| \ds\right)^{\frac{N_\gamma}{N_\gamma-2}} \quad \mbox{for all} \;\; h\in (0,\overline{h}).
\end{displaymath}
From \eqref{estterm1}, we deduce
\begin{displaymath}
h^{2^*_\gamma}|\{u>h\}| \leq \int_{\R^N} T_h(u)^{2^*_\gamma} \dz \leq C\left(h\int_h^M s^{2^*_\gamma-2}|\{u>s\}| \ds\right)^{\frac{N_\gamma}{N_\gamma-2}} \quad \mbox{for all} \;\; h\in (0,\overline{h}).
\end{displaymath}
Reasoning as in \cite[p.154]{V}, it turns out that $u\in L^{2_{*,\gamma}-1,\infty}(\R^N)$. Then \eqref{interpolation} and $u\in L^\infty(\R^N)$ ensure $u\in L^r(\R^N)$ for all $r\in(2_{*,\gamma}-1,\infty]$.
\end{proof}

We will also use a local boundedness theorem, Theorem \ref{locboundthm}. We provide its proof in Appendix \ref{appendix_a} to make the paper self-contained and for the reader’s convenience. 

Before proving Theorem \ref{thm:decay}, we preliminarily produce a (non-optimal) decay estimate for solutions to \eqref{problema:decay}. 

\begin{thm}\label{prelimdecay}
Let $\lambda_1>0$. Then, for any $\kappa<S^{N_\gamma/4}$, $\kappa'>0$, and $r>r'>0$, there exists $K_0>0$ (depending on $\kappa,\kappa',r,r',N,\gamma,\lambda_1,\delta$) 
such that
\begin{equation}
\label{decayprel}
u(z)\leq K_0d(z)^{\frac{2-N_\gamma}{2}} \quad \mbox{in} \;\; B_\gamma^c(0,r)
\end{equation}
for every $u\in D^\gamma_0(\R^N)$ solution to \eqref{problema:decay} satisfying
$$ \|\nabla_\gamma u\|_{L^2(B_\gamma^c(0,r'))} \leq \kappa \quad \mbox{and} \quad \|\nabla_\gamma u\|_2 \leq \kappa'. $$
\end{thm}

\begin{proof}
Fix $\kappa,\kappa',r,r'$ as in the statement, and set $r''\coloneqq\frac{r+r'}{2}$. Notice that \eqref{decayprel} is ensured by
\begin{equation}\label{cladec}
d(z,B_\gamma(0,r''))u(z)^{\frac{2}{N_\gamma-2}} \leq K' \quad \mbox{for all}\;\; z\in B_\gamma^c(0,r),
\end{equation}
where $K'>0$ is a suitable constant depending on $k,k',r,r',N,\lambda_1,\delta$. 
Indeed, if \eqref{cladec} holds true, then there exists $K''=K''(r,r')>0$ such that $d(z,B_\gamma(0,r''))>K''d(z)$ for all $z\in B_\gamma^c(0,r)$, since
$$ d(z)>r \quad\Rightarrow \quad r''<\frac{r''}{r}\,d(z) \quad\Rightarrow \quad d(z,B_\gamma(0,r''))=d(z)-r''>\left(1-\frac{r''}{r}\right)d(z)\eqqcolon K''d(z),$$
recalling that $r''<r$. Hence, our aim is to prove \eqref{cladec}.

Assume by contradiction that \eqref{cladec} fails, so that there exist $(u_n)\subseteq D^\gamma_0(\R^N)$ and $(z_n)\subseteq B_\gamma^c(0,r)$ such that, for all $n\in\N$, $u_n>0$ in $\R^N$ and
\begin{enumerate}[label=${\rm a_{\arabic*})}$, ref=${\rm a_{\arabic*}}$]
\itemsep0.5em
\item \label{equn}  $-\Delta_\gamma u_n =\lambda_1  w_1 (z) u_n^{-\eta} + u_n^{2^*_\gamma-1}$ in $\R^N$, 
\item \label{extbound} $\|\nabla_\gamma u_n\|_{L^2(B_\gamma^c(0,r'))}\leq \kappa$,
\item \label{sobolevbound} $\|\nabla_\gamma u_n\|_2 \leq \kappa'$,
\item  $d(z_n,B_\gamma(0,r''))u_n(z_n)^{\frac{2}{N_\gamma-2}}>2n$. 
\end{enumerate}
According to Theorem \ref{regularity}, $u_n\in L^\infty(\R^N)$ for all $n\in\N$. Then Lemma \ref{lem_dou} 
(with $k=n$, $M=u_n^{\frac{2}{N_\gamma-2}}$, and $D= \overline{B}_\gamma^c(0, r'')\subseteq {B}_\gamma^c(0, r'') = W$) furnishes $(\xi_n)\subseteq  \overline{B}_\gamma^c(0, r'')$ such that
\begin{enumerate}[label=${\rm b_{\arabic*})}$, ref=${\rm b_{\arabic*}}$]
\itemsep0.5em
\item \label{distance} $d(\xi_n,B_\gamma(0,r''))u_n(\xi_n)^{\frac{2}{N_\gamma-2}}>2n$,
\item  $u_n(z_n)\leq u_n(\xi_n)$, 
\item \label{doubling} $u_n(t)\leq 2^{\frac{N_\gamma-2}{2}} u_n(\xi_n)$ for all $t\in B_{\gamma}(\xi_n, n u_n(\xi_n)^{\frac{2}{2-N_\gamma}})$.
\end{enumerate}
For every $n\in\N$ we define $\mu_n\coloneqq u_n(\xi_n)^{-1}$ and
\begin{equation}
\label{blowupwt}
\tilde{u}_n(t) \coloneqq \mu_n u_n(w), \quad w=(w_x,w_y)=(\mu_n^{\frac{2}{N_\gamma-2}}t_x, \mu_n^{\frac{2(1+\gamma)}{N_\gamma-2}}t_y)+\xi_n,\quad  \;\; t=(t_x,t_y)\in\R^N.
\end{equation} 
Note that the variable $w$ depends on $n$, although we omit this dependence in the notation to keep it lighter and more readable.

In particular, since 
$$\frac{\partial \tilde u_n}{\partial (t_x)_i}=\mu_n^{\frac{N_\gamma}{N_\gamma-2}}\frac{\partial u_n}{\partial (w_x)_i},\qquad \frac{\partial \tilde u_n}{\partial (t_y)_j}=\mu_n^{\frac{2\gamma+N_\gamma}{N_\gamma-2}}\frac{\partial u_n}{\partial (w_y)_j}, \qquad i=1, \dots, m \quad\quad j=1, \dots ,\ell$$
so that
\begin{equation}
    \label{eq:gradienti}
    \begin{aligned}\nabla_{\gamma,t}\tilde u_n&=(\nabla_{t_x}\tilde u_n, |t_x|^{\gamma}\nabla_{t_y}\tilde u_n)=(\mu_n^{\frac{N_\gamma}{N_\gamma-2}}\nabla_{w_x} u_n, |w_x|^{\gamma}\mu_n^{\frac{2\gamma+N_\gamma}{N_\gamma-2}}\mu_n^{-\frac{2\gamma}{N_\gamma-2}}\nabla_{w_y} u_n)\\&=(\mu_n^{\frac{N_\gamma}{N_\gamma-2}}\nabla_{w_x} u_n, |w_x|^{\gamma}\mu_n^{\frac{N_\gamma}{N_\gamma-2}}\nabla_{w_y} u_n)=\mu_n^{\frac{N_\gamma}{N_\gamma-2}}\nabla_{\gamma,w} u_n
\end{aligned}
\end{equation}
and
$$
\begin{aligned}
\Delta_{\gamma,t}\tilde u_n&=\Delta_{t_x}\tilde u_n+|t_x|^{2\gamma}\Delta_{t_y}\tilde u_n=\mu_n^{\frac{N_\gamma+2}{N_\gamma-2}}\Delta_{w_x} u_n+|w_x|^{2\gamma}\mu_n^{-\frac{4\gamma}{N_\gamma-2}}\mu_n^{\frac{4\gamma+N_\gamma+2}{N_\gamma-2}}\Delta_{w_y} u_n\\&=\mu_n^{\frac{N_\gamma+2}{N_\gamma-2}}\Delta_{\gamma,w} u_n=\mu_n^{2^*_\gamma-1}\Delta_{\gamma,w} u_n.
\end{aligned}    
$$
Thus, from \eqref{equn} we infer
\begin{equation}
\label{equnblow}
-\Delta_{\gamma,t} \tilde{u}_n(t) = \lambda_1  w_1 (w)\mu_n^{2^*_\gamma-1+\eta} \tilde u_n^{-\eta}(t) + \tilde{u}_n^{2^*_\gamma-1}(t) \quad \mbox{in} \;\; \R^N.
\end{equation}
where $w$ is defined in \eqref{blowupwt}.
Exploiting \eqref{doubling}, along with the definition of $\mu_n$ and \eqref{blowupwt}, yields
\begin{equation}
\label{propblow}
\tilde{u}_n(0)=1 \quad \mbox{and} \quad \tilde{u}_n(t)\leq 2^{\frac{N_\gamma-2}{2}} \quad \mbox{for all }\;\; t\in B_\gamma(0,n).
\end{equation}
We claim that 
\begin{equation}\label{decnonlin}
\lim_{n\to\infty} \mu_n^{2^*_\gamma-1+\eta}f_1 (w)= 0 
\end{equation}
locally uniformly for $t\in\R^N$, where
$$f_1(w)=w_1 (w) \tilde u_n^{-\eta}(t).$$
Reasoning up to sub-sequences, we reduce to two cases:
\begin{enumerate}[label={${\rm (\Roman*)}$},ref={${\rm \Roman*}$}]
\itemsep0.5em
\item \label{vanishing} $\mu_n\to 0$;
\item \label{lowerbound} $\mu_n\geq c$ for all $n\in\N$, being $c>0$ opportune.
\end{enumerate}
In case \eqref{vanishing}, it is readily seen that \eqref{decnonlin} holds true uniformly in $\R^N$, since $f_1\in L^\infty(\R^N)$ since $\tilde u_n>0$, by Remark \ref{subcomparisonrmk}, and $w_1\in L^\infty(\R^N)$. 

Otherwise, if \eqref{lowerbound} occurs, fix any compact $H\subseteq\R^N$. Since $B_\gamma(0,n)\nearrow \R^N$, then $H\subseteq B_\gamma(0,n)$ for all $n\in\N$ sufficiently large, say $n>\nu\in\N$. Thus, \eqref{distance} and $d(\xi_n,B_\gamma(0,r''))<d(\xi_n)$ yield, for all $t\in H$,
$$ d(w)=d((\mu_n^{\frac{2}{N_\gamma-2}}t_x, \mu_n^{\frac{2(1+\gamma)}{N_\gamma-2}}t_y)+\xi_n)\geq d(\xi_n
)-\mu_n^{\frac{2}{N_\gamma-2}}d(t) \geq \mu_n^{\frac{2}{N_\gamma-2}}(2n-d(t
)) \geq \mu_n^{\frac{2}{N_\gamma-2}}n  $$
for all $n>\nu$. This, together with \eqref{propblow} and \ref{condw1}, we get
$$\mu_n^{2^*_\gamma-1+\eta}w_1 (w) \tilde u_n^{-\eta}(t)\le C \mu_n^{2^*_\gamma-1+\eta-\frac{2(\delta+2\gamma)}{N_\gamma-2}}n^{-\delta-2\gamma}\le C c^{2^*_\gamma-1+\eta-\frac{2(\delta+2\gamma)}{N_\gamma-2}}n^{-\delta-2\gamma},$$
 by Remark \ref{subcomparisonrmk}, since $$2^*_\gamma-1+ \eta - 2\frac{\delta + 2\gamma}{N_\gamma-2} \le 2^*_\gamma-1+ \eta - 2 \frac{N_\gamma + \eta(N_\gamma-2) + 2\gamma}{N_\gamma -2 } = -1-\eta - \frac{4\gamma}{N_\gamma-2}<0,$$
by the choice of $\delta$, for all $t\in H$ and $n>\nu$. Letting $n\to\infty$, besides taking into account the arbitrariness on $H$, proves \eqref{decnonlin}, taking into account that $-\delta-2\gamma<0$.

Fix any $z\in\R^N$ and observe that $B_\gamma(z,2) \subseteq B_\gamma(0,n)$ for any $n$ sufficiently large. 
Thus, \eqref{propblow} in \eqref{equnblow} gives
\begin{equation}\label{utildeprob}
0\leq -\Delta_{\gamma,t} \tilde u_n\leq \lambda_1  w_1 (w)\mu_n^{2^*_\gamma-1+\eta} \tilde u_n^{-\eta}(t) + 2^{ \frac{N_\gamma-2}{2} (2^*_\gamma-1)} \quad \mbox{in} \;\; B_\gamma(z,2)
\end{equation}
for all $n$ large enough. 

Thus \eqref{utildeprob}, together with \eqref{decnonlin}, implies that $(\tilde u_n)$ is bounded in $C^{0,\tau}(\overline{B_\gamma}(z,2))$ for some $\tau\in (0,1]$ by the non-homogeneous Harnack inequality, \cite[Theorem 5.5]{gutierrezlanconelli}, and the arguments in \cite{kogojlanconelli}.

In particular, since $z$ is arbitrary, Ascoli-Arzelà's theorem and a diagonal argument ensure
\begin{equation}
\label{c1locconv}
\tilde u_n\to\tilde{u}_\infty \quad \mbox{in} \;\;  C_\loc(\R^N)
\end{equation}
for some $\tilde{u}_\infty\in C_\loc(\R^N)$.
Moreover, \eqref{eq:gradienti} and \eqref{sobolevbound} give
\begin{displaymath} 
\begin{aligned}
\|\nabla_{\gamma,t} \tilde{u}_n\|_2 &=\int_{\R^N} |\nabla_{\gamma,t} \tilde{u}_n(t)|^2 \dt= \int_{\R^N} \mu_n^{\frac{2N_\gamma}{N_\gamma-2}}|\nabla_{\gamma,w} {u}_n(w)|^2 \dt\\&= \int_{\R^N} |\nabla_{\gamma,w} {u}_n(w)|^2 \,\mathrm{d}w= \|\nabla_{\gamma,w} u_n\|_2 \leq \kappa',
\end{aligned}
\end{displaymath}
so that $(\tilde{u}_n)$ is bounded in $D^{\gamma}_0(\R^N)$ and $\tilde{u}_n \rightharpoonup \tilde{u}_\infty$ in $D^{\gamma}_0(\R^N)$ up to a subsequence. By using the compactness of the embedding $D_0^\gamma(\R^N)\compact L^q_{loc}(\R^N)$ for all $q\in [1,2^*_\gamma)$, we have that $\tilde{u}_n\to \tilde{u}_\infty$ in $L^q_{loc}(\R^N)$ for all $q\in [1,2^*_\gamma)$.
Let $\varphi \in C_c^\infty(\mathbb{R}^N)$, and call $K$ the support of $\varphi$. Since $\nabla_{\gamma,t} \tilde{u}_n$ converges weakly to $\nabla_{\gamma,t} \tilde{u}_\infty$ in $L^2(\R^N)$ we have
\begin{displaymath}
    \int_{\mathbb{R}^N} \nabla_{\gamma,t} \tilde{u}_n \nabla_{\gamma,t} \varphi \dt \to \int_{\mathbb{R}^N} \nabla_{\gamma,t} \tilde{u}_\infty \nabla_{\gamma,t} \varphi \dt.
\end{displaymath}
Moreover, 
\begin{align*}
    \left\vert \int_{\mathbb{R}^N} |\tilde{u}_n|^{2^*_\gamma-1} \varphi \dt - \int_{\mathbb{R}^N} |\tilde{u}_\infty|^{2^*_\gamma-1} \varphi \dt \right\vert &\leq \int_K \left\vert |\tilde{u}_n|^{2^*_\gamma-1} - |\tilde{u}_\infty|^{2^*_\gamma-1} \right \vert |\varphi| \dt \\
    &\leq 2^{2^*_\gamma -2} \sup_{K} \left\vert \tilde{u}_n - \tilde{u}_\infty \right \vert^{2^*_\gamma-1} \Vert \varphi \Vert_{1} \\
    &= o(1) \Vert \varphi \Vert_{1}
\end{align*}
by \eqref{c1locconv}.
It follows from the boundedness of $(\tilde{u}_n)$, \eqref{equnblow} and \eqref{decnonlin} that 
\begin{equation}\label{eqinfty}
 -\Delta_{\gamma,t} \tilde{u}_\infty =\tilde u_\infty^{2^*-1} \quad \mbox{in} \;\; \R^N
\end{equation}
in the weak sense.
Observe that, for any $R>0$, one has
\begin{equation}\label{localenergy}
\|\nabla_{\gamma,t} \tilde{u}_n\|_{L^2(B_\gamma(0,R))} = \|\nabla_{\gamma,w} u_n\|_{L^2(B_{\gamma}(\xi_n, R\mu_n^{\frac{2}{N-2}}))} \quad \mbox{for all} \;\; n\in\N. 
\end{equation}
Now we prove that, for all $n\in\N$ large enough,
\begin{equation}\label{disjoint}
B_{\gamma}(\xi_n, R\mu_n^{\frac{2}{N-2}}) \cap B_\gamma(0,{r'}) = \emptyset.
\end{equation}
To this end, take any $y\in B_\gamma(0,{r'})$ and observe that \eqref{distance} and the fact that $r''>r'$ imply
\begin{displaymath}
\begin{aligned}
\mu_n^{-\frac{2}{N-2}}d(y-\xi_n)&\ge \mu_n^{-\frac{2}{N-2}}d(\xi_n)-\mu_n^{-\frac{2}{N-2}}d(y) = \mu_n^{-\frac{2}{N-2}}d(\xi_n,B_\gamma(0,r''))+\mu_n^{-\frac{2}{N-2}}(r''-d(y))\\
&\geq \mu_n^{-\frac{2}{N-2}}d(\xi_n,B_\gamma(0,r'')) > 2n,
\end{aligned}
\end{displaymath}
exploiting also $\xi_n\in \overline{B}_\gamma^c(0, r'')$ and $y\in B_\gamma(0,r'')$. Therefore, \eqref{disjoint} holds true for all $n>\frac{R}{2}$.

Moreover, from $\tilde{u}_n \rightharpoonup \tilde{u}_\infty$ in $D^{\gamma}_0(\R^N)$, we have $\nabla\tilde{u}_n \rightharpoonup \nabla\tilde{u}_\infty$ in $L^2(\R^N)$, and in $L^2(B_\gamma(0,R))$, giving
$$\|\nabla_{\gamma,t}\tilde{u}_\infty\|_{L^2(B_\gamma(0,R))}\le \liminf_{n\to\infty}\|\nabla_{\gamma,t}\tilde{u}_n\|_{L^2(B_\gamma(0,R))}\leq \liminf_{n\to\infty} \|\nabla_{\gamma,w} u_n\|_{L^2(B_\gamma^c(0,r'))} \leq \kappa $$
where we used \eqref{localenergy}--\eqref{disjoint} and \eqref{extbound} for $n$ large. 
Now we pass to the limit for $R\to\infty$ we end with 
\begin{displaymath}
    \|\nabla_{\gamma,t} \tilde{u}_\infty\|_2 \leq \kappa. 
\end{displaymath}
Now, since $ \tilde{u}_\infty$ satisfies \eqref{eqinfty}, then
\begin{equation}
    \label{grad_estim}\|\nabla_\gamma \tilde{u}_\infty\|_2 \geq  S^{N_\gamma/4}.
\end{equation}
In fact, testing \eqref{eqinfty} with $\tilde{u}_\infty$ and using Sobolev's embedding one has
$$ \|\nabla_\gamma \tilde{u}_\infty\|_2^2 = \|\tilde{u}_\infty\|_{2^*_\gamma}^{2^*_\gamma} \leq S^{-2^*_\gamma/2} \|\nabla_\gamma \tilde{u}_\infty\|_2^{2^*}, $$
and we get \eqref{grad_estim} just dividing by $ S^{-2^*_\gamma/2}\|\nabla_\gamma \tilde{u}_\infty\|_2^2>0$.

The hypothesis $\kappa<S^{N_\gamma/4}$ forces $\tilde{u}_\infty=0$, which contradicts $\tilde{u}_\infty(0)=1$, guaranteed by \eqref{propblow} and \eqref{c1locconv}.
\end{proof}

\begin{rmk}
In the classical setting, see \cite{Souplet}, after performing a blow up argument by the Doubling Lemma as in the proof of Theorem \ref{prelimdecay}, only \eqref{vanishing} occurs. In our situation we cannot get rid immediately of the case \eqref{lowerbound} due to the nature of the estimate, see \cite[Remark 4.4]{BG2}.
\end{rmk}

\subsection{Decay}
Finally, we establish an upper and lower bound for the function $u$ solution to \eqref{problema:decay}, which leads to the decay of $u$ as $d(z)\to + \infty$, concluding the proof of Theorem \ref{thm:decay}.
\begin{thm}
Let $\lambda_1>0$ and $u\in D^\gamma_0(\R^N)$ be a solution to \eqref{problema:decay}. Then there exist $C_0>0$, depending on $N,u,\gamma$, and $C_1>0$, depending on $N,\lambda_1,\delta,u,\gamma$, such that 
\begin{equation}\label{main1}
C_0\left(1+d(z)^{N_\gamma-2}\right)^{-1} \leq u(z)\leq C_1\left(1+d(z)^{N_\gamma-2}\right)^{-1}
\end{equation}
for all $z\in\R^N$.
\end{thm}

\begin{proof}
Take any $u\in\D_0^\gamma(\R^N)$ solution to \eqref{problema:decay}. 
Set $\kappa\coloneqq\frac{S^{N_\gamma/4}}{2}$. Then there exist $r',\kappa'>0$ (depending on $N,\gamma,u$) such that $\|\nabla_\gamma u\|_{L^2(B_\gamma^c(0,r'))}\leq\kappa$ and $\|\nabla_\gamma u\|_2\leq \kappa'$.
For any $R>0$ and $\xi = (\xi_x,\xi_y)\in\R^N$ we define
\begin{equation}
\label{blowup2}
u_R(\xi) \coloneqq R^{N_\gamma-2}u(R\xi_x, R^{1+\gamma}\xi_y).
\end{equation}
A rescaling argument, together with \ref{condw1}, Remark \ref{subcomparisonrmk} and \eqref{laplgr} yields
\begin{equation}
\label{scaledeq}
\begin{aligned}
-\Delta_\gamma u_R(\xi) &= - R^{N_\gamma} \Delta_\gamma u(R\xi_x, R^{1+\gamma}\xi_y) \\&= R^{N_\gamma}(\lambda_1 w_1( R\xi_x, R^{1+\gamma}\xi_y) u( R\xi_x, R^{1+\gamma}\xi_y)^{-\eta}+ u( R\xi_x, R^{1+\gamma}\xi_y)^{2^*_\gamma-1})\\&\eqqcolon f_R(\xi,u)
 \quad \mbox{in} \;\; \R^N\setminus\{0\}.
\end{aligned}
\end{equation}
where
\begin{equation}\label{stimafr}
\begin{aligned}
f_R(\xi,u)&\leq R^{N^\gamma} C\lambda_1 w_1( R\xi_x, R^{1+\gamma}\xi_y)+ R^{-2}u_R(\xi)^{2^*_\gamma-1} \\
&\le Cc_1\lambda_1 R^{N_\gamma-\delta-2\gamma} d(\xi)^{-\delta-2\gamma}+ R^{-2}u_R(\xi)^{2^*_\gamma-1}.\end{aligned}\end{equation}
Take any $r>r'$ and consider $R\geq r$. Thus, writing $u_R^{2^*_\gamma-1}=u_R^{2^*_\gamma-2}u_R$ and applying Theorem \ref{prelimdecay} to $u$ entails
\begin{equation}\label{stimaurstar}
\begin{aligned}
R^{-2}u_R^{2^*_\gamma-1}&\leq R^{-2}\left(K_0  R^{N_\gamma-2}(d( R\xi_x, R^{1+\gamma}\xi_y))^{\frac{2-N_\gamma}{2}}\right)^{2^*_\gamma-2}u_R \\
&= R^{-2}\left(K_0  R^{\frac{N_\gamma-2}{2}}d(\xi)^{\frac{2-N_\gamma}{2}}\right)^{2^*_\gamma-2}u_R
\\
&= K_0^{2^*_\gamma-2}d(\xi)^{-2}u_R \leq K_0^{2^*_\gamma-2} u_R  \quad \mbox{for all} \;\; \xi\in B_\gamma^c(0,1),
\end{aligned}
\end{equation}
since $B_\gamma^c(0,1)\subseteq B_\gamma^c(0,r/R)$.

Notice that $R^{N_\gamma-\delta-\gamma}d(\xi)^{-\delta-\gamma} \leq r^{N_\gamma-\delta-\gamma}$ for all $\xi\in B_\gamma^c(0,1)$, since $R\ge r$ and $\delta>N_\gamma$ by \ref{condw1}. Then, by inserting \eqref{stimafr}, \eqref{stimaurstar} in \eqref{scaledeq}, we have
\begin{displaymath}
-\Delta_\gamma u_R \leq a u_R + b \quad \mbox{in} \;\; B_\gamma^c(0,1),
\end{displaymath}
being $a\coloneqq K_0^{2^*_\gamma-2}$ and $b\coloneqq Cc_1\lambda_1 r^{N_\gamma-\delta-2\gamma}$. 

Then, by applying Theorem \ref{locboundthm} with $\Omega=B_\gamma(0,9)\setminus \overline{B_\gamma(0,1)}$ we deduce
\begin{equation}
\label{supest}
\|u_R\|_{L^\infty(B_\gamma(0,7)\setminus B_\gamma(0,3))} \leq C \|u_R\|_{L^{\alpha}(B_\gamma(0,9)\setminus B_\gamma(0,1))},
\end{equation}
for some $C>0$ and for every $\alpha>1$ .

Now we want to use \eqref{embedding} with $s=2_{*,\gamma}-1$ and $s-\eps=\alpha >1$, so that 
$$\eps=2_{*,\gamma}-1-\alpha\in (0,2_{*,\gamma}-1)\quad\iff\quad\alpha\in(1, 2_{*,\gamma}-1)=\left(1,\frac{N_\gamma}{N_\gamma-2}\right)$$
in order to have
\begin{equation}
\label{weaklebbound}
\|u_R\|_{L^{1+\eps}(B_\gamma(0,9)\setminus B_\gamma(0,1))} \leq C_\eps\|u_R\|_{L^{2_{*,\gamma}-1,\infty}(B_\gamma(0,9)\setminus B_\gamma(0,1))},
\end{equation}
enlarging $C_\eps$ if necessary.

In order to apply Lemma \ref{weaklebboundlemma} we point out that  $\|\cdot\|_{2_{*,\gamma}-1,\infty}$ is invariant under the scaling \eqref{blowup2}.
In fact, by the change of variables $\tilde \xi_x \mapsto R\xi_x$ and $\tilde \xi_y \mapsto R^{1+\gamma}\xi_y$ we get
\begin{displaymath}
\left| \left\{|u_R|>h \right\} \right| = \left| \left\{ R^{N_\gamma-2}|u(R\xi_x, R^{1+\gamma}\xi_y)|>h \right\} \right| = R^{-N_\gamma} \left| \left\{ R^{N_\gamma-2}|u|>h \right\} \right|.
\end{displaymath}
Now setting $k\coloneqq h R^{2-N_\gamma}$ we have
$$
|| u_R ||_{s,\infty} = \sup_{k>0} \left( k R^{N_\gamma-2} R^{-N_\gamma/s} | \{ |u|>k \} |^{1/s} \right) = R^{N_\gamma-2 - N_\gamma/s  } || u ||_{s,\infty},
$$
and for $s=2_{*,\gamma}-1$ the exponent $N_\gamma-2-N_\gamma/s$ vanishes.

Hence, using \eqref{supest}--\eqref{weaklebbound}
and Lemma \ref{weaklebboundlemma} to \eqref{scaledeq} (up to constants), yield
\begin{equation}
\label{supestfinal}
\|u_R\|_{L^\infty(B_\gamma(0,7)\setminus B_\gamma(0,3))} \leq \hat{C} \quad \mbox{for all} \;\; R\geq r.
\end{equation}
for a suitable $\hat{C}>0$ depending on $N,\gamma,\lambda_1,\delta,u$. 

Finally, for any $z\in B_\gamma^c(0,5r)$, applying \eqref{supestfinal} with $R=\frac{d(z)}{5}$ we obtain
\begin{displaymath}
u(z)\leq C d(z)^{2-N_\gamma} \quad \mbox{for all} \;\; z\in B_\gamma^c(0,5r),
\end{displaymath}
where $C>0$ depends on $N,\gamma,\lambda_1,\delta,r,u$. 
Exploiting Theorem \ref{regularity}, it turns out that
\begin{equation}
\label{finalests}
u(z)\leq C_1\left(1+d(z)^{N_\gamma-2}\right)^{-1} 
\end{equation}
for all $z\in\R^N$, for suitable $C_1>0$ depending on $N, \gamma,\lambda_1,\delta,u$.

Now fix any $r>0$ and set $c\coloneqq\inf_{B_\gamma(0,r)} u$, which is positive by the fact that $u$ is positive a.e. in $\R^N$ and it is continuous by Theorem \ref{regularity}.

Recalling the fundamental solution of the Grushin operator $\Gamma$ defined in \eqref{gammadef}, let $\Gamma_r$ be such that $\Delta_\gamma \Gamma_r=0$ in $\R^N\setminus\{0\}$ with $\Gamma_r(z)=c$ for all $z\in \partial B_\gamma(0,r)$, that is,
\begin{displaymath}
\Gamma_r(z)\coloneqq \frac{c}{r^{{2-N_\gamma}}}\Gamma(z)= c\left(\frac{d(z)}{r}\right)^{{2-N_\gamma}} \quad \mbox{for all} \;\; z\in\R^N\setminus\{0\}.
\end{displaymath}
Then the weak comparison principle in exterior domains,
Lemma \ref{weakcomp_ext} ensures $u\geq \Gamma_r$ in $B_\gamma^c(0,r)$, so there exists $\hat{c}>0$ (depending on $N,r,c,\gamma$) such that
$$ u(z) \geq \hat{c}d(z)^{{2-N_\gamma}} \quad \mbox{for all} \;\; z\in B_\gamma^c(0,r). $$
Hence
\begin{equation}
\label{finalestu}
u(z)\geq C_0\left(1+d(z)^{N_\gamma-2}\right)^{-1} \quad \mbox{for all} \;\; z\in\R^N,
\end{equation}
where $C_0>0$ is an opportune constant depending on $N,u,\gamma$.
Putting \eqref{finalests} and \eqref{finalestu} together entail \eqref{main1}. 
\end{proof}

\section*{Acknowledgments}
\noindent

\noindent
 The authors are member of the {\em Gruppo Nazionale per l'Analisi Ma\-te\-ma\-ti\-ca, la Probabilit\`a e le loro Applicazioni}
 (GNAMPA) of the {\em Istituto Nazionale di Alta Matematica} (INdAM). \\
 Laura Baldelli is partially supported by the
 ``Maria de Maeztu'' Excellence Unit IMAG, reference CEX2020-001105-M, funded by MCIN/AEI/10.13039/501100011033/ and by the INdAM-GNAMPA Project 2025 titled {\em Regolarit\`a ed Esistenza per Operatori Anisotropi} (E5324001950001). \\

\section*{Author statements}
\noindent
{\bf Author contribution.} All authors have accepted responsibility for the entire content of this manuscript and consented to its submission to the journal, reviewed all the results and approved the final version of the manuscript. All authors contributed equally to this work. \\
{\bf Conflict of interest.} The authors state no conflict of interest.

\appendix
\section{}
\label{appendix_a}

Local boundedness results for (sub)solutions to elliptic equations have been developed in general setting, see for instance \cite{gilbargtrudinger,Han,PS}. In this appendix we prove the following result, which extends some ideas of \cite{gutierrezlanconelli} for weak solutions to $X$-elliptic operators.
\begin{thm}
\label{locboundthm}
Let $\Omega\subseteq\R^N$ be a bounded domain and $u\in H^1_\gamma(\R^N)_{\loc}$ be a non-negative solution to
\begin{equation}\label{probloc}
    -\Delta_\gamma u \leq au +b \quad \mbox{in} \;\; \Omega,
\end{equation} 
where $a,b>0$. Then, for all $\alpha>1$, $z\in\Omega$, and $\mathcal R>0$ such that $B_\gamma(z,2\mathcal R)\subseteq\Omega$, one has
\begin{equation*}
    \sup_{B_\gamma(z,\mathcal R)} u \leq C\left[ \mathcal R^{-\frac{N_\gamma}{\alpha}}\|u\|_{L^\alpha(B_\gamma(z,2\mathcal R)}+\mathcal R^{2}b \right],
\end{equation*}
where $C>0$ depends on $N$, $\mathcal R$, $\alpha$, $\gamma$, $a$, $b$.
\end{thm}
We first recall a technical lemma, which is a direct consequence of Young's inequality.

\begin{lemma}[\text{\cite[Lemma 6.2.2]{PS}}]
    \label{lemma_ps}
    Let $\zeta\in\R$ and $\omega_1,\omega_2>0$. If $\zeta^2\le \omega_1 \zeta+\omega_2$, then also
    \begin{displaymath}
         \zeta\le \omega_1 + \sqrt{2\omega_2}.
    \end{displaymath}
\end{lemma}
\begin{proof}[Proof of Theorem \ref{locboundthm}.]
The proof is an adaptation of the Moser iteration technique for the Grushin setting.

We do not claim any originality, but we provide a proof since we could not find a precise reference. 

Throughout the proof, we will denote by $C$ a generic a constant, which may change from line to line, bout does not depend on the function $u$.

Assume initially $\mathcal{R}=1$.
Let $m$, $q$, $k>0$ be such that $m>k$ and set $w\coloneqq u+k$. If $q\ge 1$ we define
$$\psi(t)\coloneqq \frac{r^2}{q}\begin{cases}
t^q \qquad & \text{if } 0<t<m,\\
qm^{q-1}t-(q-1)m^q\qquad & \text{if } t\ge m.
\end{cases}$$
Otherwise, i.e. if $0<q<1$, 
$$\psi(t)\coloneqq r^2 t^q, \qquad t>0,$$
where, in both cases, $r$ is given by $q-1=2(r-1)$. 

Let $\eta\in C^1(\Omega)$ be nonnegative and vanishing in a neighborhood of $\partial\Omega$. Since $\psi'(t)$ is uniformly bounded when $t\ge k$, then
$$\varphi(z)\coloneqq\eta^2(z)\psi(w(z))$$
can be chosen as a test function in \eqref{probloc}, that is
\begin{equation}
\label{eq:1}
\int_\Omega \nabla_\gamma u \nabla_\gamma\varphi\dz \le a \int_\Omega u\varphi\dz + b \int_\Omega \varphi\dz.    
\end{equation}
Moreover, we have $ u = w - k, \nabla_\gamma w = \nabla_\gamma u$ and
$$
\nabla_\gamma \varphi(z) = \eta^2(z) \psi'(w(z)) \nabla_\gamma w(z) + 2 \eta(z) \psi(z) \nabla_\gamma\eta(z).
$$
We define, for $q \geq 1$,
\[
v(w) \coloneqq 
\begin{cases}
w^r, & \text{if } k \leq w < m, \\
r m^{r-1} w - (r - 1) m^r, & \text{if } w \geq m,
\end{cases}
\]
and simply $v = v(w) \coloneqq w^r$ when $0< q < 1$. For simplicity from now on we only treat the case $q\ge 1$, since the case $0<q<1$ follows from similar ideas, see \cite{PS}. 

We have the following estimates:
\begin{align}
\psi'(w)|\nabla_\gamma w|^2 &= |v'(w)|^2 |\nabla_\gamma w|^2 = |\nabla_\gamma v(w)|^2,  \label{est_1}\\
\psi'(w)w^2 &= [w v'(w)]^2 \leq r^2 v^2(w), \label{est_1'} \\
\psi(w)w &\leq v(w) w v'(w)\leq r v^2(w), \label{est_2} \\
\psi(w)|\nabla_\gamma w| &\leq v(w) |v'(w) \nabla_\gamma w| = v(w) |\nabla_\gamma v(w)|. \label{est_3}
\end{align}
Indeed, if $w<m$, they follow quite easily, while, if $w\ge m$, noting that
\begin{align*}
    \psi'(w)&=r^2m^{q-1}, \qquad 
    \psi(w) w = r^2 m^{q-1} w^2 - \frac{r^2}{q}(q-1) m^qw, \qquad 
    v'(w)  = r m^{r-1},\\
    v(w) w v'(w) & = r^2 m^{2(r-1)}w^2 - r(r-1)m^{2r-1}w = r^2 m^{q-1}w^2 - r(r-1) m^q w,
\end{align*}
then \eqref{est_1} follows easily, while the first inequality in \eqref{est_2} and \eqref{est_3} hold true since $\frac{r^2}{q}(q-1)\ge r(r-1)$. 
Quite simple calculations permit us to prove \eqref{est_1'} by using $w\ge m$.
While, recalling
$$
v^2(w) = m^{q-1}\left[r^2(w-m)^2 + 2rm (w-m) + m^2 \right], 
$$
then the second inequality in \eqref{est_2} is equivalent to require 
\begin{align*}
    r^2 m^{q-1}w^2 - r(r-1) m^q w &\le m^{q-1} r\left[r^2(w-m)^2 + 2rm (w-m) + m^2 \right]\\
    & = m^{q-1} r \left[r^2w^2 + r^2m^2 -2r^2mw +  2rm (w-m) + m^2  \right],
\end{align*}
that is 
$$
w^2(r^2-r) + wm(-2r^2+3r-1) +m^2(r-1)^2\ge 0,
$$
which holds true since $ w\ge m$.

Making some calculations and using \eqref{est_1}, \eqref{est_3} we get
$$\begin{aligned}\int_\Omega &\nabla_\gamma u \nabla_\gamma(\eta^2(z)\psi(w(z)))\dz=\int_\Omega \nabla_\gamma w [2\eta(z) \psi(w(z))\nabla_\gamma\eta(z)+\eta^2(z)\psi'(w(z))\nabla_\gamma w(z)]\dz\\ & =\int_\Omega  2\eta(z) \psi(w(z))\nabla_\gamma\eta(z) \nabla_\gamma w\dz +\int_\Omega \eta^2(z)\psi'(w(z))|\nabla_\gamma w(z)|^2\dz\\ &=\int_\Omega  2\eta(z) \psi(w(z))\nabla_\gamma\eta(z) \nabla_\gamma w\dz+\int_\Omega \eta^2(z)|\nabla_\gamma v(z)|^2\dz\\
& \ge -2\int_\Omega  \eta(z) v(z)|\nabla_\gamma\eta(z)| |\nabla_\gamma v(z)|\dz+\int_\Omega \eta^2(z)|\nabla_\gamma v(z)|^2\dz
\end{aligned}$$
and, recalling that $u=w-k$, we have by \eqref{est_2}
$$\begin{aligned}
a \int_\Omega& u(z)\eta^2(z)\psi(w(z))\dz + b \int_\Omega \eta^2(z)\psi(w(z))\dz \\ &=a \int_\Omega w(z)\eta^2(z)\psi(w(z))\dz + (b - ka) \int_\Omega \eta^2(z)\psi(w(z))\dz\\
&\le a r \int_\Omega \eta^2(z) v(z)^2\dz + (b - ka) \int_\Omega \eta^2(z)\psi(w(z))\dz.
\end{aligned}$$
So that from \eqref{eq:1} we get
\begin{equation}\label{eq:1'}
\begin{aligned}
\int_\Omega \eta^2(z)|\nabla_\gamma v(z)|^2\dz\le & 2\int_\Omega  \eta(z) v(z)|\nabla_\gamma\eta(z)| |\nabla_\gamma v(z)| \dz + a r \int_\Omega \eta^2(z) v(z)^2\dz  \\&\quad + (b - ka) \int_\Omega \eta^2(z)\psi(w(z))\dz.
\end{aligned}
\end{equation}
If $b - ka \leq 0$ we simply neglect the last term. If $b - ka>0$, since  $w = u+k \ge k>0$, the estimate \eqref{est_2} implies
\begin{displaymath}
    \psi(w) \le r \frac{v^2}{w} \le r \frac{v^2}{k},  
\end{displaymath}
which we insert into \eqref{eq:1'} to compare the last two integrals.
So we obtain
\begin{equation}\begin{aligned}
    \label{eq:2}
    \int_\Omega \eta^2(z)|\nabla_\gamma v(z)|^2\dz\le 2\int_\Omega  \eta(z) v(z)&|\nabla_\gamma\eta(z)| |\nabla_\gamma v(z)| \dz \\&+ \max\left\{\frac{b}{k}, a\right\} r \int_\Omega \eta^2(z)v^2(z)\dz.
\end{aligned}\end{equation}
Define
$$
\zeta\coloneqq \frac{\|\eta \nabla_\gamma v \|_{L^2(\Omega)}}{\|\eta v \|_{L^{2^*_\gamma}(\Omega)}},\quad \xi\coloneqq \frac{\| \eta v\|_{L^2(\Omega)}}{\| \eta v \|_{L^{2^*_\gamma}(\Omega)}}, \quad \hat{\xi} \coloneqq \frac{\| v \nabla_\gamma \eta \|_{L^2(\Omega)}}{\| \eta v \|_{L^{2^*_\gamma}(\Omega)}}.
$$
Now, we properly estimate \eqref{eq:2}. In particular, $\int_\Omega \eta^2(z)|\nabla_\gamma v(z)|^2\dz = \|\eta v\|_{L^{2^*_\gamma}(\Omega)}^2 \zeta^2$, 
$$
2\int_\Omega  \eta(z) v(z)|\nabla_\gamma\eta(z)| |\nabla_\gamma v(z)| \dz\le 2 \| v\nabla_\gamma \eta\|_{L^2(\Omega)} \|\eta \nabla_\gamma v \|_{L^2(\Omega)} = 2 \zeta \hat{\xi} \|\eta v\|_{L^{2^*_\gamma}(\Omega)}^2.
$$
Moreover
$$
r \int_\Omega \eta^2(z) v^2(z)\dz = r  \| \eta v\|_{L^2(\Omega)}^{2}=r  \xi^{2}\| \eta v \|_{L^{2^*_\gamma}(\Omega)}^2.
$$
Summarizing and dividing by $\| \eta v\|_{L^{2^*_\gamma}(\Omega)}^2$, \eqref{eq:2} becomes
\begin{equation}
\label{eq:3}
  \zeta^2 \leq 2\hat{\xi} \zeta + \max\left\{\frac{b}{k}, a\right\}r \xi^2  \le 2\hat{\xi}\zeta + \max\left\{\frac{b}{k}, a\right\}\left(1+\frac{1}{r}\right)r^2\xi^2.
\end{equation}
Now, by \eqref{eq:3}, Lemma \ref{lemma_ps}, the trivial estimate 
$\left(2(1/r+1)\right)^{1/2}\le 2 (1/r+1)$
{and together with the choice $k=b$}, we get
\begin{align}
    \label{eq:4}
    \zeta &\le 2\hat{\xi} + \left( 2  \max\left\{1, a\right\}\left(1+\frac{1}{r}\right)r^2\xi^2\right)^{1/2} \nonumber \\
    &\le 2\hat{\xi} + 2\left(1+\frac{1}{r}\right)\sqrt{\max\left\{1, a\right\}} r\xi \le 2C(1+r)(\hat{\xi} + \xi),
\end{align}
where $C$ depends on $a, b$.
The Sobolev inequality implies that 
$$\|\eta v \|_{2^*_\gamma}\le S \| \nabla_\gamma(\eta v) \|_2 \le  S \left( \|v\nabla_\gamma \eta \|_2 + \| \eta \nabla_\gamma v \|_2\right),$$
where $S$ is defined in \eqref{eq:sobconst}. So it follows $1\le S(\zeta + \hat{\xi})$  and \eqref{eq:4} becomes
\begin{displaymath}
    1\le 2S C(1+r)(\hat{\xi} + \xi) + S\hat{\xi},
\end{displaymath}
which implies
\begin{equation}
    \label{eq:5}
    1\le C (1+r) (\xi + \hat{\xi}).
\end{equation}
Recalling the definition of $\xi$ and $\hat{\xi}$, then \eqref{eq:5} gives
\begin{equation}
    \label{eq:6}
    \| \eta v\|_{L^{2^*_\gamma}(\Omega)} \le C  (1+r) (\| \eta v\|_{L^2(\Omega)} +  \| v\nabla_\gamma \eta\|_{L^2(\Omega)} ),
\end{equation}
where $C$ depends on $N,\gamma, a, b$.

Now we introduce some further conditions on the test function $\eta$. Let $h$, $h'$ be such that $1<h'<h<3$,
and set $\eta\equiv 1$ in $B_\gamma(0,h')$, $\eta\equiv 0$ in $\Omega \setminus B_\gamma(0,h)$, with $0\le\eta\le 1$ and $\sup_{B_\gamma(0,h)\setminus B_\gamma(0,h')} |\nabla_\gamma \eta|\le 2/(h-h')$. Then, from \eqref{eq:6} there follows
\begin{equation}
    \label{eq:7}
    \| v\|_{L^{2^*_\gamma}(B_\gamma(0,h'))}\le C \frac{(1+r)}{h-h'} \| v \|_{L^2(B_\gamma(0,h))}.
\end{equation}
For $r\ne 0$ let us define
$$
\Phi(r,h) \coloneqq \left( \int_{B_\gamma(0,h) } w^r \right)^{1/r}.
$$
Suppose $r>0$, $r\ne 2$ and $\overline k \coloneqq 2^*_\gamma/2 = N_\gamma/(N_\gamma-2)$. 
Then \eqref{eq:7} can be rewritten as { (after letting $m\to\infty$, so $v\to w^r$)}
$$
\Phi(2\overline kr, h')\le \left( \frac{C (1+r)}{h-h'} \right)^{1/r} \Phi(2r,h).
$$
Fix $\alpha>1$ as in the statement of the theorem and take successively
$$
r = r_j = \overline{k}^j \cdot \frac{\alpha}{2}, \quad h =  h_j = 1 + 2^{-j}, \quad h' = h_{j+1}, \qquad j=0,1,2,\dots
$$
We have, since $C(1+r)/(h-h')\le {2\alpha} C (2\overline k)^j$,
$$
\Phi(2r_{j+1},1)\le \Phi(2r_{j+1}, h_{j+1})\le \left[ {2\alpha} C (2\overline k)^j\right]^{\frac{1}{r_j}} \Phi(2r_j, h_j)
$$
and, iterating,
\begin{equation}
    \label{eq;8}
    \Phi(2r_{j+1}, 1)\le \prod_{i=0}^j \left[ {2\alpha} C (2\overline k)^i\right]^{\frac{1}{r_i}} \cdot \Phi(2r_0,h_0) = \prod_{i=0}^j \left[ {2\alpha} C (2\overline k)^i\right]^{\frac{1}{r_i}} \cdot \Phi(\alpha,2).
\end{equation}
In order to let $j$ tend to infinity, observe that
$$
\prod_{i=0}^\infty \left[ {2\alpha} C (2\overline k)^i\right]^{\frac{1}{r_i}} = \left[{2\alpha}C \right]^{ \sum_{i=0}^\infty  \frac{1}{r_i}}\cdot \prod_{i=0}^\infty [(2\overline k)^i]^\frac{1}{r_i},
$$
and the series converges since $\overline k>1$, in particular
$$
\left[{2\alpha} C \right]^{ \sum_{i=0}^\infty  \frac{1}{r_i}} =  \left( 2\alpha C\right)^{\frac{N_\gamma}{\alpha}}.
$$
Otherwise
$$
\prod_{i=0}^\infty [(2\overline k)^i]^{\frac{1}{r_i}} = e^{\log(2\overline k) \sum_{i=0}^\infty \frac{i}{r_i}}<\infty
$$
and so, letting $j\to\infty$ in \eqref{eq;8} we get
\begin{equation}\label{prelbound}
\sup_{B_\gamma(0,1)} w\le C \left(\int_{B_\gamma(0,2)} w^\alpha \dz\right)^{1/\alpha},
\end{equation}
where $C$ depends on $N,\alpha,\gamma,a,b.$ 
Now clearly $\sup_{B_\gamma(0,1)} w = \sup_{B_\gamma(0,1)} u + b$, while
$$
\|w\|_ {L^\alpha(B_\gamma(0,2))} \le \|u\|_ {L^\alpha(B_\gamma(0,2))} + |B_\gamma(0,2)|^{\frac{1}{\alpha}}b , 
$$
so \eqref{prelbound} becomes
\begin{equation}\label{estimR1}
    \sup_{B_\gamma(0,1)} u \le C  \left(\|u\|_{L^\alpha(B_\gamma(0,2))} + b  \right) - b \le C   \left(\|u\|_{L^\alpha(B_\gamma(0,2))} + b  \right),
\end{equation}
where $C$ continues to depend on $N$, $\alpha$, $\gamma,a, b$. This concludes the proof when $\mathcal{R}=1$.

The general case $\mathcal{R}>0$ is obtained by a change of scale:
\begin{displaymath}
    u_{\mathcal R}(\bar z)\coloneqq u(z)=u(\mathcal R \bar x, \mathcal R^{1+\gamma}\bar y),
\end{displaymath}
where $\bar z =(\bar x, \bar y)\in B_\gamma(0,1)$ and $z = (x,y)=(\mathcal R \bar x, \mathcal R^{1+\gamma}\bar y)\in B_\gamma(0,\mathcal R)$.

If $u$ satisfies \eqref{probloc}, then $u_{\mathcal R}$ satisfies
$$-\Delta_\gamma u_{\mathcal R}(\bar z)\le \mathcal R^2 au_{\mathcal R}(\bar z)+\mathcal R^2 b.$$
Applying \eqref{estimR1} to $u_{\mathcal R}$
we get
$$\sup_{B_\gamma(0,\mathcal R)} u  \le C \left(\mathcal R^{-\frac{N_\gamma}{\alpha}}\|u\|_{L^\alpha(B_\gamma(0,2\mathcal R))} +\mathcal R^2 b \right),$$
concluding the proof of the theorem.
\end{proof}

\end{document}